\pdfoutput=1

\documentclass{siamart171218}

\usepackage{amsfonts}
\usepackage{graphicx}
\usepackage{epstopdf}
\usepackage{algorithmic}
\ifpdf
  \DeclareGraphicsExtensions{.eps,.pdf,.png,.jpg}
\else
  \DeclareGraphicsExtensions{.eps}
\fi
\usepackage{amssymb}
\usepackage{subfigure}

\newcommand{\beq} {\begin{equation}}
\newcommand{\eeq} {\end{equation}}
\newcommand{\bdm} {\begin{displaymath}}
\newcommand{\edm} {\end{displaymath}}
\newcommand{\bit}{\begin{itemize}}
\newcommand{\eit}{\end{itemize}}
\newcommand{\bde}{\begin{description}}
\newcommand{\ede}{\end{description}}
\newcommand{\bce}{\begin{center}}
\newcommand{\ece}{\end{center}}
\newcommand{\ben} {\begin{enumerate}}
\newcommand{\een} {\end{enumerate}}
\newcommand{\bea} {\begin{eqnarray}}
\newcommand{\eea} {\end{eqnarray}}
\newcommand{\barr} {\begin{array}}
\newcommand{\earr} {\end{array}}
\newcommand{\bean} {\begin{eqnarray*}}
\newcommand{\eean} {\end{eqnarray*}}
\newcommand{\edoc} {\end{document}}

\newcommand{\Grad} {\ensuremath{\nabla}}

\newcommand{\Div} {\ensuremath{\nabla\cdot}}

\newcommand{\Curl} {\ensuremath{\nabla\times}}

\newcommand{\eval}[2][\right]{\relax
  \ifx#1\right\relax \left.\fi#2#1\rvert}

\newcommand{\e}{e}

\providecommand{\norm}[1]{\left\lVert#1\right\rVert}

\newcommand{\half} {\ensuremath{\frac{1}{2}}}

\newcommand{\mc}[1]{\mathcal{#1}}

\newcommand{\mb}[1]{\mathbf{#1}}

\newcommand{\nor}[1]{\left\| #1 \right\|}
\newcommand{\nort}[1]{{\left\vert\kern-0.25ex\left\vert\kern-0.25ex\left\vert #1 
    \right\vert\kern-0.25ex\right\vert\kern-0.25ex\right\vert}}
\newcommand{\snor}[1]{\left| #1 \right|}
\newcommand{\LRp}[1]{\left( #1 \right)}
\newcommand{\LRs}[1]{\left[ #1 \right]}
\newcommand{\LRa}[1]{\left< #1 \right>}
\newcommand{\LRc}[1]{\left\{ #1 \right\}}
\newcommand{\pp}[2]{\frac{\partial #1}{\partial #2}}

\newcommand{\bs}[1]{\boldsymbol{#1}}

\newcommand{\jump}[1] {\ensuremath{[\![{#1}]\!]}}

\newcounter{tablecell}

\newcommand{\eqnlab}[1]{\label{eq:#1}}
\newcommand{\theolab}[1]{\label{theo:#1}}

\newcommand{\lemlab}[1]{\label{lem:#1}}

\newcommand{\lemref}[1]{\ref{lem:#1}}
\newcommand{\eqnref}[1]{\eqref{eq:#1}}

\newcommand{\seclab}[1]{\label{sect:#1}}
\newcommand{\secref}[1]{\ref{sect:#1}}

\renewcommand{\u}{u}
\renewcommand{\v}{v}

\newcommand{\ed}{e}
\newcommand{\w}{w}
\newcommand{\wb}{\bs{\w}}
\newcommand{\J}{H}
\newcommand{\Jb}{\bs{\J}}
\newcommand{\q}{q}
\renewcommand{\r}{r}
\newcommand{\rh}{\hat{\r}}

\newcommand{\ub}{\bs{\u}}

\newcommand{\vb}{\bs{\v}}

\renewcommand{\d}{{d}}
\newcommand{\db}{\bs{\d}}
\newcommand{\R}{{\mathbb{R}}}

\newcommand{\F}{\bs{F}}
\newcommand{\f}{f}
\newcommand{\g}{g}
\newcommand{\fb}{\bs{\f}}
\newcommand{\gb}{\bs{\g}}

\newcommand{\pOmega}{{\partial \Omega}}
\newcommand{\K}{K}
\newcommand{\Kp}{K^+}
\newcommand{\Km}{K^-}
\newcommand{\pK}{{\partial \K}}

\newcommand{\Gh}{{\mc{E}_h}}
\newcommand{\Gho}{{\mc{E}_h^o}}
\newcommand{\Ghb}{{\mc{E}_h^{\partial}}}
\newcommand{\Poly}{{\mc{P}}}

\newcommand{\p}{{p}}

\newcommand{\mub}{\bs{\mu}}
\newcommand{\lambdab}{\bs{\lambda}}

\renewcommand{\L}{L}
\newcommand{\Lb}{\bs{\L}}

\newcommand{\n}{\bs{n}}
\newcommand{\nm}{\n^-}
\newcommand{\np}{\n^+}

\newcommand{\ubh}{{\hat{\ub}}}

\newcommand{\Fh}{{\hat{\F}}}

\renewcommand{\c}{{c}}
\newcommand{\cb}{\bs{\c}}

\newcommand{\veps}{{\varepsilon}}

\renewcommand{\H}{J}
\newcommand{\Hb}{\bs{\H}}

\newcommand{\hb}{\bs{h}}

\newcommand{\G}{G}
\newcommand{\Gb}{\bs{\G}}

\newcommand{\s}{s}

\renewcommand{\Pr}{\mathbb{P}}
\newcommand{\thetab}{\bs{\theta}}
\newcommand{\sigmab}{\bs{\sigma}}
\newcommand{\Omegah}{{\Omega_h}}

\newcommand{\ra}{\rightarrow}

\newcommand{\pd}{\partial}
\newcommand{\curl}{\Curl}

\renewcommand{\div}{\Div}

\newcommand{\bb}{\boldsymbol{b}}
\newcommand{\bbt}{\bb^t}
\newcommand{\bbh}{\hat{\bb}}
\newcommand{\bbht}{\bbh^t}

\newcommand{\tr}{\operatorname{tr}}
\newcommand{\lap}{\Delta}

\newcommand{\vel}{\bs{\vartheta}}

\newcommand{\tred}[1]{\textcolor{red}{#1}}
\newcommand{\tblu}[1]{\textcolor{blue}{#1}}
\newcommand{\tbrown}[1]{\textcolor{brown}{#1}}
\newcommand{\tcyan}[1]{\textcolor{cyan}{#1}}
\newcommand{\tgreen}[1]{\textcolor{green}{#1}}
\newcommand{\tpurple}[1]{\textcolor{purple}{#1}}

\newcommand{\eL}{\veps_{\Lb}}

\newcommand{\eb}{\veps_{\bb}}

\newcommand{\ebt}{\veps_{\bb^t}}
\newcommand{\eu}{\veps_{\ub}}

\newcommand{\euh}{\veps_{\ubh}}

\newcommand{\eH}{\veps_{\Hb}}

\newcommand{\ebht}{\veps_{\bbh^t}}

\newcommand{\ep}{\veps_\p}

\newcommand{\er}{\veps_\r}
\newcommand{\erh}{\veps_{\hat{r}}}
\newcommand{\poly}[1]{\mc{P}_{#1}}
\newcommand{\polyb}[1]{\bs{\mc{P}}_{#1}}

\newcommand{\polyt}[1]{\tilde{\bs{\mc{P}}}_{#1}}
\newcommand{\delu}{\delta_{\ub}}
\newcommand{\delw}{\delta_{\wb}}
\newcommand{\tb}{\bs{t}}
\newcommand{\Ir}{\mathbb{I}}

\newcommand{\algn}[1]{\begin{align} #1 \end{align}}
\newcommand{\algns}[1]{\begin{align*} #1 \end{align*}}

\newcommand{\Ha}{\ensuremath{\textup{Ha}}}
\newcommand{\Rey}{\ensuremath{\textup{Re}}}
\newcommand{\Rm}{\ensuremath{\textup{Rm}}}

\newcommand{\elh}{\veps_{\Lb}^h}
\newcommand{\eli}{\veps_{\bs{L}}^I}
\newcommand{\eeuh}{\veps_{\ub}^h}
\newcommand{\eeui}{\veps_{\ub}^I}
\newcommand{\euhh}{\veps_{\hat{\ub}}^h}
\newcommand{\euhi}{\veps_{\hat{\ub}}^I}
\newcommand{\eph}{\veps_p^h}
\newcommand{\epi}{\veps_p^I}
\newcommand{\ehh}{\veps_{\Hb}^h}    
\newcommand{\ehi}{\veps_{\Hb}^I}
\newcommand{\eebh}{\veps_{\bb}^h}
\newcommand{\eebi}{\veps_{\bb}^I}
\newcommand{\ebhh}{\veps_{\bbh}^h}

\newcommand{\eebth}{\veps_{\bb^t}^h}
\newcommand{\eebti}{\veps_{\bb^t}^I}
\newcommand{\ebthh}{\veps_{\hat{\bb}^t}^h}
\newcommand{\ebthi}{\veps_{\hat{\bb}^t}^I}
\newcommand{\eerh}{\veps_\r^h}

\newcommand{\erhh}{\veps_{\rh}^h}
\newcommand{\erhi}{\veps_{\rh}^I}

\newcommand{\elis}{\veps_{\Lb^*}^I}

\newcommand{\eeuis}{\veps_{\ub^*}^I}

\newcommand{\euhis}{\veps_{\ubh^*}^I}

\newcommand{\epis}{\veps_{p^*}^I}

\newcommand{\ehis}{\veps_{\Hb^*}^I}

\newcommand{\eebis}{\veps_{\bb^*}^I}

\newcommand{\ebhis}{\veps_{\bbh^*}^I}

\newcommand{\ebthis}{\veps_{(\bbh^*)^t}^I}

\newcommand{\eeris}{\veps_{\r^*}^I}

\newcommand{\Pis}{\Pi^*}
\newcommand{\Pit}{\tilde{\Pi}}
\newcommand{\mltln}[1]{\begin{multline} #1 \end{multline}}
\newcommand{\mltlns}[1]{\begin{multline*} #1 \end{multline*}}

\newcommand{\Mb}{\bs{M}}
\newcommand{\Lambdabt}{\bs{\Lambda}^t}

\newcommand{\LbH}{\Lb_h}
\newcommand{\ubH}{\ub_h}
\newcommand{\pH}{p_h}
\newcommand{\HbH}{\Hb_h}
\newcommand{\bbH}{\bb_h}
\newcommand{\rH}{r_h}
\newcommand{\ubhH}{\ubh_h}
\newcommand{\bbhtH}{\bbht_h}
\newcommand{\rhH}{\rh_h}
\newcommand{\GbM}{\mb{G}_h}
\newcommand{\VbM}{\mb{V}_h}
\newcommand{\QbM}{\mb{Q}_h}
\newcommand{\JbM}{\mb{H}_h}
\newcommand{\CbM}{\mb{C}_h}
\newcommand{\SbM}{\mb{S}_h}
\newcommand{\MubM}{\mb{M}_h}
\newcommand{\LambM}{\mb{\Lambda}^t_h}
\newcommand{\GambM}{\mb{\Gamma}_h}
\newcommand{\XbM}{\mb{X}_h}

\newcommand{\FbhH}{\Fh_h}
\newcommand{\pOmegah}{{\pOmega_h}}
\newcommand{\Linfty}{{L^\infty}}
\newcommand{\Woneinfty}{{W^{1,\infty}}}

\graphicspath{{./figures/}}

\newsiamremark{remark}{Remark}
\newsiamremark{hypothesis}{Hypothesis}
\crefname{hypothesis}{Hypothesis}{Hypotheses}
\newsiamthm{claim}{Claim}

\headers{An HDG Method for MHD}{J. Lee, S. Shannon, T. Bui-Thanh, and J. N. Shadid}

\title{Analysis of an HDG method for linearized incompressible resistive MHD equations
  \thanks{
    {\bf Funding: } 
      This work was
      partially supported by DOE grants DE-SC0010518 and
      DE-SC0011118, NSF Grant DMS-1620352,
      and by DOE NNSA ASC Algorithms effort,
      the DOE Office of Science AMR program at Sandia National Laboratory under contract DE-AC04-94AL85000.
    We are grateful for the support.}
}

\author{
  Jeonghun J. Lee\footnotemark[2], 
  Stephen Shannon\footnotemark[2], 
  Tan Bui-Thanh\footnotemark[2] \footnotemark[3], \\
  \and John N. Shadid\footnotemark[4] \footnotemark[5]
}

\usepackage{amsopn}

\begin{document}

\maketitle

\renewcommand{\thefootnote}{\fnsymbol{footnote}}
\footnotetext[2]{Institute for Computational Engineering Sciences (ICES), The University of Texas at Austin, 
  Austin, TX 78712 
  (\email{jeonghun@ices.utexas.edu}, \email{shannon@ices.utexas.edu}, \email{tanbui@ices.utexas.edu}).
}
\footnotetext[3]{Department of Aerospace Engineering and Engineering Mechanics, The University of Texas at Austin, 
  Austin, TX 78712.
}
\footnotetext[4]{Computational Mathematics Department, Sandia National Laboratories, P.O. Box 5800, MS 1321, 
  Albuquerque, NM 87185 (\email{jnshadi@sandia.gov}).
}
\footnotetext[5]{Department of Mathematics and Statistics, University of New Mexico, Albuquerque, NM 87131.}
\renewcommand{\thefootnote}{\arabic{footnote}}

\begin{abstract}
  We present a hybridized discontinuous Galerkin (HDG) method for
  stationary linearized incompressible magnetohydrodynamics (MHD)
  equations. At the heart of the paper is the introduction of an HDG
  flux of the dual saddle-point form of the MHD equations that
  facilitates the hybridization of discontinuous Galerkin (DG) method.
  We carry out the {\em a priori} error estimates for the proposed HDG
  method on simplicial meshes in both two- and three-dimensions.
  The analysis provides optimal convergence for the fluid velocity and the magnetic variables,
  and quasi-optimal convergence for the remaining  quantities. Numerical examples are presented to verify the theoretical findings.
\end{abstract}



\section{Introduction}
An important base-level representation for continuum approximation of
the dynamics of charged fluids in the presence of electromagnetic
fields is the resistive magnetohydrodynamics (MHD) model.
MHD models describe important physical phenomena in astrophysical systems
(e.g. solar flares, and planetary magnetic field generation) and in critical
science and technological applications (e.g., magnetically confined fusion energy devices)
\cite{GoedbloedPoedts2004}, for example.  The single fluid resistive MHD model involes the
partial differential equations (PDEs) describing
conservation of mass, momentum, and energy, coupled with the
low-frequency Maxwell's equations.  This multiphysics PDE system is
highly nonlinear and characterized by multiple
interacting physical phenomena spanning a wide range of length-
and time-scales.  These characteristics make the task of developing scalable, robust,
accurate, and efficient computational methods 
extremely challenging.

The most common computational solution
strategies for MHD have been the use of explicit
and partially implicit time integration methods. Notably are implicit-explicit
\cite{aydemir-jcp-85-imhd,m3d,jardin-pop-05-mhd},
semi-implicit \cite{schnack-barnes-si,harned-mikic-si,nimrod2},
and operator-splitting \cite{hujeirat-mnras-98-irmhd,AlegraALE08}
techniques that
include
some use of implicit operators in the formulation.
The implicitness of these approaches is used to enhance efficiency by removing
stringent explicit time-scale constraints in the problem, either from diffusion or
from fast-wave phenomena
\cite{chacon-pop-08-3dmhd,KnollWave05}.

In addition to the challenges associated with designing robust and efficient time integrators, there are a number of spatial discretization issues 
including the dual saddle-point structure of the velocity-pressure $\LRp{\ub,p}$ 
and the enforcement of the solenoidal involution/constraint  on the magnetic induction $(\nabla \cdot \bb = 0)$.
This adds considerable complexity to the numerical approximation of 
resistive MHD system.
In the context of finite volume  and  finite element methods, there are 
four popular approaches to deal with these difficulties: 1) physics-compatible
discretizations that directly enforce key mathematical properties of the
continuous problem (see e.g. \cite{nedelec:1980, misha97, BR:2001}); 2)
methods that transform to potential-based formulations to
eliminate one or both saddle-point sub-systems
\cite{luisMHD,m3d,Lankalapalli07,shadid-jcp-10_mhd}; 3)
exact and weighted-exact penalty formulations
\cite{Gunzburgeretal91,Gerbeau00,CostabelDauge00,CostabelDauge02}; and 4)
and stabilization methods that 
regularize
the dual saddle-point structure \cite{Salah_MHD1999,Codina_MHD2011,Shadidetal2016_3DVMSMHD}.

In this paper we propose a {\em hybridized discontinuous Galerkin} (HDG)
formulation for a linearized version of the resistive MHD system.
The hybridization technique and post-processing have been proposed to reduce computational costs 
of saddle-point problems and to improve the accuracy of numerical solutions \cite{AB85}. 
HDG methods were developed by Cockburn, coauthors, and others
to mitigate the computational costs of classical discontinuous
Galerkin (DG) methods. They have  been proposed for
various types of PDEs including, but not limited to, Poisson-type equations
\cite{CockburnGopalakrishnanLazarov09, CockburnGopalakrishnanSayas10,
  NguyenPeraireCockburn09a,
EggerSchoberl10}, the Stokes equation
\cite{CockburnGopalakrishnan09, NguyenPeraireCockburn10},
the Oseen equations \cite{CesmeliogluCockburnNguyenEtAl13},
and the incompressible Navier-Stokes equations \cite{NguyenPeraireCockburn11}.

In HDG discretizations, the coupled unknowns are single-valued traces
introduced on the mesh skeleton, i.e., the faces, and for high order implicit systems the resulting
matrix is substantially smaller and sparser compared to standard DG
approaches. Once they are solved for, the volume DG unknowns can be
recovered in an element-by-element fashion, completely independent of
one another. Therefore HDG methods have an intrinsic structure for
parallel computing which is essential for large scale applications.
Nevertheless, devising an HDG method for coupled PDE systems is challenging
because construction of a consistent and robust HDG flux is nontrivial.
We adopt the upwind HDG framework proposed in
\cite{Bui-Thanh15, bui2016construction, bui2015rankine}
since it provides  a systematic construction of HDG methods for a large class of PDEs.

Our work starts with section \secref{notations} where notations and
conventions are introduced to enable the construction of HDG method
in section \secref{HDG}.  Specifically, the proposed HDG method is
introduced directly on the dual saddle-point structure of the MHD
system and its well-posedness is analyzed using an energy
approach. This is followed by the {\em a priori error estimation} in
section \secref{error_analysis} where we combine an energy analysis, specially designed projections, and a duality argument to provide convergence rates for all variables. 
{\em Our development can serve as a standalone high-order solver
  for linearized MHD equations, or can be used as the fast-time scale
  solver in an implicit/explicit time integration method, and/or the
solver for a sub-step in a fixed-point nonlinear solver}. Various
numerical results will be presented in section \secref{numerics} to verify our theoretical
findings. Section \secref{conclusions} concludes the paper with future work. This is followed by four appendices in which we detail the definition and analysis of projection operators, state some auxiliary results, discuss the well-posedness of the adjoint equation, and present a postprocessing procedure to
enforce the solenoidal constraints.

\section{Notation}

\seclab{notations}
In this section we introduce common notations and conventions to be
used in the rest of the paper. Let $\Omega \subset \R^\d$, $d=2,3$,
be a bounded domain {such that it is simply-connected and its boundary $\pOmega$ is
a Lipschitz manifold with only one component.} Suppose that we
have a triangulation of $\Omega$, i.e., a partition of $\Omega$ into a finite
number of nonoverlapping $d$-dimensional simplices. We assume that the triangulation
is shape-regular, i.e., 
for all $d$-dimensional simplices in the triangulation, 
the ratio of the diameter of the simplex and the radius of an inscribed $d$-dimensional ball is uniformly bounded.
We will use $\Omegah$ and $\Gh$ to denote
the sets of $d$- and ($d-1$)-dimensional simplices of the triangulation,
and call $\Gh$ the mesh skeleton of the triangulation. The boundary and interior mesh skeletons
are defined by $\Ghb := \{ e \in \Gh \,:\, e \subset \pd \Omega\}$ and
$\Gho := \Gh \setminus \Ghb$. We also define $\pOmegah := \LRc{\pK:\K \in \Omegah}$.
The mesh size of triangulations is $h := \max_{\K \in \Omegah} \ensuremath{\text{diam}}(\K)$.

We use $\LRp{\cdot,\cdot}_D$ (respectively $\LRa{\cdot,\cdot}_D$)
to denote the $L^2$-inner product on $D$ if $D$ is a $d$- (respectively $(d-1)$-)
dimensional domain. 
The standard notation $W^{s,p}(D)$, $s \ge 0$, $1 \le p \le \infty$,
is used for the Sobolev space on $D$ based on $L^p$-norm with
differentiability $s$ (see, e.g., \cite{Evans-book}) and $\nor{\cdot}_{W^{s,p}(D)}$
denotes the associated norm. In particular, if $p = 2$, we use $H^s(D) := W^{s,p}(D)$
and $\nor{\cdot}_{s,D}$.
$W^{s,p}(\Omegah)$ denotes the space of functions whose restrictions on $\K$
reside in $W^{s,p}(\K)$ for each $\K \in \Omegah$ and its norm is
$\nor{u}_{W^{s,p}(\Omegah)}^p := \sum_{\K \in \Omegah} \nor{u|_{\K}}_{W^{s,p}(\K)}^p$
if $1 \le p < \infty$ and $\nor{u}_{W^{s,\infty}(\Omegah)} := \max_{\K \in \Omegah} \nor{u|_{\K}}_{W^{s,\infty}(\K)}$.
For simplicity, we use $\LRp{\cdot, \cdot}$, $\LRa{\cdot,\cdot}$,
$\nor{\cdot}_s$, $\nor{\cdot}_{\pOmegah}$, and $\nor{\cdot}_{W^{s,\infty}}$ for
$\LRp{\cdot, \cdot}_{\Omega}$, $\LRa{\cdot,\cdot}_{\pOmegah}$,
$\nor{\cdot}_{s, \Omega}$, $\nor{\cdot}_{0, \pOmegah}$, and $\nor{\cdot}_{W^{s,\infty}(\Omegah)}$, respectively.
We define $\| u, v \| := \| u \| + \| v \|$.
Furthermore, we denote by $A \lesssim B$ the inequality $A \le \lambda B$ with a constant $\lambda>0$
independent of the mesh size, and by $A \sim B$ the combination of  $A \lesssim B$ and $B \lesssim A$.

For vector- or matrix-valued functions these notations are naturally extended with a
component-wise inner product.
We define similar spaces (respectively inner products and norms) on a single element and a single skeleton face/edge
by replacing $\Omegah$ with $\K$ and $\Gh$ with $\e$.
We define the gradient of a vector, the divergence of a matrix, and the outer product symbol $\otimes$ as:
\[
  \LRp{\nabla \ub}_{ij} = \pp{u_i}{x_j}, \quad
  \LRp{\div \Lb}_i = \div \Lb\LRp{i,:} = \sum_{j=1}^3\pp{\bs{L}_{ij}}{x_j}, \quad
  \LRp{\bs{a} \otimes \bs{b}}_{ij} = a_i b_j = \LRp{\bs{a}\bs{b}^T}_{ij}.
\]
In this paper $\n$ denotes a unit outward normal vector field on faces/edges.
If $\pd \Km \cap \pd \Kp \in \Gh$ for two
distinct simplices $\Km, \Kp$, then $\nm$ and $\np$ denote
the outward unit normal vector fields on $\pd \Km$ and $\pd \Kp$, respectively, and
$\nm = - \np$ on $\pd \Km \cap \pd \Kp$. We simply use $\n$ to denote either $\nm$ or $\np$ in an
expression that is valid for both cases, and this convention is also
used for other quantities (restricted) on a face/edge $\e \in \Gh$. For a scalar quantity
$u$ which is double-valued on $e := \pd \Km \cap \pd \Kp$, the jump term on $e$ is defined by
$\jump{u \n}|_e = u^+ \np + u^- \nm$ where $u^+$ and $u^-$ are the traces of $u$ from $\Kp$-
and $\Km$-sides, respectively. For double-valued vector quantity $\ub$ and matrix quantity $\Lb$,
jump terms are $\jump{\ub \cdot \n}|_e = \ub^+ \cdot \np + \ub^- \cdot \nm$ and
$\jump{\Lb \n}|_e = \Lb^+ \np + \Lb^- \nm$ where $\Lb \n$ denotes the matrix-vector product.

We define $\Poly_k\LRp{\K}$ as
the space of polynomials of degree at most $k$ on a domain $\K$, and
\algns{
  \Poly_k\LRp{\Omegah} = \LRc{ u \in L^2(\Omega) \;:\; u|_{\K} \in \Poly_k\LRp{\K} \; \forall \K \in \Omegah } .
}
The space of polynomials on the mesh skeleton $\Poly_k\LRp{\Gh}$ is similarly defined,
and their extensions to vector- or matrix-valued polynomials $\LRs{\Poly_k(\Omegah)}^d$,
$\LRs{\Poly_k(\Omegah)}^{d\times d}$, $\LRs{\Poly_k(\Gh)}^d$, etc, are straightforward.

\section{HDG Formulation}
\seclab{HDG}
We consider the following nondimensional linearized incompressible MHD system \cite{HoustonSchoetzauWei09}
\begin{subequations}
  \eqnlab{mhdlin}
  \begin{align}
    \label{eq:mhd1lin_1}
    - \frac{1}{\Rey} \lap \ub + \nabla p + (\wb \cdot \nabla) \bs{u} + \kappa \db \times (\nabla \times \bb) &= \gb, \\
    \label{eq:mhd3lin_1}
    \nabla \cdot \ub &= 0, \\
    \label{eq:mhd2lin_1}
    \frac{\kappa}{\Rm} \nabla \times (\nabla \times \bb) + \nabla r - \kappa \nabla \times (\ub \times \db) &= \fb, \\
    \label{eq:mhd4lin_1}
    \nabla \cdot \bb &= 0
  \end{align}
\end{subequations}
where $\ub$ is velocity of the fluid (plasma or liquid metal), $\bb$ the magnetic field, $p$ the fluid pressure, and $r$ a scalar potential.
The following are constant parameters: a fluid Reynolds number $\Rey > 0$, a magnetic Reynolds number $\Rm > 0$, 
and a coupling parameter $\kappa = \Ha^2 / (\Rey\Rm)$, with the Hartmann number $\Ha > 0$. 
Here, $\db \in \LRs{W^{1,\infty}\LRp{\Omega}}^d$ is a prescribed magnetic field 
and $\wb \in \LRs{W^{1,\infty}\LRp{\Omegah}}^d \cap H(div,\Omega)$ with $\nabla \cdot \wb = 0$ is a prescribed velocity.

By introducing auxiliary variables $\Lb$ and $\Hb$, we cast \eqnref{mhdlin} into a first order hyperbolic system:
\begin{subequations}
  \eqnlab{mhdlin-first}
  \begin{align}
    \label{eq:mhd1lin1n}
    \Rey \Lb - \nabla \ub &= \bs{0}, \\
    \label{eq:mhd2lin1n}
    - \nabla \cdot \Lb + \nabla p + (\wb \cdot \nabla) \ub + \kappa \db \times (\nabla \times \bb) &= \gb, \\
    \label{eq:mhd3lin1n}
    \nabla \cdot \ub &= 0, \\
    \label{eq:mhd4lin1n}
    \frac{\Rm}{\kappa} \Hb - \nabla \times \bb &= \bs{0}, \\
    \label{eq:mhd5lin1n}
    \nabla \times \Hb + \nabla r  - \kappa \nabla \times (\ub \times \db) &= \fb, \\
    \label{eq:mhd6lin1n}
    \nabla \cdot \bb &= 0.
  \end{align}
\end{subequations}
In this paper, we consider the following (Dirichlet) boundary conditions for the MHD system \eqnref{mhdlin-first}
\begin{subequations}
  \eqnlab{MHD_BCs_Dirichlet}
  \begin{align}
    \ub &= \bs{u}_D                 & \textrm{ on } \partial \Omega, \\
    \bbt:=-\n \times (\n \times \bb) &= \hb_D       & \textrm{ on } \partial \Omega, \\
    r &= 0                              & \textrm{ on } \partial \Omega,
  \end{align}
\end{subequations}
where we have defined the tangent component of a vector $\bs{a}$ as $\bs{a}^t := -\n\times(\n\times\bs{a})$.
Additionally, we require the compatibility condition for $\ub_D$:
\begin{equation}
  \label{compatibility}
  \LRa{\ub_D \cdot \n, 1}_{\pOmega} = 0 .
\end{equation}
For the uniqueness of the pressure, $p$, we require that the pressure has zero mean, i.e.,
\begin{equation}
  \label{eq:global_3_1}
  (p,1)_\Omega = 0.
\end{equation}

Following the upwind HDG framework in \cite{Bui-Thanh15} we define the HDG flux as
\begin{equation}
  \label{eq:HDGflux}
  \LRs{
    \begin{array}{c}
      \Fh^1 \cdot \n \\
      \Fh^2 \cdot \n \\
      \Fh^3 \cdot \n \\
      \Fh^4 \cdot \n \\
      \Fh^5 \cdot \n \\
      \Fh^6 \cdot \n
    \end{array}
  } =
  \LRs{
    \hspace{-0.4em}
    \begin{array}{c}
      -\ubh \otimes \n \\
      -\Lb\n + m\ub +\p\n + \half \kappa \db \times \LRp{\n \times \LRp{\bbt+\bbht}} + \alpha_1 \LRp{\ub - \ubh} \\
      \ubh \cdot \n \\
      -\n \times \bbht \\
      \n \times \Hb  + \rh\n -\half \kappa \n \times \LRp{\LRp{\ub +\ubh}\times \db} + \alpha_2\LRp{\bbt-\bbht} \\
      \bb \cdot \n + \alpha_3\LRp{\r - \rh}
    \end{array}
    \hspace{-0.4em}
  }
\end{equation}
where $\bbht$, $\ubh$, and $\rh$ are the single-valued trace quantities 
residing on the mesh skeleton $\Gh$. They will be new unknowns 
in the discretizations, which will be described later, to hybridize the DG method. Here, $m
:= \wb \cdot \n$, and $\alpha_1$, $\alpha_2$, and $\alpha_3$ are
constant parameters.  As will be shown later, the conditions $\alpha_1
> \half \nor{\wb}_\Linfty$, $\alpha_2 > 0$, and $\alpha_3 > 0$ are
sufficient for the well-posedness of our HDG formulation.
Note that all 6 components of the HDG flux, $\Fh$, for simplicity are denoted
in the same fashion (by a bold italic symbol).
It is, however, clear from \eqnref{mhdlin-first} that  $\Fh^1$ is a third order tensor,
$\Fh^2$ is a second order tensor, $\Fh^3$ is a vector, etc, and that the
normal HDG flux components, $\Fh^i \cdot \n$, defined in \eqref{eq:HDGflux},
are tensors of one order lower.

For discretization we introduce the discontinuous
piecewise polynomial spaces
\begin{gather*}
  \GbM := \LRs{\Poly_k(\Omegah)}^{d \times d}, \qquad \VbM := \LRs{\Poly_k(\Omegah)}^d, \qquad \QbM := \Poly_k(\Omegah) , \\
  \JbM := \LRs{\Poly_k(\Omegah)}^{\tilde{d}}, \qquad \CbM := \LRs{\Poly_k(\Omegah)}^d, \qquad \SbM := \Poly_k(\Omegah), \qquad \MubM := \LRs{\Poly_k(\Gh)}^d , \\
  \LambM := \LRc{ \lambdab \in \LRs{\Poly_k(\Gh)}^d \;:\; \lambdab \cdot \n_e = 0      \; \forall  e \in \Gh}, \qquad \GambM := \LRs{\Poly_k(\Gh)}^d,
\end{gather*}
where $\tilde{d} = 3$ if $d = 3$, and $\tilde{d} = 1$ if $d = 2$.

Let us introduce two identities which are useful throughout the paper:
\begin{subequations}
  \eqnlab{eq:int-by-parts}
  \algn{
    \LRp{ \ub , \db \times \LRp{\Curl \bb} }_\K
    &= \LRp{ \bb, \Curl \LRp{\ub \times \db} }_\K + \LRa{\db \times \LRp{\n \times \bb}, \ub}_\pK, \\
    \LRs{ \db \times \LRp{\n \times \bb} } \cdot \ub
    &= - \LRs{\n \times \LRp{\ub \times \db}} \cdot \bb .
  }
\end{subequations}
These identities follow from integration by parts and vector product identities.

Next, we multiply \eqref{eq:mhd1lin1n} through \eqref{eq:mhd6lin1n} by
test functions $(\Gb,\vb,q,\Jb,\cb,s)$, integrate by parts all terms, and introduce the HDG
flux \eqref{eq:HDGflux} in the boundary terms.  This results in a
local discrete weak formulation, which we shall call the
\textit{local solver} of the HDG method, for the
MHD system \eqnref{mhdlin-first}:
\begin{subequations}
  \eqnlab{local}
  \begin{align}
    \label{eq:local_1_1}
    \Rey\LRp{\LbH, \Gb}_\K + \LRp{\ubH, \Div\Gb}_\K
    + \LRa{\FbhH^1 \cdot \n, \Gb}_\pK &= 0, \\
    \label{eq:local_2_1}
    \LRp{\LbH, \Grad\vb}_\K - \LRp{\pH, \Div \vb}_\K - \LRp{\ubH \otimes \wb , \Grad \vb}_\K \quad & \\
    + \kappa\LRp{\bbH, \Curl\LRp{\vb \times \db}}_\K
    + \LRa{\FbhH^2 \cdot \n, \vb}_\pK &= \LRp{\gb, \vb}_\K, \notag \\
    \label{eq:local_3_1}
    -\LRp{\ubH, \Grad \q}_\K
    + \LRa{\FbhH^3 \cdot \n, \q}_\pK &= 0, \\
    \label{eq:local_4_1}
    \frac{\Rm}{\kappa}\LRp{\HbH, \Jb}_\K - \LRp{\bbH, \Curl \Jb}_\K
    + \LRa{\FbhH^4 \cdot \n, \Jb}_\pK &= 0, \\
    \label{eq:local_5_1}
    \LRp{\HbH, \Curl \cb}_\K - \LRp{\rH, \Div \cb}_\K - \kappa\LRp{\ubH, \db \times \LRp{\Curl\cb}}_\K \quad & \\
    + \LRa{\FbhH^5 \cdot \n, \cb}_\pK &= \LRp{\fb,\cb}_\K, \notag \\
    \label{eq:local_6_1}
    -\LRp{\bbH, \Grad \s}_\K
    + \LRa{\FbhH^6 \cdot \n, \s}_\pK &= 0
  \end{align}
\end{subequations}
for all $(\Gb,\vb,q,\Jb,\cb,s) \in \GbM\LRp{K} \times \VbM\LRp{K} \times \QbM\LRp{K}
\times \JbM\LRp{K} \times \CbM\LRp{K} \times \SbM\LRp{K}$
and for all $\K \in \Omegah$,
where $\ubH$, $\LbH$, ..., are the discrete counterparts of $\ub$, $\Lb$, ...,
and $\FbhH^i$ is the discrete counterpart of $\Fh^i$ in \eqref{eq:HDGflux}
by replacing the unknowns $\ub$, $\Lb$, ..., with their discrete counterparts.

Since $\bbhtH$, $\ubhH$, and $\rhH$ are the new unknowns, we
need to equip extra equations to close the system \eqnref{local}. To
that end, we observe that an
element $\K$ communicates with its neighbors only through the trace unknowns. 
For the HDG method to be conservative, we (weakly) enforce the
continuity of the HDG flux \eqref{eq:HDGflux} across each interior
edge, i.e., $\LRa{\jump{\FbhH \cdot \n},\bs{\delta}}_e = 0, \forall \e \in \Gho$.
Since $\ubhH, \bbhtH$
and $\rhH$ are single-valued on $\Gh$, 
$\jump{\FbhH^1 \cdot \n} = \bs{0}$, $\jump{\FbhH^3 \cdot \n} = 0$, and $\jump{\FbhH^4 \cdot \n} = \bs{0}$.
The conservation constraints to be enforced reduce to
\begin{equation}
  \label{global}
  \LRa{\jump{\FbhH^2 \cdot \n},\mub}_\ed = 0, \qquad 
  \LRa{\jump{\FbhH^5 \cdot \n}, \lambdab^t}_\ed = 0, \qquad 
  \LRa{\jump{\FbhH^6 \cdot \n}, \gamma}_\ed = 0
\end{equation}
for all $(\mub,\lambdab^t,\gamma) \in \MubM\LRp{e} \times \LambM\LRp{e} \times \GambM\LRp{e}$, 
and for all $e$ in $\Gho$.
Finally, we enforce the Dirichlet boundary conditions through the trace unknowns:

\begin{align}
  \eqnlab{BCs}
  \LRa{\ubhH,\mub}_e = \LRa{\ub_D,\mub}_e, \qquad 
  \LRa{\bbhtH,\lambdab^t}_e = \LRa{\hb_D,\lambdab^t}_e, \qquad
  \LRa{\rhH,\gamma}_e = 0
\end{align}
for all $(\mub,\lambdab^t,\gamma) \in \Mb_h\LRp{e} \times \Lambdabt_h\LRp{e} \times \Gamma_h\LRp{e}$ 
for all $e$ in $\Ghb$.

In \eqnref{local}, \eqref{global}, and \eqnref{BCs}, we seek
$(\LbH,\ubH,\pH,\HbH,\bbH,\rH) 
\in \GbM \times \VbM \times \QbM \times \JbM \times \CbM \times \SbM$ and
$(\ubhH,\bbhtH,\rhH) \in \MubM \times \LambM \times \GambM$.
From this point forward, for simplicity in writing, we will not 
state explicitly that equations hold for all test functions, for all elements, or
for all edges.

We will refer to $\LbH,\ubH,\pH,\HbH,\bbH,$ and $\rH$ as the
\textit{local variables}, and to equation \eqnref{local} on each element as the \textit{local solver}. 
This reflects the fact that we can solve for local variables
element-by-element as function of $\ubhH, \bbhtH,$ and $\rhH$. On the
other hand, we will refer to $\ubhH, \bbhtH,$ and $\rhH$
as the \textit{global variables}, which are governed by equations
\eqref{global} and \eqnref{BCs} on the mesh skeleton. 
Finally, for the uniqueness of the discrete pressure $\pH$, 
we enforce the discrete counterpart of \eqref{eq:global_3_1}:
\begin{align}
  \label{eq:global_5_1}
  (\pH,1) &= 0.
\end{align}

\subsection{Well-posedness of the HDG formulation}
\seclab{wellposednessHDG}
In this subsection we discuss well-posedness of the system \eqnref{local}--\eqnref{BCs}.
\begin{theorem}
  The HDG system \eqnref{local}--\eqnref{global_5_1}
  is well-posed. 
\end{theorem}
\begin{proof}
  Since the problem is a system of linear equations
  with the same number of equations and unknowns,
  without loss of generality, we only need show
  that $\LRp{\gb,\fb, \ub_D, \bs{h}_D} = \bs{0}$ implies
  $(\LbH,\ubH,\pH,\HbH,\bbH, \rH,\ubhH,\bbhtH,\rhH) = \bs{0}$.
  To begin, we take $(\Gb,\vb,q,\Jb,\cb,s) = (\LbH,\ubH,\pH,\HbH,\bbH,\rH)$ and get 
  \begin{subequations}
    \begin{align*}
      \Rey\LRp{\LbH, \LbH}_\K + \LRp{\ubH, \Div\LbH}_\K
      + \LRa{-\ubhH \otimes \n, \LbH}_\pK &= 0, \\
      \LRp{\LbH, \Grad \ubH}_\K - \LRp{\pH, \Div \ubH}_\K - \LRp{\ubH \otimes \wb , \Grad \ubH}_\K
      + \kappa\LRp{\bbH, \Curl\LRp{\ubH \times \db}}_\K \quad & \\
      + \LRa{-\LbH\n + m\ubH +\pH\n + \half \kappa \db \times \LRp{\n \times \LRp{\bbt_h+\bbht_h}} + \alpha_1 \LRp{\ubH - \ubhH}, \ubH}_\pK &= 0, \\
      -\LRp{\ubH, \Grad \q}_\K
      + \LRa{\ubhH \cdot \n, \q}_\pK &= 0, \\
      \frac{\Rm}{\kappa}\LRp{\HbH, \HbH}_\K - \LRp{\bbH, \Curl \HbH}_\K
      + \LRa{-\n \times \bbht_h, \HbH}_\pK &= 0, \\
      \LRp{\HbH, \Curl \bbH}_\K - \LRp{\rH, \Div \bbH}_\K - \kappa\LRp{\ubH, \db \times \LRp{\Curl\bbH}}_\K \quad & \\
      + \LRa{\n \times \HbH  + \rhH\n -\half \kappa \n \times \LRp{\LRp{\ubH +\ubhH}\times \db} + \alpha_2\LRp{\bbt_h-\bbht_h}, \bbH}_\pK &= 0, \\ 
      -\LRp{\bbH, \Grad \rH}_\K
      + \LRa{\bbH \cdot \n + \alpha_3\LRp{\rH - \rhH}, \rH}_\pK &= 0 .
    \end{align*}
  \end{subequations}
  If we integrate by parts the first four terms of the second equation and the first term of the fifth equation, 
  sum the resulting equations, and sum over all elements, then we arrive at 
  \begin{align}
    \Rey\nor{\LbH}_0^2 + \frac{\Rm}{\kappa}\nor{\HbH}_0^2
    - \LRa{\ubhH \otimes \n, \LbH}
    + \LRa{ \frac{m}{2} \ubH, \ubH } \notag
    + \LRa{ \alpha_1 (\ubH-\ubhH), \ubH } \\
    \label{eq:sum_energy4_1}
    + \LRa{ \half \kappa \db \times \LRp{\n \times \bbhtH}, \ubH }
    + \LRa{\ubhH \cdot \n, \pH}
    - \LRa{ \n \times \bbhtH, \HbH }
    + \LRa{ \rhH \n, \bbH } \\
    + \LRa{ \alpha_2(\bbH^t-\bbhtH), \bbH^t } \notag
    - \LRa{ \half\kappa\n \times \LRp{\ubhH \times \db}, \bbH }
    + \LRa{\alpha_3 \LRp{\rH-\rhH}, \rH}
    &= 0.
  \end{align}
  Here, we have used 
  \algns{
    - \LRp{\ubH,\wb \cdot \nabla \ubH}_\K \allowbreak
    = - \half\LRp{\wb,\nabla(\ubH\cdot\ubH)}_\K \allowbreak
    &= \half\LRp{(\Div\wb)\ubH,\ubH}_\K \allowbreak - \LRa{\frac{m}{2}\ubH,\ubH}_\pK \\
    &= - \LRa{\frac{m}{2}\ubH,\ubH}_\pK 
  }
  which is obtained from $\Div \wb = 0$. 

  Next, we set $({\mub, \lambdab^t, \gamma}) = ({\ubhH,\bbhtH,\rhH})$,
  and sum \eqref{global} over all interior edges to obtain
  \begin{align}
    \LRa{-\LbH\n + m\ubH +\pH\n + \half\kappa\db \times \LRp{\n \times \bbH^t}
    + \alpha_1\LRp{\ubH - \ubhH},\ubhH}_{\pOmegah\setminus\pOmega} \notag & \\
    \label{eq:sum_energy5_1}
    + \LRa{\n \times \HbH -\half\kappa\n \times \LRp{\ubH \times \db}
    + \alpha_2\LRp{\bbH^t -\bbhtH}, \bbhtH}_{\pOmegah\setminus\pOmega} & \\
    \notag
    +\LRa{\bbH \cdot \n + \alpha_3\LRp{\rH - \rhH}, \rhH}_{\pOmegah\setminus\pOmega}
    &= 0
  \end{align}
  where we have used the continuity of $\db$ to eliminate 
  $\langle{\db \times ({\n \times \bbhtH}),\ubhH}\rangle_{\pOmegah\setminus\pOmega}$ and
  $\langle{\n \times \LRp{\ubhH \times \db},\bbhtH}\rangle_{\pOmegah\setminus\pOmega}$.

  Since $\ub_D = \bs{0}$ and $\hb_D = \bs{0}$ by assumption, we conclude from the boundary conditions \eqnref{BCs}
  that $\ubhH = \bs{0}$, $\bbhtH = \bs{0}$, and $\rhH = 0$  on $\pOmega$.
  The integrals in \eqref{eq:sum_energy5_1} can then be written over $\pOmegah$
  since the contribution on the domain boundary, $\pOmega$, is zero.
  Subtracting \eqref{eq:sum_energy5_1} from \eqref{eq:sum_energy4_1} we arrive at
  \begin{align}
    \label{eq:sum_energy6_2}
    \Rey\nor{\LbH}_0^2 + \frac{\Rm}{\kappa}\nor{\HbH}_0^2
    + \LRa{ \alpha_1 (\ubH-\ubhH), (\ubH-\ubhH) }
    + \LRa{ \frac{m}{2} \ubH, \ubH } \quad & \\
    \quad - \LRa{ m \ubH, \ubhH }  \notag
    + \alpha_2 \nor{\bbH^t-\bbhtH}_\pOmegah^2
    + \alpha_3 \nor{\rH-\rhH}_\pOmegah^2
    &= 0.
  \end{align}
  Finally, using the fact that $\wb \in H(div,\Omega)$ and $\ubhH = \bs{0}$ on $\pOmega$,
  we can freely add
  $0=\LRa{\frac{m}{2}\ubhH,\ubhH}$ to rewrite \eqref{eq:sum_energy6_2} as
  \begin{align}
    \label{eq:sum_energy6_1}
    \Rey\nor{\LbH}_0^2 + \frac{\Rm}{\kappa}\nor{\HbH}_0^2
    + \LRa{ \LRp{\alpha_1+\frac{m}{2}} (\ubH-\ubhH), (\ubH-\ubhH) } \quad & \\
    + \alpha_2 \nor{\bbH^t-\bbhtH}_\pOmegah^2
    + \alpha_3 \nor{\rH-\rhH}_\pOmegah^2
    &= 0. \notag
  \end{align}

  Choosing $\alpha_1 > \half \nor{\wb}_\Linfty$, $\alpha_2 > 0$, and $\alpha_3 > 0$, we can conclude that
  $\LbH = \bs{0}$ and $\HbH = \bs{0}$,
  that $\ubH = \ubhH$, $\bbH^t = \bbhtH$, and $\rH = \rhH$ on $\Gho$,
  and that $\ubH = \bs{0}$, $\bbH^t = \bs{0}$, and $\rH = 0$ on $\pOmega$.

  Now, we integrate \eqref{eq:local_1_1} by parts to obtain
  $\Grad\ubH = \bs{0}$ in $\K$,
  which implies that $\ubH$ must be elementwise constant.
  The fact that
  $\ubH = \ubhH$ on $\Gho$
  means $\ubH$ is also continuous on $\Gh$, and since
  $\ubH = \bs{0}$ on $\pOmega$
  we conclude that $\ubH = \bs{0}$, and therefore $\ubhH = \bs{0}$.

  Note that since $\bbH^t = \bbhtH$ on $\Gho$, $\bbH^t$ is
  continuous on $\Omega$.
  Furthermore, the third conservation constraint in \eqref{global} implies that $\bbH \cdot \n$
  is continuous on $\Omega$.  Integrating both \eqref{eq:local_4_1} and
  \eqref{eq:local_6_1} by parts, we can conclude that $\Curl \bbH =
  \bs{0}$ and $\Div \bbH = 0$ on $\Omega$.
  When $\bbH \in H(div,\Omega) \cap H(curl,\Omega)$ and $\bbH^t = \bs{0}$ on $\pOmega$,
  and when $\Omega$ is simply-connected with one component to the boundary,
  there is a constant $C>0$ such that $\norm{\bbH}_0
  \leq C (\norm{\Div \bbH}_0 + \norm{\Curl \bbH}_0)$ (see, e.g., \cite[Lemma
  3.4]{GiraultRaviart86}). This implies that $\bbH = \bs{0}$, and hence $\bbhtH =
  \bs{0}$.

  Considering the vanishing unknowns above, integrating by parts reduces 
  \eqref{eq:local_2_1} and \eqref{eq:local_5_1} to
  $(\Grad \pH, \vb)_\K = 0$ and $(\Grad \rH, \cb)_\K = 0$, respectively.
  Thus, $\pH$ and $\rH$ are elementwise constants.
  Since $\rH = \rhH$ on $\Gho$, then $\rH$ is continuous on $\Omega$,
  and since $\rH = 0$ on $\pOmega$,
  we can conclude that $\rH = 0$, and hence $\rhH = 0$.
  Finally, we use the first conservation constraint in \eqref{global} to
  conclude $\pH$ is continuous and hence constant on $\Omega$.
  Using the zero-average condition \eqref{eq:global_5_1} yields $\pH = 0$.
\end{proof}

\subsection{Well-posedness of the local solver}
The key design of the HDG method is that it allows us to separate the
computation of the volume (DG) unknowns $(\LbH,\ubH,\pH,\HbH,\bbH,\rH)$
and the trace unknowns $({\ubhH,\bbhtH,\rhH})$. In practice, we first
solve \eqnref{local} for local unknowns $(\LbH,\ubH,\pH,\HbH,\bbH,\rH)$ as a
function of $({\ubhH,\bbhtH,\rhH})$. These are then substituted into
the conservative algebraic equation \eqref{global} on the mesh
skeleton to solve for the unknown $({\ubhH,\bbhtH,\rhH})$. Finally,
the local unknowns 
$( \LbH, \ubH, \allowbreak \pH, \allowbreak \HbH, \bbH, \rH )$ 
are computed, as in the
first step, using $({\ubhH,\bbhtH,\rhH})$ from the second step. It is
therefore important to study the well-posedness of the local
solver. 

Similar to HDG methods for Stokes equation
\cite{NguyenPeraireCockburn10, CockburnNguyenPeraire10, Bui-Thanh15},
it turns out that the local solver is not well-posed unless extra
conditions are imposed on the pressure.
Two methods for achieving the well-posedness of the local solver for the Stokes equations
are proposed in \cite{NguyenPeraireCockburn10}.
One is a pseudotransient approach, and the other involves introducing the 
element average edge pressure as global unknowns.
These methods are both suitable for our setting here. Here, we present a new approach in which we introduce the elementwise
pressure integral as a global unknown and require their sum to vanish.
Toward this goal, we introduce the space of elementwise constants, 
$\XbM := \Poly_0(\Omegah)$.  
Next we augment \eqref{eq:local_3_1} to read
\begin{align}
  \label{eq:local_3_1aug}
  -\LRp{\ubH, \Grad \q}_\K 
  + \LRa{\ubhH \cdot \n, \q}_\pK
  + \LRp{\pH, \bar{\q}}_\K
  = |\K|^{-1}\LRp{\rho_h,\bar{\q}}_\K 
\end{align}
with $\rho_h \in \XbM$, $\bar{\q}|_\K := |\K|^{-1} \LRp{\q,1}_\K$ the average of $\q$ in $\K$, 
and $|\K|$ the volume of element $\K$.
Next we augment the global solver with
\begin{align}
  \label{compatibility_1aug}
  \LRa{\ubhH \cdot \n, \xi}_\pK + \sum_\K \rho_h|_\K &= 0 
\end{align}
for all $\xi$ in $\XbM$, and remove the constraint \eqref{eq:global_5_1}, 
which will be automatically satisfied by this construction.

To justify \eqref{eq:local_3_1aug} and \eqref{compatibility_1aug} we make the following observations.
First, summing \eqref{compatibility_1aug} over all elements and using the compatibility condition
on $\ub_D$, \eqref{compatibility}, we conclude 
\begin{equation}
  \label{sumrho}
  \sum_{\K \in \Omegah} \rho_h|_\K = 0 
\end{equation}
and
\begin{equation}
  \label{compatibility_1}
  \LRa{\ubhH \cdot \n, \xi}_\pK = 0 \qquad \forall \K \in \Omegah .
\end{equation}
Next, setting $\q = 1$ on $\K$ in \eqref{eq:local_3_1aug} and using \eqref{compatibility_1}, 
we can conclude that
\begin{equation}
  \label{eq:divFree_1}
  (\pH,1)_\K = \rho_h|_\K, 
\end{equation}
and therefore that \eqref{eq:local_3_1} holds for each $\K$.
Additionally, \eqref{eq:divFree_1} and \eqref{sumrho} imply that \eqref{eq:global_5_1} holds.
Finally, we note that we have added the same number of new unknowns $\rho_h$ as the number of  equations in
\eqref{compatibility_1aug}.

For this modified HDG scheme we claim well-posedness of the local solver. 
\begin{theorem}
  The local solver given by \eqnref{local} such that \eqref{eq:local_3_1} is replaced by \eqref{eq:local_3_1aug},
  is well-posed. In other words, given $(\ubhH, \bbhtH, \rhH, \gb, \fb, \rho_h)$, there exists 
  a unique solution \\
  $(\LbH,\ubH,\pH,\HbH,\bbH,\rH)$ to the system.
\end{theorem}
\begin{proof}
  We show that
  $(\ubhH, \bbhtH, \rhH, \gb, \fb, \rho_h) = \bs{0}$ implies
  $(\LbH,\ubH,\pH,\HbH,\bbH,\rH) = \bs{0}$. To begin,
  set $(\ubhH, \bbhtH, \rhH, \gb, \fb, \rho_h) = \bs{0}$.
  Then \eqref{eq:local_3_1aug} reduces to $-\LRp{\ubH,\Grad\q}_\K + \LRp{\pH,\bar{\q}}_\K = 0$,
  and taking $\q$ as constant gives $\LRp{\pH,\bar{\q}}_\K = 0$ and hence $-\LRp{\ubH,\Grad\q}_\K = 0$.

  Next, choose $(\Gb,\vb,\q,\Jb,\cb,\s) = (\LbH,\ubH,\pH,\HbH,\bbH,\rH)$,
  integrate by parts the first four terms in \eqref{eq:local_2_1} and
  the first term in \eqref{eq:local_5_1}, and sum the resulting equations in the local solver
  to conclude
  \begin{align}
    \label{eq:sum_energy2_1}
    &\Rey\nor{\LbH}_{0,\K}^2 
    + \LRa{ \LRp{\alpha_1 + \frac{m}{2}}\ubH, \ubH }_\pK \\
    &\quad + \frac{\Rm}{\kappa}\nor{\HbH}_{0,\K}^2  
    + \alpha_2 \nor{\bbH^t}_{0,\pK}^2
    + \alpha_3 \nor{\rH}_{0,\pK}^2
    = 0. \notag
  \end{align}
  Recalling we have set $\alpha_1 > \half \nor{\wb}_\Linfty$, $\alpha_2 > 0$, and $\alpha_3 > 0$, we can conclude that 
  \begin{align*}
    \LbH = \bs{0}, \quad   \HbH = \bs{0} \quad \text{ in } \K, \qquad 
    \ubH = \bs{0}, \quad   \bbH^t = \bs{0}, \quad \rH = 0 \quad \text{ on } \pK.
  \end{align*}
  Using an argument similar to that in Section \secref{wellposednessHDG}
  we can conclude that $\ubH = \bbH = \bs{0}$ in $\K$. From
  \eqref{eq:local_2_1} and \eqref{eq:local_5_1} we can conclude
  $\LRp{\Grad \pH, \vb}_\K = 0$ and $\LRp{\Grad \rH, \cb}_\K = 0$, respectively.
  Thus, $\pH$ and $\rH$ must be constant, and since $\rH=0$ on $\pK$, $\rH$ is
  identically zero in $\K$.  
  Now since $(\pH,\bar{q})_\K  = 0$,
  we have $\pH = 0$ in $\K$. 
\end{proof}
\begin{remark}
  Note that introducing $\rho_h$ and equations \eqref{eq:local_3_1aug} and \eqref{compatibility_1aug}
  does not alter the solution of the original HDG scheme. Indeed, if
  $(
  \LbH,\ubH,\pH,\HbH,\bbH,\rH,\allowbreak 
  \ubhH,\allowbreak\bbhtH,\allowbreak\rhH,\allowbreak\rho_h 
  )$
  is a solution of the modified scheme, it is also a
  solution of the original one because the modified scheme contains all the equations
  of the original one, except for \eqref{eq:local_3_1} and \eqref{eq:global_5_1}.
  But we have already shown that 
  \eqref{eq:local_3_1aug}, \eqref{compatibility_1aug}, and \eqref{compatibility} 
  imply  \eqref{eq:local_3_1},
  and that \eqref{eq:divFree_1} and \eqref{sumrho} imply \eqref{eq:global_5_1}. 
  Conversely, reversing these arguments, if
  $(
  \LbH,\ubH,\pH,\HbH,\bbH,\rH,\allowbreak 
  \ubhH,\allowbreak\bbhtH,\allowbreak\rhH 
  )$
  is a solution of the original scheme,
  we can define $\rho_h$ as in \eqref{eq:divFree_1} and add \eqref{eq:divFree_1} to \eqref{eq:local_3_1}
  to recover \eqref{eq:local_3_1aug}.
  Then, taking $\q$ as constant in \eqref{eq:local_3_1} implies \eqref{compatibility_1}.
  Also, \eqref{eq:divFree_1} and \eqref{eq:global_5_1} imply \eqref{sumrho}.
  Finally, adding \eqref{compatibility_1} and \eqref{sumrho} we recover \eqref{compatibility_1aug},
  showing that
  $({\LbH,\ubH,\pH,\HbH,\bbH,\rH, \ubhH, \bbhtH, \rhH, \rho_h})$ is a
  solution of the modified one. Since the original HDG scheme is well-posed, so is the modified one. 
\end{remark}

\section{Error analysis}
\seclab{error_analysis}
For an unknown $\sigma$ we use $\veps_\sigma$ to denote the error
between the exact solution $\sigma$ and its finite element
approximation $\sigma_h$. For example, $\veps_{\Lb} := \Lb - \LbH$ and
$\veps_{\ubh} := \ubh - \ubhH$, where $\ubh$ is the trace of the
exact solution $\ub$ on the mesh skeleton. We use $\Pi \sigma$ to
denote some interpolation (or projection) of the unknown $\sigma$ into
{its associated finite element space},
and decompose $\veps_\sigma$ into
$\veps_\sigma^I + \veps_\sigma^h$ where
\algn{ \label{eq:error-split}
\veps_\sigma^I := \sigma - \Pi \sigma, \qquad \text{and } \quad 
\veps_\sigma^h := \Pi \sigma - \sigma_h.  
}
Here the superscript $I$ of $\veps^I$ denotes the `I'nterpolation (in
fact projection) error, and the superscript $h$ of $e^h$ indicates the
difference between the interpolation of the exact solution 
and the
finite element approximation. We will see that $\Pi \sigma$ may not
depend only on $\sigma$. In particular, we define a collective
projection $\bs{\Pi} (\Lb, \ub, \p, \Hb, \bb, \r, \ubh,
\bbht, \rh)$ in Appendix \secref{projections} for our
HDG scheme. Each component of $\bs{\Pi}$ may depend on the
others. Nonetheless, for simplicity of presentation we use $\Pi
\Lb$ to denote the $\Lb$-component of $\bs{\Pi}$ for example. The
analysis and the properties of the proposed projection can be
referred to Appendix \secref{projections}. This projection
simplifies the error equation substantially as we now see.

\begin{lemma} \label{lemma1}
  Assume that the exact solution $(\Lb,\ub,\p,\Hb,\bb,\r)$ of \eqnref{mhdlin-first}-\eqnref{MHD_BCs_Dirichlet}
  is sufficiently regular.
  Then the exact solution
  satisfies \eqnref{local}--\eqnref{BCs}.
  That is, if we replace $(\LbH,\ubH,\pH,\HbH,\bbH,\rH,\ubhH,\bbhtH,\rhH)$ 
  with $(\Lb,\ub,\p,\Hb,\bb,\r, \ub, \bbt, \r)$ in \eqnref{local}--\eqnref{BCs},
  then {\eqnref{local}--\eqnref{BCs}} hold true for all
  $(\Gb,\vb,q,\Jb,\cb,s,\mub,\lambdab^t,\gamma) \in 
  \GbM\LRp{K} \times \VbM\LRp{K} \times \QbM\LRp{K} \times \JbM\LRp{K} \times \CbM\LRp{K} \times \SbM\LRp{K}
  \times \MubM(e) \times \LambM(e) \times \GambM(e)$.
\end{lemma}
\begin{proof}
  Multiply \eqnref{mhdlin-first} by
  $(\Gb,\vb,q,\Jb,\cb,s) \in \GbM\LRp{K} \times \VbM\LRp{K} \times \QbM\LRp{K} \times \JbM\LRp{K} \times \CbM\LRp{K} \times \SbM\LRp{K}$,
  integrate over $\Omega$, and integrate by parts.
  By the regularity assumptions on the solution, the solution components 
  $((\Lb \n)\otimes\n, \ub, \p, \Hb^t, \bb, \r)$ are single valued on $\mc{E}$.
  and the exact solution satisfies \eqnref{local}.
  With the additional fact that $(\wb \cdot \n)\n$ and $\db$ are also single-valued on $\mc{E}_h$,
  we have that the exact solution satisfies \eqref{global}.
  Finally, the boundary conditions \eqnref{MHD_BCs_Dirichlet} trivially imply that the exact solution satisfies \eqnref{BCs}.
\end{proof}

\begin{lemma}[Error equation]
  The discretization errors satisfy
  \algn{ 
    \label{eq:disc-energy}
    E_h^2:=&E_h^2 (\elh, \ehh, \eeuh - \euhh, \eebth - \ebthh, \eerh - \erhh) \\
    &\quad := 
    \Rey \nor{\eL^h}_0^2 + \frac{\Rm}{\kappa}\nor{\ehh}_0^2 
    + \LRa{ \LRp{\alpha_1+\frac{m}{2}} (\eu^h-\euh^h), (\eu^h-\euh^h) } \notag \\
    &\qquad  + \alpha_2 \nor{ \eebth-\ebht^h}_\pOmegah^2 
    + \alpha_3 \nor{\er^h-\erh^h}_\pOmegah^2  \notag \\
    &\quad =  - \Rey\LRp{\eL^I, \eL^h} 
    - \kappa\LRp{\eb^I, \Curl\LRp{\eu^h \times \db}}  
    + \kappa\LRp{\eu^I, \db \times \LRp{\Curl \eb^h}} \notag \\ 
    &\qquad - \LRa{ \n \times \ehi  - \half\kappa\n \times \LRp{\LRp{\eu^I+\euh^I}\times \db} + \alpha_2\ebt^I , \ebt^h-\ebht^h} . \notag
  }
\end{lemma}
\begin{proof}
  Using the fact that the numerical solution and exact solution both satisfy \eqnref{local} (see Lemma~\ref{lemma1}),
  the linearity of the operators lead to the following error equations:
  \begin{subequations}
    \eqnlab{local_error}
    \begin{align}
      \label{eq:local_1_1_error}
      \Rey\LRp{\eL, \Gb} + \LRp{\eu, \Div\Gb}
      - \LRa{\euh \otimes \n, \Gb} &= 0, \\
      \label{eq:local_2_1_error}
      \LRp{\eL, \Grad\vb} - \LRp{\ep, \Div \vb} - \LRp{\eu \otimes \wb , \Grad \vb}
      + \kappa\LRp{\eb, \Curl\LRp{\vb \times \db}} \quad & \\
      + \LRa{-\eL \n + m\eu +\ep\n + \half \kappa \db \times \LRp{\n \times \LRp{\ebt+\ebht}} + \alpha_1 \LRp{\eu - \euh}, \vb} &= 0, \notag \\
      \label{eq:local_3_1_error}
      -\LRp{\eu, \Grad \q}
      + \LRa{\euh \cdot \n, \q} &= 0, \\
      \label{eq:local_4_1_error}
      \frac{\Rm}{\kappa}\LRp{\eH, \Jb} - \LRp{\eb, \Curl \Jb}
      - \LRa{\n \times \ebht, \Jb} &= 0, \\
      \label{eq:local_5_1_error}
      \LRp{\eH, \Curl \cb} - \LRp{\er, \Div \cb} - \kappa\LRp{\eu, \db \times \LRp{\Curl\cb}} \quad & \\
      + \LRa{\n \times \eH  + \erh\n -\half \kappa \n \times \LRp{\LRp{\eu +\euh}\times \db} + \alpha_2\LRp{\ebt-\ebht}, \cb} &= 0, \notag \\
      \label{eq:local_6_1_error}
      -\LRp{\eb, \Grad \s}
      + \LRa{\eb \cdot \n + \alpha_3\LRp{\er - \erh}, \s} &= 0.
    \end{align}
  \end{subequations}
  Next, we split the error terms into their interpolation and approximation components as in \eqref{eq:error-split} 
  using the projections $\Pi$ defined in Appendix~\secref{projections}. Due to the cancellation properties of $\Pi$ we obtain reduced error equations:
  \begin{subequations}
    \eqnlab{local_error2}
    \begin{align}
      \label{eq:local_1_1_error2}
      \Rey\LRp{\elh, \Gb} + \LRp{\eeuh, \Div\Gb}
      - \LRa{\euhh \otimes \n, \Gb} 
      = 
      - \Rey\LRp{\eli, \Gb}, \\
      \label{eq:local_2_1_error2}
      \LRp{\elh, \Grad\vb} - \LRp{\eph, \Div \vb} - \LRp{\eeuh \otimes \wb , \Grad \vb}
      + \kappa\LRp{\eebh, \Curl\LRp{\vb \times \db}} \qquad \notag \\
      + \LRa{-\elh \n + m\eeuh +\eph\n + \half \kappa \db \times \LRp{\n \times \LRp{\eebth+\ebthh}}
      + \alpha_1 \LRp{\eeuh - \euhh}, \vb} \qquad \\
      =- \kappa\LRp{\eebi, \Curl\LRp{\vb \times \db}}, \notag \\
      \label{eq:local_3_1_error2}
      -\LRp{\eeuh, \Grad \q}
      + \LRa{\euhh \cdot \n, \q} 
      = 0, \\
      \label{eq:local_4_1_error2}
      \frac{\Rm}{\kappa}\LRp{\ehh, \Jb} - \LRp{\eebh, \Curl \Jb}
      - \LRa{\n \times \ebthh, \Jb} 
      = 0, \\
      \label{eq:local_5_1_error2}
      \LRp{\ehh, \Curl \cb} - \LRp{\eerh, \Div \cb} - \kappa\LRp{\eeuh, \db \times \LRp{\Curl\cb}}  \qquad \\
      + \LRa{\n \times \ehh  + \erhh\n -\half \kappa \n \times \LRp{\LRp{\eeuh +\euhh}\times \db} + \alpha_2\LRp{\eebth-\ebthh}, \cb} \qquad \notag \\
      = \kappa\LRp{\eeui, \db \times \LRp{\Curl\cb}} 
      - \LRa{\n \times \ehi -\half \kappa \n \times \LRp{\LRp{\eeui +\euhi}\times \db} + \alpha_2\eebti, \cb} 
      , \notag \\
      \label{eq:local_6_1_error2}
      -\LRp{\eebh, \Grad \s}
      + \LRa{\eebh \cdot \n + \alpha_3\LRp{\eerh - \erhh}, \s} 
      = 0.
    \end{align}
  \end{subequations}
  More details of the cancellation properties of $\Pi$ used in the above formula are:
  \algns{
    \eqref{uhat-proj}, \eqref{eq:Lu-projection2} 					&\ra \eqref{eq:local_1_1_error2} \\
    \eqref{p-proj}, \eqref{eq:Lu-projection1}, \eqref{eq:Lu-projection3} 		&\ra \eqref{eq:local_2_1_error2} \\
    \eqref{uhat-proj}, \eqref{eq:Lu-projection2} 					&\ra \eqref{eq:local_3_1_error2} \\
    \eqref{bhat-proj}, \eqref{J-proj}, \eqref{eq:br-proj1}			&\ra \eqref{eq:local_4_1_error2} \\
    \eqref{rhat-proj}, \eqref{bhat-proj}, \eqref{J-proj}, \eqref{eq:br-proj2}	&\ra \eqref{eq:local_5_1_error2} \\
    \eqref{rhat-proj}, \eqref{eq:br-proj1}, \eqref{eq:br-proj3} 			&\ra \eqref{eq:local_6_1_error2} 
  }
  Notice that \eqnref{local_error2} looks like \eqnref{local}, but with the approximation error replacing the finite element solution,
  and with some nonzero right hand side terms.
  Since the approximation error is in the finite element spaces, we can choose the test functions to be the approximation error terms.
  Similar to the procedure to arrive at \eqref{eq:sum_energy4_1}, we 
  take $(\Gb,\vb,q,\Jb,\cb,s) = (\elh,\eeuh,\eph,\ehh,\eebh,\eerh)$,
  integrate by parts the first four terms of \eqref{eq:local_2_1_error2} and the first term of \eqref{eq:local_5_1_error2},
  and sum the resulting equations in \eqnref{local_error2} to arrive at 
  \begin{align}
    \Rey\nor{\elh}_0^2 + \frac{\Rm}{\kappa}\nor{\ehh}_0^2
    - \LRa{\euhh \otimes \n, \elh}
    + \LRa{ \frac{m}{2} \eeuh, \eeuh } \notag
    + \LRa{ \alpha_1 (\eeuh-\euhh), \eeuh } \\
    \label{eq:sum_energy4_1_error2}
    + \LRa{ \half \kappa \db \times \LRp{\n \times \ebthh}, \eeuh }
    + \LRa{\euhh \cdot \n, \eph}
    - \LRa{ \n \times \ebthh, \ehh }
    + \LRa{ \erhh \n, \eebh } \\
    + \LRa{ \alpha_2(\eebth-\ebthh), \eebth } \notag
    - \LRa{ \half\kappa\n \times \LRp{\euhh \times \db}, \eebh }
    + \LRa{\alpha_3 \LRp{\eerh-\erhh}, \eerh} \notag \\
    =
    - \Rey\LRp{\eli, \elh}
    - \kappa\LRp{\eebi, \Curl\LRp{\eeuh \times \db}} 
    + \kappa\LRp{\eeui, \db \times \LRp{\Curl\eebh}} \notag \\
    - \LRa{\n \times \ehi -\half \kappa \n \times \LRp{\LRp{\eeui +\euhi}\times \db} + \alpha_2\eebti, \eebh}. \notag
  \end{align}
  For the boundary conditions and conservation conditions,
  since the exact solution satisfies \eqref{global}--\eqnref{BCs}, we have
  \begin{subequations}
    \begin{align}
      \eqnlab{global_error}
      \LRa{-\eL \n + m\eu +\ep\n + \half \kappa \db \times \LRp{\n \times \LRp{\ebt+\ebht}} 
      + \alpha_1 \LRp{\eu - \euh}, \mub}_{\pOmegah\setminus\pOmega} &= 0, \\
      \label{second-conservation-error-equation}
      \LRa{\n \times \eH - \half \kappa \n \times \LRp{\LRp{\eu+\euh} \times \db} 
      + \alpha_2\LRp{\ebt-\ebht}, \lambdab^t}_{\pOmegah\setminus\pOmega} &= 0, \\
      \LRa{\eb \cdot \n + \alpha_3\LRp{\er - \erh}, \gamma}_{\pOmegah\setminus\pOmega} &= 0, \\
      \LRa{\euh,\mub}_\pOmega &= 0, \\ 
      \LRa{\ebht,\lambdab^t}_\pOmega &= 0, \\
      \label{rhat-boundary-error-equation}
      \LRa{\erh,\gamma}_\pOmega &= 0.
    \end{align}
  \end{subequations}
  We split the error terms into their interpolation and approximation components as before, 
  and use the projections defined in Appendix~\secref{projections} to cancel terms. 
  More detailed roles for cancellations are
  \algns{
    \eqref{uhat-proj}, \eqref{eq:Lu-projection3} 		\ra \eqref{global_error2_cons_u}; \;
    \euhh : \text{single-valued}, \eqref{bhat-proj} 	\ra \eqref{global_error2_cons_b}; \; 
    \eqref{rhat-proj}, \eqref{eq:br-proj3} 		\ra \eqref{global_error2_cons_r}, 
  }
  and recall that the definitions of $\Pi$ for $\ubh$, $\bbh$, $\rh$ are $L^2$ projections. 
  Then we have
  \begin{subequations}
    \eqnlab{global_error2}
    \begin{align}
      \label{global_error2_cons_u}
      \LRa{-\elh \n + m\eeuh +\eph\n + \half \kappa \db \times \LRp{\n \times \eebth} 
      + \alpha_1 \LRp{\eeuh - \euhh}, \mub}_{\pOmegah\setminus\pOmega} 
      = 0, \\
      \label{global_error2_cons_b}
      \LRa{\n \times \ehh - \half \kappa \n \times \LRp{\eeuh \times \db} + \alpha_2\LRp{\eebth-\ebthh}, \lambdab^t}_{\pOmegah\setminus\pOmega}  \\
      = -\LRa{\n \times \ehi - \half \kappa \n \times \LRp{\LRp{\eeui+\euhi} \times \db} + \alpha_2\eebti, \lambdab^t}_{\pOmegah\setminus\pOmega} \notag , \\
      \label{global_error2_cons_r}
      \LRa{\eebh \cdot \n + \alpha_3\LRp{\eerh - \erhh}, \gamma}_{\pOmegah\setminus\pOmega} = 0, \\
      \label{global_error2_BC_u}
      \LRa{\euhh,\mub}_\pOmega = 0, \\ 
      \label{global_error2_BC_b}
      \LRa{\ebthh,\lambdab^t}_\pOmega = 0, \\
      \label{global_error2_BC_r}
      \LRa{\erhh,\gamma}_\pOmega = 0.
    \end{align}
  \end{subequations}
  In \eqref{global_error2_cons_b}, we can also erase $\euhi$ but we did not do it on purpose. 

  Equations \eqref{global_error2_BC_u}--\eqref{global_error2_BC_r} imply that on $\pOmega$, $\euhh = \bs{0}$, $\ebthh = \bs{0}$, and $\erhh= 0$.
  In consideration of this zero contribution on $\pd \Omega$, summing of the formulae \eqref{global_error2_cons_u}--\eqref{global_error2_cons_r} with $({\mub, \lambdab^t, \gamma}) = ({\euhh,\ebthh,\erhh})$ including $\pd \Omega$ gives 
  \begin{align}
    \label{global_error3}
    \LRa{-\elh \n + m\eeuh +\eph\n + \half \kappa \db \times \LRp{\n \times \eebth} + \alpha_1 \LRp{\eeuh - \euhh}, \euhh}
    \qquad \notag \\ 
    + \LRa{\n \times \ehh - \half \kappa \n \times \LRp{\eeuh \times \db} + \alpha_2\LRp{\eebth-\ebthh}, \ebthh}
    \qquad \notag \\ 
    + \LRa{\eebh \cdot \n + \alpha_3\LRp{\eerh - \erhh}, \erhh}
    \qquad \notag \\ 
    = 
    \LRa{\n \times \ehi - \half \kappa \n \times \LRp{\LRp{\eeui+\euhi} \times \db} + \alpha_2\eebti, \ebthh}
  \end{align}
  Subtracting \eqref{global_error3} from \eqref{eq:sum_energy4_1_error2}, we arrive at 
  \algns{
    &\Rey\nor{\elh}_0^2 + \frac{\Rm}{\kappa}\nor{\ehh}_0^2
    + \LRa{ \alpha_1 (\eeuh-\euhh), (\eeuh-\euhh) }
    + \LRa{ \frac{m}{2} \eeuh, \eeuh} \quad & \\
    &\quad - \LRa{ m \eeuh, \euhh} 
    + \alpha_2 \nor{\eebth-\ebthh}_\pOmegah^2
    + \alpha_3 \nor{\eerh-\erhh}_\pOmegah^2 \\
    &\quad =     - \Rey\LRp{\eli, \elh}
    - \kappa\LRp{\eebi, \Curl\LRp{\eeuh \times \db}} 
    + \kappa\LRp{\eeui, \db \times \LRp{\Curl\eebh}} \notag \\
    & \qquad - \LRa{\n \times \ehi -\half \kappa \n \times \LRp{\LRp{\eeui +\euhi}\times \db} + \alpha_2\eebti, \eebth - \ebthh}. 
  }
  Following the same procedure to get \eqref{eq:sum_energy6_1} from \eqref{eq:sum_energy6_2}, we obtain the conclusion.
\end{proof}

\begin{lemma}
  \theolab{energyEstimate0}
  There holds:
  \begin{align}
    \eqnlab{energy-est}
    E_h^2 
    &\lesssim \Rey \nor{\eL^I}_0 \nor{\eL^h}_0 
    + \kappa \nor{\db}_\Woneinfty 
    \LRp{ \nor{\eebi}_0 \nor{\eeuh}_0+ \nor{\eeui}_0 \nor{\eebh}_0 } \\
    & \quad + \LRp{ \nor{\ehi}_{\pOmegah} + \kappa \nor{\db}_\Linfty \nor{\eeui,\euhi}_{\pOmegah} + \alpha_2 \nor{\eebti}_{\pOmegah} } \nor{\ebt^h-\ebht^h}_{\pOmegah}.  \notag
  \end{align}
\end{lemma}
\begin{proof}
  It is clear that bounding the energy is the same as bounding the right
  hand side of \eqnref{disc-energy}. The estimate of $\Rey \LRp{\eL^I, \eL^h}$ is
  straightforward by Cauchy-Schwarz inequality.  To estimate  
  $\kappa\LRp{\eb^I, \Curl\LRp{\eu^h \times \db}}$, note
  first that an algebraic computation gives 
  \algn{ 
    \label{eq:curlcurl}
    \kappa\LRp{\eb^I,\Curl\LRp{\eu^h \times \db}} 
    & = \kappa\LRp{\eb^I, \eu^h \LRp{ \Div \db } - \LRp{\eu^h \cdot \Grad} \db } \\ 
    &\quad+ \kappa\LRp{\eb^I, \LRp{\db \cdot \Grad} \eu^h - \db \LRp{\Div \eu^h}} .  \notag
  } 
  The boundedness of the left hand side can be obtained by 
  \algns{ 
    &\kappa \snor{\LRp{\eb^I, \eu^h \LRp{ \Div \db } - \LRp{\eu^h \cdot \Grad} \db}_K} 
    \leq \kappa \| \eb^I \|_{0,K} \| \eu^h \|_{0,K} \| \db \|_{W^{1,\infty}(K)}, \\ 
    &\kappa \snor{\LRp{\eb^I, \LRp{\db \cdot \Grad} \eu^h - \db \LRp{\Div \eu^h}}_K} \\
    &\qquad = \kappa \snor{\LRp{\eb^I, \LRp{ \LRp{ \db - \Pr_0 \db } \cdot \Grad} \eu^h - \LRp{ \db - \Pr_0 \db }\LRp{\Div \eu^h}}_K} \\ 
    &\qquad \lesssim \kappa h_K \| \eb^I \|_{0,K} \| \db \|_{W^{1,\infty}(K)} \| \Grad \eu^h \|_{0,K}
    \lesssim \kappa \| \eb^I \|_{0,K} \| \db \|_{W^{1,\infty}(K)} \|\eu^h \|_{0,K}  
  }
  where $\Pr_0 \db$ is the $L^2$ projection of $\db$ to the piecewise constant space on $\K$, and 
  here we used \eqref{eq:br-proj1}, H\"older
  inequality $\| f_1 f_2 f_3 \|_{L^1} \leq \| f_1 \|_{L^2} \| f_2
  \|_{L^\infty} \| f_3 \|_{L^2}$, the Bramble--Hilbert lemma (see e.g., \cite{Brenner-Scott-book}) and the
  inverse estimate in the last two inequalities.

  For an estimate of $\kappa\LRp{\eu^I, \db \times \LRp{\Curl \eb^h}}$, we first note that
  \algns{ 
    \kappa\LRp{\eu^I, \db \times \LRp{\Curl \eb^h}} = \kappa\LRp{\eu^I, (\db - \Pr_0 \db) \times \LRp{\Curl \eb^h}}
  }
  due to \eqref{eq:Lu-projection2}. A similar argument as above gives
  \algns{  
    \kappa |\LRp{\eu^I, \db \times \LRp{\Curl \eb^h}} |_K 
    &\leq \kappa \|\eu^I \|_{0,K} \| \db - \Pr_0 \db \|_{L^\infty(K)} \| \Curl \eb^h \|_{0,K} \\
    &\lesssim \kappa \| \eu^I \|_{0,K} \| \db \|_{W^{1,\infty}(K)} \| \eb^h \|_{0,\K}.
  }
  Finally, we simply use the Cauchy-Schwarz/H\"older inequality for the last term on the right hand side of \eqnref{disc-energy}.
\end{proof}

\begin{corollary}[Energy estimate]
  \theolab{energyEstimate}
  There holds:
  \begin{align}
    \eqnlab{energy-est1}
    E_h^2 &\lesssim \Rey \nor{\eL^I}_0^2 
    + \kappa \nor{\db}_\Woneinfty
    \LRp{ \nor{\eebi}_0 \nor{\eeuh}_0+ \nor{\eeui}_0 \nor{\eebh}_0 } \\
    & \quad + \alpha_2^{-1} \nor{\ehi}_\pOmegah^2 
    + \kappa^2 \nor{\db}_\Linfty^2 \nor{\eeui}_{\pOmegah}^2 
    + \alpha_2 \nor{\eebti}_{\pOmegah}^2 .  \notag
  \end{align}
\end{corollary}

\begin{proof}
  Apply Young's inequality to each of the terms on the right side of \eqref{eq:energy-est}
  involving $\nor{\elh}_0$ and $\|{\eebth - \ebthh}\|_\pOmegah$.
  Note also that $\Pi \ubh$ is the best approximation of $\ub$ on $\pOmegah$, so
  $\nor{\euhi}_\pOmegah$ is bounded by $\nor{\eeui}_\pOmegah$.
\end{proof}

In the energy estimate \eqnref{energy-est1}, we do not
have direct control on $\| \eu^h \|_0$ and $\| \eb^h \|_0$. 
In the following, we employ an
indirect approach to control these quantities via a duality
argument. A similar approach for the Oseen equation has been conducted
in \cite{CesmeliogluCockburnNguyenEtAl13}. It is more complicated for
our MHD system due to the coupling between fluids and
electromagnetics.  In particular, $\eu^h$ and $\eb^h$ are coupled and
have to be simultaneously analyzed.  To begin, we define a
dual problem of the MHD system \eqnref{mhdlin-first} as
\begin{subequations}
  \eqnlab{adjoint-mhd}
  \algn{
    \label{eq:amhd1lin1n}
    \Rey \Lb^* - \nabla \ub^* &= \bs{0}, \\
    \label{eq:amhd2lin1n}
    - \nabla \cdot \Lb^* - \nabla p^* - (\wb \cdot \nabla) \ub^* - \kappa \db \times (\nabla \times \bb^*) &= \bs{\theta}, \\
    \label{eq:amhd3lin1n}
    - \nabla \cdot \ub^* &= 0, \\
    \label{eq:amhd4lin1n}
    \frac{\Rm}{\kappa} \Hb^* - \nabla \times \bb^* &= \bs{0}, \\
    \label{eq:amhd5lin1n}
    \nabla \times \Hb^* - \nabla r^* + \kappa \nabla \times (\ub^* \times \db) &= \bs{\sigma}, \\
    \label{eq:amhd6lin1n}
    - \nabla \cdot \bb^* &= 0
  }
\end{subequations}
with homogeneous boundary conditions. Here, $\thetab$ and $\sigmab$ are two given functions in
$L^2\LRp{\Omega}$, and the superscript ``*'' is used to denote the corresponding unknowns in the adjoint equation. 
We assume the following regularity estimate holds for the adjoint problem \eqnref{adjoint-mhd}
\algn{ 
  \eqnlab{adjoint-regularity} 
  \nor{\ub^*}_{2} +\nor{\bb^*, \Lb^*, \Hb^* , \p^*, \r^*}_{1} 
  \lesssim \nor{\thetab, \sigmab}_{0}.
}
The well-posedness of \eqnref{adjoint-mhd} and the  conditions under which  the regularity estimate \eqnref{adjoint-regularity} holds are discussed in Appendix \secref{regularity_of_adjoint}.

We use the interpolation operators $\Pi^*$ defined in Appendix A \eqnref{abr-proj}, \eqnref{adjoint-projection} below and 
$\elis$, $\epis$, ... will denote $\Lb^* - \Pi^* \Lb^*$, $\p^* - \Pi^* \p^*$, etc.
Testing \eqnref{amhd2lin1n} with $\eeuh$ and \eqnref{amhd5lin1n} with $\eebh$ we have
\algn{ 
  \eqnlab{ubadjoint}
  &\LRp{\eeuh, \bs{\theta}} + \LRp{\eebh, \bs{\sigma}} \\ 
  &\quad 		= \LRp{\eeuh, -\Div \Lb^* - \Grad \p^* - (\wb \cdot \Grad) \ub^* - \kappa \db \times (\curl \bb^*) } \notag\\
  &\qquad + \LRp{ \eebh, \curl \Hb^* - \Grad r^* + \kappa \curl (\ub^* \times \db) } \notag\\ 
  &\quad 		= \LRp{ \Grad \eeuh, \Lb^* } + \LRp{ \Div \eeuh, \p^* } + \LRp{ (\wb \cdot \Grad) \eeuh, \ub^* } 
  - \kappa \LRp{ \eeuh, \db \times (\curl \bb^*) } \notag\\
  &\qquad + \LRa{ \eeuh, - \Lb^* \n - p^* \n - m \ub^* }  
  + \LRp{ \curl \eebh, \Hb^* } + \LRp{ \Div \eebh, \r^* } \notag\\
  &\qquad + \kappa \LRp{ \eebh, \curl (\ub^* \times \db)} + \LRa{ \eebh, \n \times \Hb^* - r^* \n } \notag \\
  &\quad 		= \LRp{ \Grad \eeuh, \Pis \Lb^* } + \LRp{ \Div \eeuh, \Pis \p^* } 
  + \LRp{ (\wb \cdot \Grad) \eeuh, \Pis \ub^* } \notag \\
  &\qquad- \kappa \LRp{ \eeuh, \db \times (\curl \bb^*) }
  + \LRa{ \eeuh, - \Lb^* \n - p^* \n - m \ub^* }  
  + \LRp{ \curl \eebh, \Pis \Hb^* } \notag\\
  &\qquad + \LRp{ \Div \eebh, \Pis \r^* }
  + \kappa \LRp{ \curl (\ub^* \times \db), \eebh } + \LRa{ \eebh, \n \times \Hb^* - r^* \n }\notag \\
  &\quad 		= \LRp{\eeuh, -\Div \Pis \Lb^* - \Grad \Pis \p^* - (\wb \cdot \Grad) \Pis \ub^* - \kappa \db \times (\curl \bb^*)}  \notag\\ 
  &\qquad + \LRp{\eebh, \curl \Pis \Hb^* - \Grad \Pis r^* + \kappa \curl (\ub^* \times \db) }  \notag\\
  &\qquad + \LRa{ \eeuh, - \elis \n - \epis \n - m \eeuis } + \LRa{ \eebh, \n \times \ehis - \eeris \n } \notag 
}
where we have used integration by parts in the second equality, 
the properties of the $\Pi^*$ operators \eqref{eq:abr-proj} and \eqref{eq:adjoint-projection} in the third equality,
and integration by parts again in the last equality.

\begin{lemma}
  \lemlab{errorPrimal-adjoint}
  The following identities hold true:
  \algns{ 
    \LRp{ \eeuh, - \Div \Pis \Lb^* } 
    &= \Rey \LRp{ \elh, \Pis \Lb^* } - \LRa{ \euhh, \Pis \Lb^* \n } + \Rey \LRp{ \eli, \Pis \Lb^* } , \\
    - \LRp{ \eeuh, (\wb \cdot \Grad) \Pis \ub^* } 	
    &= - \LRp{ \elh, \Grad \Pis \ub^* } + \LRp{ \eph, \Div \Pis \ub^* } 
    - \kappa \LRp{ \eebh, \curl (\Pis \ub^* \times \db ) } \\
    &\quad - \LRa{ - \elh \n + m \eeuh + \eph \n + \half \kappa \db \times (\n \times (\eebh + \ebhh)), \Pis \ub^* } \\
    &\quad - \LRa{ \alpha_1 (\eeuh - \euhh), \Pis \ub^* } 
    - \kappa \LRp{ \eebi, \curl (\Pis \ub^* \times \db) }, \\
    \LRp{ \eeuh, - \Grad \Pis p^* }	
    &= - \LRa{ \euhh , \Pis p^* \n }, \\ 
    \LRp{ \eebh, \curl \Pis \Hb^* }	
    &= \frac{\Rm}{\kappa} \LRp{ \ehh, \Pis \Hb^* } - \LRa{ \n \times \ebthh, \Pis \Hb^* } , \\
    -\kappa \LRp{ \eeuh, \db \times (\curl \bb^*) } 
    &= - \kappa \LRp{ \eeuh, \db \times (\curl \Pis \bb^*) } - \kappa \LRp{ \eeuh, \db \times (\curl \eebis ) } \\
    &= - \LRp{ \ehh, \curl \Pis \bb^* } + \LRp{ \eerh, \Div \Pis \bb^* } 
    - \LRa{ \alpha_2 (\veps_{\bb^t} - \veps_{\bbh^t} ), \Pis \bb^* }\\
    &\quad - \LRa{ \n \times \veps_{\Hb} + \veps_{\rh} \n - \half \kappa \n \times (( \veps_{\ub} + \veps_{\ubh}) \times \db ) , \Pis \bb^* } \\
    &\quad + \kappa \LRp{ \eeui, \db \times (\curl \Pis \bb^* ) } - \kappa \LRp{ \eeuh, \db \times (\curl \eebis ) } , \\
    \LRp{\eebh, - \Grad \Pis r^* } 			&= - \LRa{ \eebh \cdot \n + \alpha_3 (\eerh - \erhh), \Pis r^* } .
  }
\end{lemma}
\begin{proof}
  In \eqnref{local_error2}, choose test functions $(\Gb,\vb,\q,\Jb,\cb,\s)$ to be the adjoint projections
  $(\Pis \Lb^*, \Pis \ub^*, \Pis \p^*, \Pis \Hb^*, \Pis \bb^*, \Pis \r^*)$.
  Noting the identities 
  \algns{
    \LRa{ \euhh \otimes \n, \Pis \Lb^*} = \LRa{ \euhh, \Pis \Lb^* \n }, \qquad \LRp{ \eeuh \otimes \wb , \Grad \Pis \ub^*} = \LRp{ \eeuh, (\wb \cdot \Grad) \Pis \ub^*}, 
  }
  all the results follow immediately except the fifth one. 
  In the fifth identity, note that the first equality is trivial by the definitions of $\Pis \bb^*$ and $\eebis$,
  and the second equality comes from \eqnref{local_error2} with the orthogonality of $\ebthi$ and $\erhi$.
\end{proof}

Substituting the result of Lemma \lemref{errorPrimal-adjoint} into \eqnref{ubadjoint} we obtain
\algns{ 
  &\LRp{\eeuh, \bs{\theta}} + \LRp{\eebh, \bs{\sigma}} \\ 
  &\; = \Rey \LRp{ \elh, \Pis \Lb^* } - \LRa{ \euhh, \Pis \Lb^* \n } + \Rey \LRp{ \eli, \Pis \Lb^* }
  - \LRa{ \euhh , \Pis p^* \n } \\
  &\quad - \LRp{ \elh, \Grad \Pis \ub^* } + \LRp{ \eph, \Div \Pis \ub^* } - \kappa \LRp{ \eebh, \curl (\Pis \ub^* \times \db ) } \\
  &\quad - \LRa{ - \elh \n + m \eeuh + \eph \n + \half \kappa \db \times (\n \times (\eebh + \ebhh)) + \alpha_1 (\eeuh - \euhh), \Pis \ub^* } \\
  &\quad - \kappa \LRp{ \eebi, \curl (\Pis \ub^* \times \db) } 
  - \LRp{ \ehh, \curl \Pis \bb^* } + \LRp{ \eerh, \Div \Pis \bb^* } \\
  &\quad - \LRa{ \n \times \veps_{\Hb} + \veps_{\rh} \n - \half \kappa \n \times (( \veps_{\ub} + \veps_{\ubh}) \times \db ) 
  + \alpha_2 (\veps_{\bb^t} - \veps_{\bbh^t} ), \Pis \bb^* } \\
  &\quad + \kappa \LRp{ \eeui, \db \times (\curl \Pis \bb^* ) } - \kappa \LRp{ \eeuh, \db \times (\curl \eebis ) } \\
  &\quad + \frac{\Rm}{\kappa} \LRp{ \ehh, \Pis \Hb^* } - \LRa{ \n \times \ebthh, \Pis \Hb^* } - \LRa{ \eebh \cdot \n + \alpha_3 (\eerh - \erhh), \Pis r^* } \\
  &\quad + \kappa \LRp{ \eebh, \curl (\ub^* \times \db) } + \LRa{ \eeuh, - \elis \n - \epis \n - m \eeuis }  + \LRa{ \eebh, \n \times \ehis - \eeris \n } .
}
Note first that the last term in the second line and the first term in the last line can be simplified as 
$\kappa \LRp{ \eebh, \curl (\eeuis \times \db) }$.
Note also that we can use \eqref{eq:amhd1lin1n}, \eqref{eq:amhd3lin1n}, \eqref{eq:amhd4lin1n}, \eqref{eq:amhd6lin1n} to get 
\begin{subequations}
  \eqnlab{eq:sub}
  \algn{ 
    \label{eq:sub1}  \Rey \LRp{\elh, \Pis \Lb^* } - \LRp{ \elh, \Grad \Pis \ub^* } &= - \Rey \LRp{\elh,  \elis } + \LRp{ \elh, \Grad \eeuis }, \\
    \label{eq:sub2} \Div \Pis \ub^* &= - \Div \eeuis, \\
    \label{eq:sub3} \frac{\Rm}{\kappa} \LRp{ \ehh, \Pis \Hb^* } - \LRp{\ehh, \curl \Pis \bb^* } &= - \frac{\Rm}{\kappa} \LRp{ \ehh, \ehis } + \LRp{\ehh, \curl \eebis } , \\
    \label{eq:sub4} - \Div \Pis \bb^* &= \Div \eebis .
  }
\end{subequations}
We can subtract the $L^2$ projections on the mesh skeleton, $\Pr_e\ub^*$, $\Pr_e\bb^*$, $\Pr_e\r^*$ of  $\ub^*$, $\bb^*$, $\r^*$, 
in flux terms using the flux conservation of interpolation operators and numerical solutions, i.e., \eqref{global_error2_cons_u}, \eqref{global_error2_cons_b}, \eqref{global_error2_cons_r}. Moreover, 
we can subtract them on the domain boundary since $\erhh$, $\ub^*$, $(\bb^*)^t$, and $\r^*$ are zero there. If we use the above three observations, i.e., \tbrown{the simplification of the two terms}, \tred{substitution of volume integral terms using \eqnref{eq:sub}}, and \tblu{the subtraction of $L^2$ projections on the mesh skeleton}, then we have (some terms are colored for readers' convenience)
\algns{ 
  &\LRp{\eeuh, \bs{\theta}} + \LRp{\eebh, \bs{\sigma}} \\ 
  &\; = \Rey \LRp{ \elh, \tred{-\elis} } - \LRa{ \euhh, \Pis \Lb^* \n } + \Rey \LRp{ \eli, \Pis \Lb^* }
  - \LRa{ \euhh , \Pis p^* \n } \\
  &\quad - \LRp{ \elh, \tred{- \Grad \eeuis } } + \LRp{ \eph, \tred{-\Div \eeuis } } + \tbrown{ \kappa \LRp{ \eebh, \curl (\eeuis \times \db ) } } \\
  &\quad - \LRa{ - \elh \n + m \eeuh + \eph \n + \half \kappa \db \times (\n \times (\eebh + \ebhh)) + \alpha_1 (\eeuh - \euhh), \Pis \ub^* \tblu{- \Pr_e \ub^* } } \\
  &\quad - \kappa \LRp{ \eebi, \curl (\Pis \ub^* \times \db) } 
  - \LRp{ \ehh, \tred{-\curl \eebis} } + \LRp{ \eerh, \tred{-\Div \eebis } } \\
  &\quad - \LRa{ \n \times \veps_{\Hb} + \veps_{\rh} \n - \half \kappa \n \times (( \veps_{\ub} + \veps_{\ubh}) \times \db ) 
  + \alpha_2 (\veps_{\bb^t} - \veps_{\bbh^t} ), \Pis \bb^* \tblu{- \Pr_e \bb^* } } \\
  &\quad + \kappa \LRp{ \eeui, \db \times (\curl \Pis \bb^* ) } - \kappa \LRp{ \eeuh, \db \times (\curl \eebis ) } \\
  &\quad + \frac{\Rm}{\kappa} \LRp{ \ehh, \tred{- \ehis} } - \LRa{ \n \times \ebthh, \Pis \Hb^* } - \LRa{ \eebh \cdot \n + \alpha_3 (\eerh - \erhh), \Pis r^* \tblu{- \Pr_e \r^* } } \\
  &\quad + \kappa \LRp{ \eebh, \curl (\ub^* \times \db) } + \LRa{ \eeuh, - \elis \n - \epis \n - m \eeuis }  + \LRa{ \eebh, \n \times \ehis - \eeris \n } .
}
Note that the integration by parts and properties of $\eeuis$ and $\Pr_e \ub^*$ yield 
\algns{
  \LRp{ \elh, \Grad \eeuis } + \LRa{\elh \n, \Pis \ub^* - \Pr_e \ub^* } &= - \LRp{\Div \elh, \eeuis } + \LRa{ \elh \n, \ub^* - \Pr_e \ub^* } = 0 , \\
  - \LRp{ \eph, \Div \eeuis } - \LRa{\eph \n, \Pis \ub^* - \Pr_e \ub^*  } &= \LRp{\Grad \eph, \eeuis } - \LRa{\eph \n, \ub^* - \Pr_e \ub^* } = 0 .
}
These two identities simplify the second and third lines of the above formula and yield
\algns{ 
  &\LRp{\eeuh, \bs{\theta}} + \LRp{\eebh, \bs{\sigma}} \\ 
  &\; = \Rey \LRp{ \elh, -\elis } 
  - \LRa{ \euhh, \Pis \Lb^* \n } 
  + \Rey \LRp{ \eli, \Pis \Lb^* } 
  - \LRa{ \euhh , \Pis p^* \n } \\
  &\quad - \LRa{ m \eeuh + \half \kappa \db \times (\n \times (\eebh + \ebhh)) + \alpha_1 (\eeuh - \euhh), \Pis \ub^* \tblu{- \Pr_e \ub^* } } \\
  &\quad - \kappa \LRp{ \eebi, \curl (\Pis \ub^* \times \db) } 
  - \LRp{ \ehh, \tred{-\curl \eebis} } + \LRp{ \eerh, \tred{-\Div \eebis } } \\
  &\quad - \LRa{ \n \times \veps_{\Hb} + \veps_{\rh} \n - \half \kappa \n \times (( \veps_{\ub} + \veps_{\ubh}) \times \db ) 
  + \alpha_2 (\veps_{\bb^t} - \veps_{\bbh^t} ), \Pis \bb^* \tblu{- \Pr_e \bb^* } } \\
  &\quad + \kappa \LRp{ \eeui, \db \times (\curl \Pis \bb^* ) } 
  - \kappa \LRp{ \eeuh, \db \times (\curl \eebis ) } 
  +  \kappa \LRp{ \eebh, \curl (\eeuis \times \db ) }  \\
  &\quad + \frac{\Rm}{\kappa} \LRp{ \ehh, \tred{- \ehis} } - \LRa{ \n \times \ebthh, \Pis \Hb^* } - \LRa{ \eebh \cdot \n + \alpha_3 (\eerh - \erhh), \Pis r^* \tblu{- \Pr_e \r^* } } \\
  &\quad + \LRa{ \eeuh, - \elis \n - \epis \n - m \eeuis }  + \LRa{ \eebh, \n \times \ehis - \eeris \n } .
}
To simplify it further, we note the identities
\algns{
  & \LRp{ \ehh, \curl \eebis } = \LRp{ \curl \ehh, \eebis} + \LRa{ \ehh, \n \times \eebis } = \LRa{ \ehh, \n \times \eebis } , \\
  & - \LRp{ \eerh, \Div \eebis } = \LRp{ \Grad \eerh, \eebis} - \LRa{\eerh \n, \eebis } = -\LRa{\eerh \n, \eebis } , \\
  & \LRa{ \n \times \veps_{\Hb} + \veps_{\rh} \n - \half \kappa \n \times (( \veps_{\ub} + \veps_{\ubh}) \times \db ) 
  + \alpha_2 (\veps_{\bb^t} - \veps_{\bbh^t} ), \Pis \bb^* - \Pr_e \bb^* }  \\
  & = \LRa{ \n \times \ehh + \erhh \n - \half \kappa \n \times (( \eeuh + \euhh) \times \db ) + \alpha_2 (\eebth - \ebthh), \Pis \bb^* { - \Pr_e \bb^*} } \\
  & \quad + \LRa{ \n \times \ehi - \half \kappa \n \times (( \eeui + \euhi) \times \db ) + \alpha_2 \eebti , \Pis \bb^* { - \Pr_e \bb^*} } , \\
  & -\LRa{\eebh \cdot \n, \Pis \r^* - \Pr_e \r^* } - \LRa{\eebh, \eeris \n} = 0 
}
where cancellations come from the properties of $\eebis$, $\Pr_e \bb^*$, $\Pr_e \r^*$ and the orthogonalities of $\erhi$ and $\ebthi$ to polynomials. 
We use these identities to the previous formula, more precisely, to the last two terms in the third line, to the fourth line, to the last terms in the last two lines. 
Then we obtain (here some terms are colored only for readers' convenience later)
\algns{ 
  &\LRp{\eeuh, \bs{\theta}} + \LRp{\eebh, \bs{\sigma}} \\ 
  &\; = \Rey \LRp{ \elh, -  \elis  } 
  - \LRa{ \tred{\euhh}, \Pis \Lb^* \n } 
  + \Rey \LRp{ \eli, \Pis \Lb^* } 
  - \LRa{ \tred{\euhh} , \Pis p^* \n } \\
  &\quad - \LRa{  m \tblu{\eeuh} + \half \kappa \db \times (\n \times (\tcyan{\eebth} + \tbrown{\ebhh})) + \alpha_1 (\tblu{\eeuh} - \tred{\euhh}), \Pis \ub^* {- \Pr_e \ub^*} } \\
  &\quad - \kappa \LRp{ \eebi, \curl (\Pis \ub^* \times \db) }  - \LRa{\n \times \ehh, \eebis  }  -  \LRa{ \tgreen{\eerh} \n,  \eebis  }  \\
  &\quad - \LRa{ \n \times \ehh + \tpurple{\erhh} \n - \half \kappa \n \times (( \tblu{\eeuh} + \tred{\euhh}) \times \db ) + \alpha_2 (\tcyan{\eebth} - \tbrown{\ebthh}), \Pis \bb^* { - \Pr_e \bb^*} } \\
  &\quad - \LRa{ \n \times \ehi - \half \kappa \n \times (( \eeui + \euhi) \times \db ) + \alpha_2 \eebti , \Pis \bb^* { - \Pr_e \bb^*} } \\
  &\quad + \kappa \LRp{ \eeui, \db \times (\curl \Pis \bb^* ) } 
  - \kappa \LRp{ \tblu{\eeuh}, \db \times (\curl \eebis ) } 
  + \kappa \LRp{ \eebh, \curl (\eeuis \times \db ) } \\ 
  &\quad + \frac{\Rm}{\kappa} \LRp{ \ehh, - { \ehis } } - \LRa{ \n \times \ebthh, \Pis \Hb^* } - \LRa{ \alpha_3 (\tgreen{\eerh} - \tpurple{\erhh}), \Pis r^* {- \Pr_e r^*} } \\
  &\quad - \LRa{ \tblu{\eeuh}, \elis \n + \epis \n + m \eeuis } + \LRa{ \tcyan{\eebth}, \n \times \ehis  } .
}
Here we used the facts $\n \times \eebh = \n \times \eebth$ and $\LRa{ \eebh, \n \times \ehis  } =\LRa{ \eebth, \n \times \ehis  }$.
Algebraic manipulations with \eqnref{eq:int-by-parts} yield
\algns{ 
  &\LRp{\eeuh, \bs{\theta}} + \LRp{\eebh, \bs{\sigma}} \\ 
  &\; = \Rey { \LRp{ \elh, -  \elis }} + \Rey { \LRp{ \eli, \Pis \Lb^* }} + \frac{\Rm}{\kappa} {\LRp{ \ehh, -  \ehis  }} + \kappa { \LRp{ \eebh, \curl (\eeuis \times \db ) } }\\ 
  &\quad - \kappa {\LRp{ \eebi, \curl (\Pis \ub^* \times \db) }} + \kappa {\LRp{ \eeui, \db \times (\curl \Pis \bb^* ) }} - \kappa {\LRp{ \eeuh, \db \times (\curl \eebis ) }}  \\ 
  &\quad {- \LRa{ \tred{\euhh}, \Pis \Lb^* \n + \Pis p^* \n - \alpha_1 (\Pis \ub^* - \Pr_e \ub^*) + \half \kappa \db \times (\n \times (\Pis \bb^* - \Pr_e \bb^*) ) }} \\
  &\quad {- \LRa{ \tblu{\eeuh} , \elis \n + \epis \n + m \eeuis + \half \kappa \db \times (\n \times ( \Pis \bb^* - \Pr_e \bb^* ) )  }} \\
  &\quad + {\LRa{ \tbrown{\ebthh}, \half \kappa \n \times \LRp{ (\Pis \ub^* - \Pr_e \ub^*) \times \db } + \n \times \Pis \Hb^* + \alpha_2 (\Pis \bb^* - \Pr_e \bb^*) }} \\
  &\quad + {\LRa{ \tcyan{\eebth}, \half \kappa \n \times \LRp{ (\Pis \ub^* - \Pr_e \ub^*) \times \db } + \n \times \ehis - \alpha_2 (\Pis \bb^* - \Pr_e \bb^*) }} \\ 
  &\quad + {\LRa{ \tpurple{\erhh}, \alpha_3 (\Pis r^* - \Pr_e r^*) - (\Pis \bb^* - \Pr_e \bb^*) \cdot \n }}  \\
  &\quad + {\LRa{ \tgreen{\eerh}, - \alpha_3 (\Pis r^* - \Pr_e r^*) - \eebis \cdot \n }} 
   {- \LRa{ \tblu{\eeuh} , \LRp{ m + \alpha_1 } \LRp{ \Pis \ub^* - \Pr_e \ub^* } }} \\
  &\quad + {\LRa{ \ehh, \n \times \eebis + \n \times (\Pis \bb^* - \Pr_e \bb^*) }} \\
  &\quad - {\LRa{ \n \times \ehi - \half \kappa \n \times (( \eeui + \euhi) \times \db ) + \alpha_2 \eebti , \Pis \bb^*  - \Pr_e \bb^* }} .  
}
We reduce this further, particularly from the third to ninth lines. 
For the third line, note that $\euhh$, $\db$, and the exact solution are single-valued, and $\LRa{ \euhh, \Pis \ub^* - \Pr_e \ub^*} = \LRa{ \euhh, \eeuis}$. Note also that $\Pis \bb^* - \Pr_e \bb^* = - \eebis + \ebhis$ holds. Then the third line is 
\algns{
  \LRa{ \euhh, \elis \n + \epis \n - \alpha_1 \eeuis - \half \kappa \db \times (\n \times ( - \eebis + \ebthis ) ) } = \tred{0}
}
where the vanishing equality comes from \eqref{eq:aLu-proj3}. For the fourth line, note first that 
\algns{
  \LRa{\eeuh, m \eeuis  + m (\Pis \ub^* - \Pr_e \ub^*) }  &= \LRa{\eeuh, m \euhis }, \\
  \LRa{\eeuh, \alpha_1 (\Pis \ub^* - \Pr_e \ub^*) }  &= - \LRa{\eeuh, \alpha_1 \eeuis }
}
where the second equality is valid because $\alpha_1$ is constant on each facet. Then the fourth line is 
\algns{
  &\LRa{ \eeuh, \alpha_1 \eeuis - m \euhis - \elis \n - \epis \n - \half \kappa \db \times (\n \times ( - \eebis + \ebthis ) ) } , \\
  &\quad = - \LRa{ \eeuh, m \euhis + \kappa \db \times (\n \times ( - \eebis + \ebthis ) ) }  
}
where we use \eqref{eq:aLu-proj3} again. Note that
\algns{
  \Pi^* \ub^* - \Pr_e \ub^* = -\eeuis + \euhis, \qquad \Pr_e ((\Pi^* \bb^* - \Pr_e \bb^*)^t) = - \Pr_e \veps^I_{(\bb^*)^t} .
} 
From these the fifth and sixth lines are rewritten as 
\algns{
  \LRa{ \ebthh, \half \kappa \n \times ( ( - \eeuis + \euhis ) \times \db ) + \n \times (- \ehis ) - \alpha_2 \veps^I_{(\bb^*)^t} }, \\
  \LRa{ \eebth, \half \kappa \n \times \LRp{ ( - \eeuis + \euhis ) \times \db } + \n \times \ehis + \alpha_2  \veps^I_{(\bb^*)^t} }  .
}
Further, the seventh, eighth, ninth lines are vanishing, i.e., 
\algns{
  \LRa{ {\erhh}, \alpha_3 (\Pis r^* - \Pr_e r^*) - (\Pis \bb^* - \Pr_e \bb^*) \cdot \n } &= \LRa{ \erhh, - \alpha_3 \eeris + \eebis \cdot \n } = \tred{0}, \\
  \LRa{ {\eerh}, - \alpha_3 (\Pis r^* - \Pr_e r^*) - \eebis \cdot \n } &= \LRa{ \eerh, \alpha_3 \eeris - \eebis \cdot \n } = \tred{0}, \\
  \LRa{ \ehh, \n \times \eebis + \n \times (\Pis \bb^* - \Pr_e \bb^*) }&= \tred{0}
}
where we used  \eqref{eq:abr-proj3} for the first two identities, and the fact $\n \times \eebis = \n \times \veps^I_{(\bb^*)^t}$ for the third identity. Finally, $\LRp{ \ehh, -  \ehis  } = \tred{0}$ by the definition of $\ehis$. 
As a consequence, we have a reduced formula
\algn{ 
  \label{eq:adj-id}
  &\LRp{\eeuh, \bs{\theta}} + \LRp{\eebh, \bs{\sigma}} \\ 
  &\quad = \Rey\underbrace{ \LRp{ \elh, -  \elis }}_{=:I_1}  
  + \Rey\underbrace{ \LRp{ \eli, \Pis \Lb^* }}_{=:I_2}  
  + \kappa\underbrace{ \LRp{ \eebh, \curl (\eeuis \times \db ) } }_{=:I_3} \notag \\ 
  &\qquad - \kappa \underbrace{\LRp{ \eebi, \curl (\Pis \ub^* \times \db) }}_{=:I_4} 
  + \kappa \underbrace{\LRp{ \eeui, \db \times (\curl \Pis \bb^* ) }}_{=:I_5} 
  - \kappa \underbrace{\LRp{ \eeuh, \db \times (\curl \eebis ) }}_{=:I_6} \notag \\ 
  &\qquad \underbrace{- \LRa{ \eeuh, m \euhis + \kappa \db \times (\n \times ( - \eebis + \ebthis ) ) } }_{=:I_{7}} \notag \\ 
  &\qquad + \underbrace{\LRa{ \ebthh, \half \kappa \n \times ( ( - \eeuis + \euhis ) \times \db ) + \n \times (- \ehis ) - \alpha_2 \veps^I_{(\bb^*)^t} }}_{=:I_{8}} \notag \\ \displaybreak
  &\qquad + \underbrace{\LRa{ \eebth, \half \kappa \n \times \LRp{ ( - \eeuis + \euhis ) \times \db } + \n \times \ehis + \alpha_2  \veps^I_{(\bb^*)^t} }}_{=:I_{9}} \notag \\ 
  &\qquad - \underbrace{\LRa{ \n \times \ehi - \half \kappa \n \times (( \eeui + \euhi) \times \db ) + \alpha_2 \eebti , \Pis \bb^*  - \Pr_e \bb^* }}_{=:I_{10}} . \notag  
}

\noindent {\bf Estimation for $I_1$:} Combining the estimate for $\elis$ 
and the regularity estimate (see \eqref{eq:elliptic-estimate}) gives
\algn{ 
  \label{eq:I1-estm}
  \snor{\Rey I_1} 
  \le \Rey \nor{\elh}_0 \nor{\elis}_0 
  \lesssim h \Rey \nor{\elh}_0 \nor{\bs{\theta}, \bs{\sigma}}_0 .
}

\noindent {\bf Estimation for $I_2$:} 
Using \eqnref{amhd1lin1n}, \eqref{eq:Lu-projection1}, \eqref{eq:Lu-projection2},
Lemma \lemref{gradProjectionBound}, and the regularity of the adjoint solutions, we have
\algn{ 
  \label{eq:I2-estm}
  | \Rey I_2 | 
  &\leq | \Rey \LRp{ \eli, \Grad \ub^* } | + | \Rey \LRp{ \eli, - \elis } | \\
  &= | \Rey \LRp{ \eli, \Grad (\ub^* - \Pr_1 \ub^*) } + \LRp{ \eeui \otimes \wb, \Grad \Pr_1 \ub^* } | 
  + | \Rey \LRp{ \eli, - \elis } | \notag \\
  &\lesssim \Rey \LRp{ h \nor{\eli}_0 \nor{\ub^*}_2 
  + \nor{\wb - \Pr_0 \wb}_\Linfty \nor{\eeui}_0 \nor{\ub^* }_1 
  + \nor{\eli}_0 \nor{\elis}_0 } \notag \hspace{-0.5em} \\
  &\lesssim h \Rey \LRp{\nor{\eli}_0 + \nor{\wb}_\Woneinfty \nor{\eeui}_0} \nor{\bs{\theta}, \bs{\sigma}}_0 . \notag
}

\noindent {\bf Estimation for $I_4$:} By the identity \eqref{eq:curlcurl}, it suffices to estimate 
\algns{ 
  &\LRp{\eb^I, \Pis \ub^* \LRp{ \Div \db } - \LRp{\Pis \ub^* \cdot \Grad} \db }, \quad \text{ and } \\ 
  &\LRp{\eb^I, \LRp{(\db - \Pr_0 \db) \cdot \Grad} \Pis \ub^*  - (\db - \Pr_0 \db) \LRp{\Div \Pis \ub^* }} .
}
By the triangle inequality, the inverse estimate, and \eqref{eq:elliptic-estimate}, we have 
\algns{ 
  \nor{\Grad \Pis \ub^* }_0 &\le \nor{\Grad (\Pis \ub^* - \Pr_1 \ub^*)}_0 + \nor{\Grad \Pr_1 \ub^*}_0 \\
  &\lesssim h^{-1} \nor{\Pis \ub^* - \Pr_1 \ub^*}_0 + \nor{\ub^*}_1 \\
  &\le h^{-1} (\nor{\eeuis}_0 + \nor{\ub^* - \Pr_1 \ub^*}_0 ) + \nor{\ub^*}_1 \\
  &\lesssim \nor{\bs{\theta},\bs{\sigma}}_0 ,
}
and we also have
\algns{ 
  \nor{\Pis \ub^* }_0 
  \le \nor{\eeuis}_0 + \nor{\ub^*}_0 
  \lesssim \nor{\bs{\theta},\bs{\sigma}}_0,
}
thus 
\algn{ 
  \label{eq:I5-estm}
  \snor{\kappa I_4}  
  &\lesssim \kappa \nor{\db}_\Woneinfty \nor{\eebi}_0 \LRp{ \nor{\Pis \ub^*}_0 
  + h\nor{\Grad \Pis \ub^*}_0 } \\
  &\lesssim \kappa \nor{\db}_\Woneinfty \nor{\eebi}_0 \nor{\bs{\theta},\bs{\sigma}}_0. \notag
}

\noindent {\bf Estimation for $I_5$:} 
By an argument similar to the estimate of $\nor{\Grad \Pis \ub^*}_0$ above, 
$\nor{\curl \Pis \bb^*}_0 \lesssim \nor{\bs{\theta},\bs{\sigma}}_0 $.
Since $I_5 = \LRp{ \eeui, (\db - \Pr_0 \db) \times (\curl \Pis \bb^* )}$, 
\algn{ 
  \label{eq:I6-estm}
  | \kappa I_5 | 
  &\lesssim h \kappa \nor{\db}_\Woneinfty \nor{\eeui}_0 \nor{\curl \Pis \bb^*}_0 \\ 
  &\lesssim h \kappa \nor{\db}_\Woneinfty \nor{\eeui}_0 \nor{\bs{\theta},\bs{\sigma}}_0. \notag
}

\noindent {\bf Estimation for  $I_6$ and $I_7$:} Integrating  $I_6$ by parts (see \eqnref{eq:int-by-parts}) we have
\algns{ 
  - \kappa I_6 
  = - \kappa \LRp{ \eebis, \curl (\eeuh \times \db) } 
  - \kappa \LRa{ \db \times (\n \times \eebis), \eeuh }.
}
Now we can write $-\kappa I_6 + I_7$ as
\algns{ 
  -\kappa I_6 + I_7 
  &= - \kappa \LRp{ \eebis, \curl (\eeuh \times \db) } 
  + \LRa{\eeuh, -m \euhis - \kappa \db \times (\n \times \ebthis) } .
}
For the first term, as in the estimate of $I_4$, it suffices to estimate
\algns{ 
  \LRp{\eebis , \eeuh \LRp{ \Div \db } - \LRp{\eeuh \cdot \Grad} \db } \; \text{and} \;
  \LRp{\eebis, \LRp{(\db - \Pr_0 \db) \cdot \Grad} \eeuh - (\db - \Pr_0 \db) \LRp{\Div \eeuh }}.
}
Invoking the H\"older's inequality and an inverse estimate we can bound the upper bounds of the first term as
\algns{ 
  \nor{\eebis}_0 \nor{\db}_\Woneinfty \nor{\eeuh}_0 
  \lesssim h  \nor{\db}_\Woneinfty \nor{\eeuh}_0  \nor{\bs{\theta}, \bs{\sigma}}_0 .
}
For the second term, we first observe that
\algns{ 
  &\LRa{\eeuh, -m \euhis - \kappa \db \times (\n \times \ebthis) } \\
  &\qquad = \LRa{\eeuh, - (\wb - \Pr_0 \wb) \cdot \n \euhis - \kappa (\db - \Pr_0 \db) \times (\n \times \ebthis) }.
}
By the H\"{o}lder inequality, 
\algns{ 
  \snor{ \LRa{\eeuh, -m \euhis - \kappa \db \times (\n \times \ebthis) } }
  \le \nor{\eeuh}_{\pOmegah} 
  \left( \nor{\wb - \Pr_0 \wb}_{L^\infty(\pOmegah)} \nor{\ub^* - \Pr_1 \ub^*}_{\pOmegah} \right. \\
  \qquad + \left. \kappa \nor{(\db - \Pr_0 \db)}_{L^\infty(\pOmegah)} \nor{\bb^* - \Pr_0 \bb^*}_{\pOmegah} \right),
}
where we used the fact that $\Pis \hat{\bb^*}$ and $\Pis \hat{\ub^*}$ are the best approximations on $\pOmegah$. 
By Lemma~\ref{lemma:cont-inv-trace} this can be estimated by 
$\LRp{ h \kappa \nor{\db}_\Woneinfty + h^2 \nor{\wb}_\Woneinfty } \nor{\eeuh}_{0} \nor{\bs{\theta}, \bs{\sigma}}_0$.
As a consequence, 
\algn{ 
  \label{eq:I7_9-estm}
  \snor{-\kappa I_6 + I_7}
  \lesssim  \LRp{ h \kappa \nor{\db}_\Woneinfty 
  + h^2 \nor{\wb}_\Woneinfty} \nor{\eeuh}_{0} \nor{\bs{\theta}, \bs{\sigma}}_0 .
}

\noindent {\bf Estimation for $I_3$, $I_{8}$, and $I_{9}$:} Integrating $I_3$ by parts (see \eqnref{eq:int-by-parts}) gives
\algns{ 
  \kappa I_3 = \kappa \LRp{\eeuis, \db \times \LRp{\curl \eebh} } + \kappa \LRa{\n \times \LRp{\eeuis \times \db }, \eebh }.
}
Some algebraic manipulations give 
\algns{ 
  \kappa I_3 + I_{8} + I_{9} &= \kappa \LRp{\eeuis, \db \times \LRp{\curl \eebh} } + \LRa{\eebth, \kappa \n \times (\euhis \times \db) } \\
  &\quad +  \LRa{ \ebthh - \eebth, \half \kappa \n \times ( ( \euhis - \eeuis  ) \times \db ) - \n \times \ehis  - \alpha_2 \veps^I_{(\bb^*)^t}  } .
}
The first term is easily estimated by 
\algns{ 
  |\LRp{\eeuis, \db \times \LRp{\curl \eebh} }| 
  &= |\LRp{\eeuis, (\db -\Pr_0 \db) \times \LRp{\curl \eebh} }| \\
  &\lesssim h \nor{\db}_\Woneinfty \nor{\eebh}_0 \nor{\bs{\theta}, \bs{\sigma}}_0 .
}
For the second term we have
\algns{ 
  \left| \LRa{\eebth, \kappa \n \times (\euhis \times \db) } \right| 
  &\le \kappa \nor{\eebth}_{\pOmegah} \nor{\ub^* - \Pr_k \ub^*}_{\pOmegah} \nor{\db - \Pr_0 \db}_{L^\infty(\pOmegah)} \\
  &\lesssim h^2 \kappa \nor{\db}_\Woneinfty \nor{\eebh}_{0} \nor{\bs{\theta}, \bs{\sigma}}_0,
}
where we used Lemma~\ref{lemma:cont-inv-trace} and the discrete trace inequality.
Using the Cauchy--Schwarz inequality, \eqref{eq:elliptic-estimate}, and Lemma~\ref{lemma:cont-inv-trace}, the third term is bounded by 
\algns{ 
  h^\half \LRp{\kappa \nor{\db}_\Linfty  + \alpha_2 + 1} \nor{\ebthh - \eebth}_{\pOmegah} \nor{\bs{\sigma}, \bs{\theta}}_0.
}
Combining the above estimates we conclude 
\begin{multline}
  \label{eq:I4_10_11-estm}
  | \kappa I_3 + I_{8} + I_{9} | 
  \lesssim \left( \vphantom{\nor{\ebthh - \eebth}_{\pOmegah}}
  h \kappa \nor{\db}_\Woneinfty \nor{\eebh}_0 \right.  \\
  \left. + h^\half \LRp{\kappa \nor{\db}_\Linfty  
+ \alpha_2 + 1} \nor{\ebthh - \eebth}_{\pOmegah} \right) \nor{\bs{\sigma}, \bs{\theta}}_0 .
\end{multline}

\noindent {\bf Estimation for $I_{10}$:} Using the approximation capability of the projector $\bs{\Pi}$ (see Appendix \secref{projections}) we have
\algn{ 
  \label{eq:I15-estm}
  &\snor{\LRa{ \n \times \ehi - \half \kappa \n \times (( \eeui + \euhi) \times \db ) + \alpha_2 \eebti , \Pis \bb^*  - \Pr_e \bb^* }} \\
  & \quad \lesssim h^\half \LRp{ \nor{\ehi}_{\pOmegah} + \kappa \nor{\db}_\Linfty \nor{\eeui}_{\pOmegah} 
  + \alpha_2 \nor{\eebi}_{\pOmegah} }  \nor{\bs{\theta} , \bs{\sigma}}_0 . \hspace{-0.6em} \notag
}

At this point we are ready to estimate  the discretization errors for $\Lb$, $\Hb$, $\ub$, and $\bb$.
For readability let us absorb  $\alpha_1$, $\alpha_2$, $\alpha_3$, $\Rey$, $\Rm$, $\kappa$, 
and the norms on $\db$ and $\wb$ into the implicit constants.

\begin{theorem} 
  \theolab{adj-estm}
  Suppose that $\alpha_1 - \half \nor{\wb}_\Linfty$,
  $\alpha_2$, and $\alpha_3$ are chosen to be positive constants independent of $h$, $\Rey$, $\Rm$,
  $\kappa$, $\db$, and $\wb$.
  Suppose also that $h$ is sufficiently small, i.e., 
  \algn{
    \label{eq:smallness_assumption}
    h \le C \ll 1 
  }
  with $C$ (depending on the coefficients in estimates \eqref{eq:I7_9-estm} and \eqref{eq:I4_10_11-estm}).
  Then it holds that 
  \algn{ \label{eq:Eh-estm}
  E_h \lesssim h^{k+\half} \nor{ \Lb, \Hb, \ub, \bb, \r, \p}_{k+1} 
}
and the following error estimates hold:
\algn{
  \label{eq:LJ-estm}  \nor{ \Lb - \LbH , \Hb - \HbH }_0 &\lesssim h^{k+\half} \nor{ \Lb, \Hb, \ub, \bb, \r, \p}_{k+1}, \\
  \label{eq:ub-estm}  \nor{\bb - \bbH , \ub - \ubH }_0 &\lesssim h^{k+1} \nor{ \Lb, \Hb, \ub, \bb, \r, \p}_{k+1}.
}
\end{theorem}

\begin{proof}
  We proceed by taking $\bs{\theta} = \eeuh$, $\bs{\sigma} = \eebh$ in \eqref{eq:adjoint-mhd}. If we use \eqref{eq:adj-id}, 
  the estimates \eqref{eq:I1-estm}--\eqref{eq:I15-estm}, and  Young's inequality, we can obtain 
  \algn{
    \eqnlab{ub-estimate-inter}
    \nor{\eebh, \eeuh}_0 
    &\lesssim \nor{\eebi}_0 + h \nor{\eli, \elh, \eeui}_0 + h \nor{\eeuh, \eebh}_0 \\
    &\quad + h^\half \nor{\ehi, \eeui, \eebi }_{\pOmegah} + h^\half \nor{\ebthh -\eebth}_{\pOmegah}, \notag
  }
  which can be simplified to become
  \algn{
    \eqnlab{ub-estimate-inter2}
    \nor{\eebh, \eeuh}_0 
    &\lesssim \nor{\eebi}_0 + h \nor{\eli, \elh, \eeui }_0  \\
    &\quad + h^\half \nor{\ehi, \eeui, \eebi }_{\pOmegah} + h^{\half} \nor{\ebthh -\eebth}_{\pOmegah}, \notag
  }
  if the constants that multiply $\nor{\eeuh}_0$ and $\nor{\eebh}_0$ in \eqref{eq:I7_9-estm} and \eqref{eq:I4_10_11-estm} are sufficiently small,
  which is true under the assumptions we have made on $\alpha_1$, $\alpha_2$, and $\alpha_3$ together with \eqref{eq:smallness_assumption}.
  The discretization error terms on the right hand side of \eqnref{ub-estimate-inter2} (i.e., terms with superscript ``$h$'')  
  are bounded by $E_h$ (see definition of $E_h$ in \eqref{eq:energy-est}). This implies
  \algn{ 
    \nor{\eebh,\eeuh}_0
    \lesssim \nor{\eebi}_0 + h \nor{\eli, \eeui }_0 + h^\half \nor{\ehi, \eeui, \eebi }_{\pOmegah} + h^{\half} E_h.
    \eqnlab{ub-estimate-inter1}
  }
  Applying Young's inequality to the right side of \eqnref{energy-est1} for $\nor{\eeuh}_0$, $\nor{\eebh}_0$, and using \eqnref{ub-estimate-inter1}, 
  we get 
  \[
    E_h^2 \lesssim \nor{\ehi, \eeui, \eebi}_{\pOmegah}^2 + \nor{\eli, \eeui, \eebi}^2_0.
  \]
  Then \eqref{eq:Eh-estm} follows from the approximation properties of $\ehi, \eeui, \eebi, \eli$ with the trace inequality Lemma~\ref{lemma:cont-inv-trace}  discussed in Appendix \secref{projections} and \secref{auxiliary}, respectively. Further, it gives 
  \[
    \nor{ \elh, \ehh}_0 + \nor{\eeuh-\euhh,\eebth-\ebthh,\eerh-\erhh}_\pOmegah \lesssim E_h \lesssim h^{k+\half} \nor{ \Lb, \Hb, \ub, \bb, \r, \p}_{k+1},
  \]
  where the first inequality is from the definition of $E_h$ in \eqnref{disc-energy}. Then \eqref{eq:LJ-estm} follows from the triangle inequality. 
  Finally, the above estimate with \eqnref{ub-estimate-inter1} and the triangle ineqality give \eqref{eq:ub-estm}.
\end{proof}

What remains is to estimate $\nor{\eph}_0$ and $\nor{\eerh}_0$, and to
that end we extend the argument in \cite{CesmeliogluCockburnNguyenEtAl13} for our MHD system.
\begin{theorem} 
  \label{thm:eph-estm}
  There holds:
  \algns{ 
    \nor{\eph}_0 &\lesssim \nor{ \elh, \eeuh, \eebh, \eebi }_0 + h^{\half} E_h \lesssim h^{k+\half} \nor{ \Lb, \Hb, \ub, \bb, \r, \p}_{k+1}, \quad k \ge 0 .
  }
\end{theorem}
\begin{proof}
  We consider a projection operator $\Pit : \LRs{H^1(\Omega)}^d \ra \LRs{\Poly_k(\Omegah)}^d$ defined by 
  \algns{ 
    \LRp{ \Pit \vel - \vel, \vb}_\K &= 0, \qquad \vb \in \LRs{\Poly_{k-1}(\K)}^d, \\
    \LRa{ (\Pit \vel - \vel) \cdot \n, \mub \cdot \n }_{\pd \K} &= 0, \qquad \mub \in \LRs{\Poly_k^{\perp}(\K)}^d,
  }
  for $\vel \in \LRs{H^1(\Omega)}^d$ and $\K \in \Omegah$, where $\Poly_k^{\perp}(\K)$ is the subspace of 
  $\Poly_k(\K)$ which is orthogonal to $\Poly_{k-1}(\K)$ in $L^2(\K)$. Its well-posedness is based on 
  the orthogonal decomposition (see \cite[Lemma~4.1]{CockburnSayas-DivConfStokes})
  \algns{
    \Poly_k(\pK) = \{ \vb \cdot \n|_\pK \,:\, \vb \in \LRs{\Poly_k^{\perp}(\K)}^d \} \oplus \{ \q |_\pK \,:\, \q \in \Poly_k^{\perp} (\K) \},
  }
  and it has optimal approximation property by the Bramble--Hilbert lemma. 

  Since $\Pi \p$ is the $L^2$-projection, from \eqref{eq:global_5_1} we have
  \algns{ 
    \LRp{\ep^h,1} = -(\ep^I,1) = 0,
  }
  that is, $\ep^h$ has zero mean. It is known \cite{GiraultRaviart89} that there exists a function 
  $\vel \in \LRs{H_0^1(\Omega)}^d$ that $\Div \vel = \eph$ and $\nor{\vel}_1 \lesssim \nor{\eph}_0$. 
  For such a $\vel$, we have
  \begin{align} 
    \eqnlab{eph}
    \nor{\eph}_0^2 &= \LRp{\eph, \Div \vel} = - \LRp{ \Grad \eph, \vel} + \LRa{ \eph, \vel \cdot \n} \\
    \notag &= - \LRp{ \Grad \eph, \Pit \vel } + \LRa{ \eph, \vel \cdot \n} \\ 
    \notag &= \LRp{\eph, \Div \Pit \vel} + \LRa{ \eph \n, \vel - \Pit \vel } 
  \end{align}
  where we have performed integration by parts twice and used the definition of the projector $\Pit$.
  Since the exact solution $\LRp{\Lb, \p, \ub, \bb}$ and its trace also satisfy 
  the HDG local (sub-) equation \eqref{eq:local_2_1}, we can add and subtract the corresponding projections in \eqnref{projections-tan} to obtain
  \algn{ 
    \eqnlab{eph-local}
    &\LRp{ \eph, \Div \Pit \vel} 
    = \LRp{ \elh, \Grad \Pit \vel} - \LRp{ \eeuh, (\wb \cdot \Grad) \Pit \vel } 
    + \kappa \LRp{ \eebh, \Curl \LRp{ \Pit \vel \times \db } } \notag\\
    &\quad  + \LRa{ m \eeuh - \elh \n + \eph \n + \half \kappa \db \times  \LRp{ \n \times \LRp{ \eebth + \ebthh } } + \alpha_1 \LRp{ \eeuh - \euhh }, \Pit \vel } \hspace{-0.5em} \\
    & \quad + \kappa \LRp{ \eebi, \Curl \LRp{ \Pit \vel \times \db } }, \notag
  }
  where we have taken $\Pit\vel$ as the test function in \eqnref{local_2_1}.

  Combining \eqnref{eph}  and \eqnref{eph-local} yields
  \algns{  
    \nor{ \eph }_0^2 &=  \LRp{ \elh, \Grad \Pit \vel } - \LRp{ \eeuh, (\wb \cdot \Grad) \Pit \vel } 
    + \kappa \LRp{ \eebh, \Curl \LRp{ \Pit \vel \times \db } } \\
    & \quad + \LRa{ m \eeuh - \elh \n + \half \kappa \db \times  \LRp{ \n \times \LRp{ \eebth + \ebthh } } + \alpha_1 \LRp{ \eeuh - \euhh }, \Pit \vel } \\
    & \quad + \kappa \LRp{ \eebi, \Curl \LRp{ \Pit \vel \times \db } } + \LRa{ \eph \n, \vel },
  }
  which can be further simplified using two facts: first, integrating by parts twice and using the definition of $\Pit$ give
  \algns{ 
    \LRp{ \elh, \Grad \Pit \vel } &= \LRp{ \elh, \Grad \vel } - \LRa{ \elh \n, \vel } + \LRa{ \elh \n, \Pit \vel };
  }
  and second, we combine the first equation in \eqref{global} and \eqnref{Lu-projection3} to have
  \algns{ 
    \LRa{ - \elh \n + \eph \n, \vel} &= \LRa{ - \elh \n + \eph \n, \Pr_e \vel} \\
    &= - \LRa{ m \eeuh + \half \kappa \db \times  \LRp{ \n \times \LRp{ \eebth + \ebthh } } + \alpha_1 \LRp{ \eeuh - \euhh }, \Pr_e \vel}.
  }
  In particular, we obtain 
  \algns{  
    \nor{ \eph }_0^2 &=  \LRp{ \elh, \Grad \vel } - \LRp{ \eeuh, (\wb \cdot \Grad) \Pit \vel } 
    + \kappa \LRp{ \eebh, \Curl \LRp{ \Pit \vel \times \db } } \\
    & \quad + \LRa{ m \eeuh + \half \kappa \db \times  \LRp{ \n \times \LRp{ \eebth + \ebthh } } + \alpha_1 \LRp{ \eeuh - \euhh }, \Pit \vel - \Pr_e \vel} \\
    & \quad + \kappa \LRp{ \eebi, \Curl \LRp{ \Pit \vel \times \db } } .
  }
  By the triangle and H\"{o}lder inequalities,
  \algns{
    &\nor{ \eph }_0^2
    \lesssim \nor{\elh}_0 \nor{\Grad\vel}_0
    + \nor{\wb}_\Linfty \nor{ \eeuh }_0 \nor{\Grad \Pit \vel}_0
    + \kappa \nor{\db}_\Woneinfty \nor{\eebh,\eebi}_0 \nor{\Grad \Pit \vel}_0 \\
    & +
    \left( \hspace{-0.2em} \nor{\wb}_\Linfty \nor{ \eeuh }_{\pOmegah}
    \hspace{-0.6em} + \kappa \nor{\db}_\Linfty \nor{ \eebth, \eebth - \ebthh}_{\pOmegah} 
    \hspace{-0.8em} + \alpha_1 \nor{ \eeuh - \euhh }_{\pOmegah} \hspace{-0.2em} \right)
    \nor{\Pit \vel - \Pr_e \vel}_{\pOmegah} \\
    &\lesssim \LRp{ \nor{ \elh, \eeuh, \eebh, \eebi }_0 + h^{\half} E_h } \nor{\eph}_0
  }
  where we have used Lemma \lemref{gradProjectionBound}, the
  approximation capability of $\Pit$ and the $L^2$-projection, definition of $E_h$ in \eqnref{disc-energy},  the property of $\vel$,
  and we absorb all mesh independent parameters into the implicit constant in the final inequality.
  As a consequence, we have $\nor{ \eph }_0 \lesssim \nor{ \elh, \eeuh, \eebh, \eebi }_0 + h^{\half} E_h$. Then the conclusion 
  follows from the triangle inequality and the estimates of $\nor{ \elh, \eeuh, \eebh, \eebi }_0$ and $E_h$.
\end{proof}
For an analogous result for $\nor{\eerh}_0$, we make use of the following Lemma.
\begin{lemma} \label{lemma-for-r}
  There exists some $\LRs{\vel \in H^1(\Omega)}^d$ such that $\Div \vel = \eerh$, $\n \times \vel = \bs{0}$ on $\pOmega$,
  and $\nor{\vel}_1 \lesssim \nor{\eerh}_0$.
\end{lemma}
\begin{proof}
  Let $\overline{\eerh}$ be the mean-value of $\eerh$ in $\Omega$ and $1_{\Omega}$ be the indicator function on $\Omega$. It is well known that there exists $\vel_0$ in $\LRs{H_0^1(\Omega)}^d$ such that $\Div \vel_0 = \eerh - \overline{\eerh} 1_{\Omega}$, $\nor{\vel_0}_1 \lesssim \| \eerh - \overline{\eerh} 1_{\Omega} \|_0$. Since $\Omega$ is a bounded domain with polyhedral boundary, $\pd \Omega$ consists of finitely many piecewise smooth $(d-1)$-dimensional components, so there exists a nonnegative smooth function $\phi_0 \not = 0$ on $\pd \Omega$ supported on the interior of the smooth components. Let $\phi$ be the renormalized function of $\phi_0$ satisfying $\int_{\pd \Omega} \phi = \overline{\eerh}|\Omega|$. Then $\phi \n \in [H^{1/2}(\pd \Omega)]^d$ holds because $\phi$ is a smooth function on $\pd \Omega$, and there is $\Phi \in [H^1(\Omega)]^d$ such that $\Phi|_{\pd \Omega} = \phi \n$ and (cf. \cite[Theorem II.4.3]{galdi2011introduction})
  \algns{
    \nor{\Phi}_1 \lesssim \nor{\phi \n}_{H^{1/2}(\pd \Omega)} \lesssim \nor{\phi_0 \n}_{H^{1/2}(\pd \Omega)} \|{\overline{\eerh}} 1_{\Omega} \|_0 .
  }  
  According to a result in \cite[p.176]{galdi2011introduction}, there exists $\vel_1$ in $\LRs{H^1(\Omega)}^d$ such that $\Div \vel_1 = \overline{\eerh} 1_{\Omega}$, $\vel_1 = \phi \n$ on $\pd \Omega$, and also satisfies
  \algns{
    \nor{\vel_1}_1 &\lesssim \| \overline{\eerh} 1_{\Omega}\|_0 + \| \Div \Phi \|_0 \lesssim \| \overline{\eerh} 1_{\Omega}\|_0 + \| \Phi \|_1 \lesssim \| \overline{\eerh} 1_{\Omega}\|_0 .
  }
  The conclusion follows by taking $\vel = \vel_0 + \vel_1$.
\end{proof}
\begin{theorem}
  There holds:
  \algns{ 
    \nor{\r - \rH}_0 \lesssim  \nor{\ehh , \eeui, \eeuh}_0 + h^{\half} E_h \lesssim h^{k+\half} \nor{ \Lb, \Hb, \ub, \bb, \r, \p}_{k+1}, \quad k \ge 0 .
  }
\end{theorem}
\begin{proof}
  First choose a function $\vel$ that satisfies the statements of Lemma~\ref{lemma-for-r},
  and consider 
  \algns{ 
    \nor{\eerh}_0^2 = \LRp{\eerh, \Div \vel} = \LRp{ \eerh, \Div \Pit \vel} + \LRa{ \eerh \n , \vel - \Pit \vel } .
  }
  From \eqref{eq:local_5_1_error2} with $\cb = \Pit\vel$, we have 
  \algns{ 
    \LRp{ \eerh, \Div \Pit \vel} &= \LRp{ \ehh, \Curl \Pit \vel } - \kappa \LRp{ \eeuh, \db \times \LRp{ \Curl \Pit \vel } } \\
    &\quad + \LRa{ \n \times \veps_{\Hb} + \veps_{\rh} \n - \half \kappa \n \times \LRp{ \LRp{ \veps_{\ub} + \veps_{\ubh} } \times \db } + \alpha_3 \LRp{ \veps_{\bb^t} - \veps_{\bbh^t} }, \Pit \vel } \\
    &\quad - \kappa \LRp{ \eeui, \db \times \LRp{ \Curl \Pit \vel } } .
  }
  If we use this to the previous identity, and also use the flux condition \eqref{second-conservation-error-equation} 
  with $\Pr_e \vel$ as the test function, we have 
  \algns{ 
    \nor{\eerh}_0^2 &= \LRp{ \ehh, \Curl \Pit \vel } - \kappa \LRp{ \eeuh, \db \times \LRp{ \Curl \Pit \vel } } \\
    &\quad + \LRa{ \n \times \veps_{\Hb} + \veps_{\rh} \n - \half \kappa \n \times \LRp{ \LRp{ \veps_{\ub} + \veps_{\ubh} } \times \db } + \alpha_3 \LRp{ \veps_{\bb^t} - \veps_{\bbh^t} }, \Pit \vel - \Pr_e \vel } \\
    &\quad - \kappa \LRp{ \eeui, \db \times \LRp{ \Curl \Pit \vel } } + \LRa{ \eerh \n , \vel - \Pit \vel } 
  }
  where the surface integrals on the domain boundary are zero due to the fact that $\n \times \vel = \bs{0}$ on $\pOmega$
  and \eqref{rhat-boundary-error-equation}.
  Note that 
  \algns{
    \LRa{ \veps_{\rh}^I \n, \Pit \vel - \Pr_e \vel } = 0 , \qquad \LRa{ \veps_{\rh}^h \n, \Pit \vel - \Pr_e \vel } = \LRa{ \veps_{\rh}^I \n, \Pit \vel - \vel } .
  }
  With these identities and \eqref{second-conservation-error-equation} we have 
  \algns{ 
    \nor{\eerh}_0^2 &= \LRp{ \ehh, \Curl \Pit \vel } - \kappa \LRp{ \eeuh, \db \times \LRp{ \Curl \Pit \vel } } \\
    &\quad + \LRa{ \n \times \veps_{\Hb}^h - \half \kappa \n \times \LRp{ \LRp{ \veps_{\ub}^h + \veps_{\ubh}^h } \times \db } + \alpha_3 ({ \veps_{\bb^t}^h - \veps_{\bbh^t}^h }), \Pit \vel - \Pr_e \vel } \\
    &\quad - \kappa \LRp{ \eeui, \db \times \LRp{ \Curl \Pit \vel } } + \LRa{ (\eerh - \erhh)\n , \vel - \Pit \vel } 
  }
  By the triangle, H\"{o}lder, and trace inequalities, 
  \algns{ 
    \nor{\eerh}_0^2 &\lesssim \nor{\ehh}_0 \nor{\Grad \Pit \vel}_0 + \kappa \nor{\eeuh}_0 \nor{\db}_{L^\infty} \nor{\Grad \Pit \vel}_0 + \nor{\eeui}_0 \nor{\vel}_1 \\
    &\quad + \LRp{ h^{-\half} \nor{\ehh, \eeuh}_0 + \nor{ \eeuh - \euhh, \eebth- \ebthh, \eerh - \erhh }_{0,\mc{E}_h} } h^{\half} \nor{\vel}_1 .
  }
  Since $\nor{\Grad \Pit \vel}_0 \lesssim \nor{\vel}_1$ and $\nor{\vel}_1 \lesssim \nor{\eerh}_0$, we have 
  \algns{ 
    \nor{\eerh}_0 \lesssim  \nor{\ehh , \eeui, \eeuh}_0 + h^{\half} E_h  .
  }
  The conclusion follows from the triangle inequality and the estimates of $\nor{\ehh , \eeui, \eeuh}_0$ and $E_h$. 
\end{proof}

\section{Numerical Results}
\seclab{numerics}

In this section we apply the proposed HDG scheme for 2D MHD problems. The
application for large-scale 3D problems is ongoing and will be
presented elsewhere.  The first problem we consider is the Hartmann
flow whose analytical solution exists and is one dimensional in
nature. The second problem is posed on a non-convex
domain to demonstrate the approximation capability of our proposed HDG
scheme though the non-convexity is not covered by our
theory.
To further challenge the proposed HDG approach,
we will consider a singular problem as  the third example.

Before presenting the results for each example, we make some general observations to differentiate the proposed HDG scheme
from the DG method of \cite{HoustonSchoetzauWei09}.
First, recalling the primary motivation of HDG schemes, we have a reduced global system to solve for,
which offers computational savings, especially for high-order solution of large-scale problems.
Second, the HDG scheme allows the convenience of equal polynomial order approximations of all unknowns
with direct approximations of $\Grad \ub$ and $\Curl \bb$. 
On the other hand, the DG method of \cite{HoustonSchoetzauWei09} employs polynomials of order $k-1$ and $k+1$
for $\pH$ and $\rH$, respectively, where $k$ is the order of approximation of $\ubH$ and $\bbH$, 
and $\Grad \ub$, $\Curl \bb$ are approximated by the derivatives of $\ubH$ and $\bbH$.
These differences make a direct comparison of the two methods difficult.
The DG scheme is proven to be optimal (converging with $\mc{O}({h^k})$) in the DG energy norms for sufficiently smooth solutions.
By the definitions of these norms in \cite{HoustonSchoetzauWei09}, $\ubH$, $\Grad \ubH$, $\bbH$, and $\Curl \bbH$
are proven to converge at worst with $\mc{O}(h^k)$ in the $L^2$ norm.
Optimal $L^2$ convergence of $\mc{O}(h^{k+1})$ for $\ubH$ and $\bbH$ is not proven in \cite{HoustonSchoetzauWei09}, 
but is demonstrated in all numerical examples with smooth solutions.
On the other hand, the HDG scheme is proven to be $L^2$ optimal ($\mc{O}({h^{k+1}})$) for $\ubH$ and $\bbH$ 
and quasi-optimal ($\mc{O}({h^{k+\half}})$) for the remaining local quantities.
These rates are demonstrated (and sometimes exceeded) in the numerical examples involving smooth solutions.
For the numerical example involving a singular solution, the HDG scheme gives similar $L^2$ error magnitudes and
convergence rates as in the DG scheme for $\ubH$ and $\bbH$ for the same polynomial order $k$.

We note that the numerical solutions $\ubH$ and $\bbH$ in the numerical results below satisfy the divergence-free constraint in the weak sense. However, the pointwise  divergence-free constraint can be fulfilled by 
a post-processing procedure
outlined in Appendix~\secref{appendix:postproc}, which shows that  convergence rates of the post-processed solutions 
to the exact solutions are as good as those of $\ubH$ and $\bbH$. For that reason, we do not use the post-processed solution in the error computation.

\subsection{Hartmann Flow}

In this numerical study, 
we consider a conducting incompressible fluid (liquid metal, for
example) in a domain $[-\infty,\infty] \times [-l_0,l_0] \times
[-\infty,\infty]$ (bounded by infinite parallel
plates in the $x_2$ direction \cite{HoustonSchoetzauWei09,Shadidetal2016_3DVMSMHD}).  The fluid is subject to a
uniform pressure gradient $G := -\frac{\partial p}{\partial x_1}$ in the
$x_1$ direction, and a uniform external magnetic field $b_0$ in the $x_2$
direction.  
If we consider no-slip boundary conditions on the $x_2$ boundaries, the resulting
flow pattern is known as \textit{Hartmann flow} which admits an
analytical solution 
which is one dimensional in nature.  
We assume that the infinite
parallel plates are perfectly insulating.
Here, we consider the Hartmann flow in a 2D domain  $\Omega = [0,0.025] \times [-1,1]$.
If we define the characteristic velocity as $u_0 := \sqrt{G l_0/\rho}$,
and consider the driving pressure gradient $G$ as a forcing term (incorporated in $\gb$),
the non-dimensionalized solution with
$\gb = (1,0)$, $\fb = (0,0)$ reads
\begin{align*}
  \ub &= \LRp{ \frac{\Rey}{\Ha \tanh(\Ha)} \LRs{1 - \frac{\cosh(\Ha \,x_2)}{\cosh(\Ha)}} , 0}, &
  p   &= - \frac{1}{2\kappa} \LRs{\frac{\sinh(\Ha \,x_2)}{\sinh(\Ha)} - x_2}^2 - p_0, \\
  \bb &= \LRp{ \frac{1}{\kappa} \LRs{\frac{\sinh(\Ha \,x_2)}{\sinh(\Ha)} - x_2} , 1}, &
  r   &= 0,
\end{align*}
where $\Ha:=\kappa\Rey\Rm$,
and $p_0$ is a constant that enables $p$ to satisfy the zero average pressure condition \eqref{eq:global_3_1}.
We set $\wb = \ub$ and $\db = \bb$, and we enforce the boundary conditions on $\pOmega$ using the exact solution, i.e.,  $\ub_D = \ub$, $\hb_D = \bbt$, and $\r_D = 0$. 

At refinement level $l$, the domain is divided into $l \times 80 l$ squares, each of which is divided into two triangles from top right to bottom left.
In Figure \ref{SJS:fig:Hartmann_tri_conv} are the convergence plots with $\Rey = \Rm = 7.07$ and $\kappa = 200$.
The convergence rates for $\LbH$, $\ubH$, $\pH$, $\HbH$, $\bbH$, and $\rH$ are observed to be approximately
$k+\half$, $k+1$, $k+\half$, $k+1$, $k+1$, and $k+1$, respectively.
These observed rates approximately match or exceed their respective theoretical rates of
$k+\half$, $k+1$, $k+\half$, $k+\half$, $k+1$, and $k+\half$
which were proven in Section~\secref{error_analysis}.

\begin{figure}
  \begin{center}
    \subfigure{\includegraphics[scale=0.34]{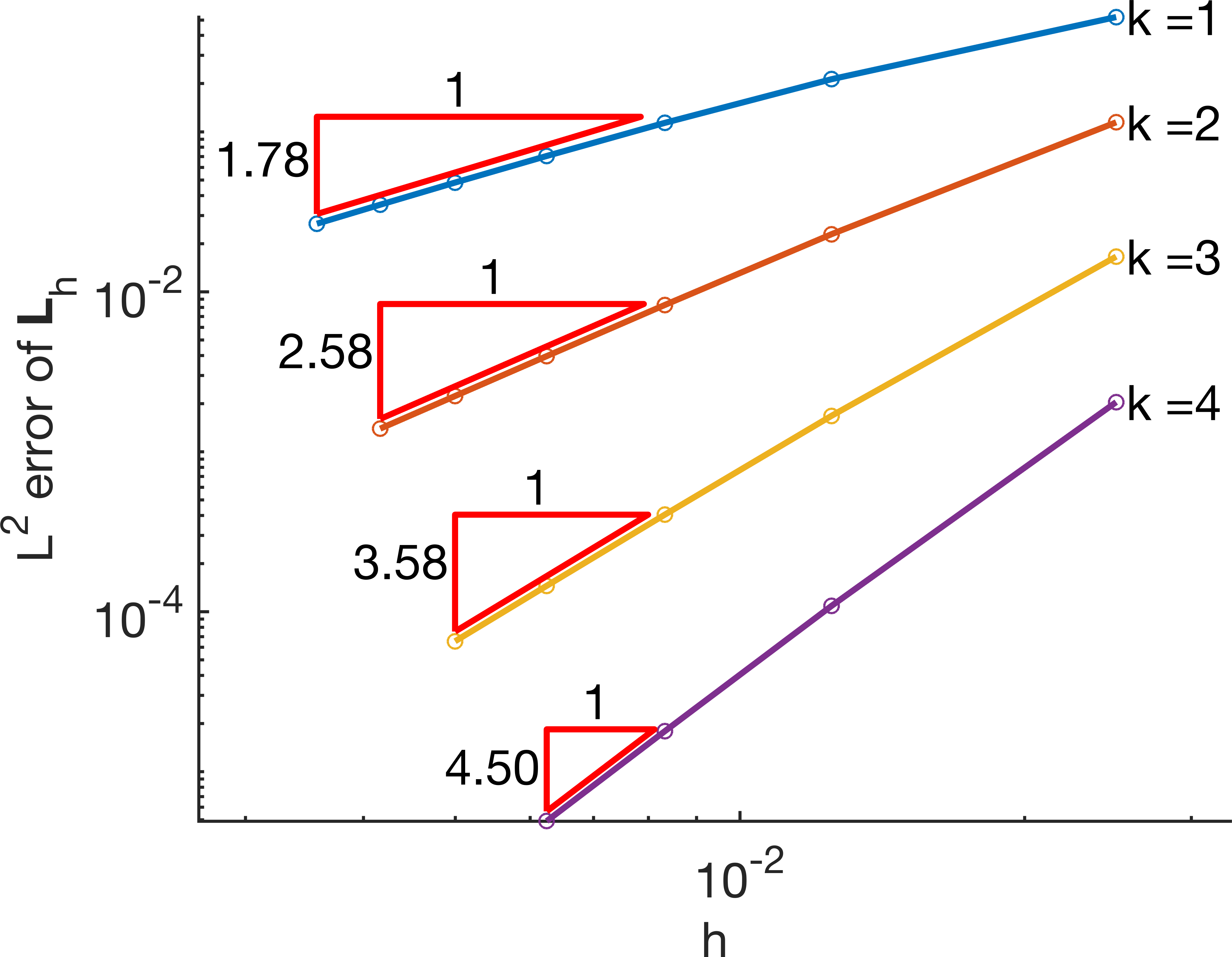}
    \label{SJS:fig:Hartmann_tri_L_conv}}
    \subfigure{\includegraphics[scale=0.34]{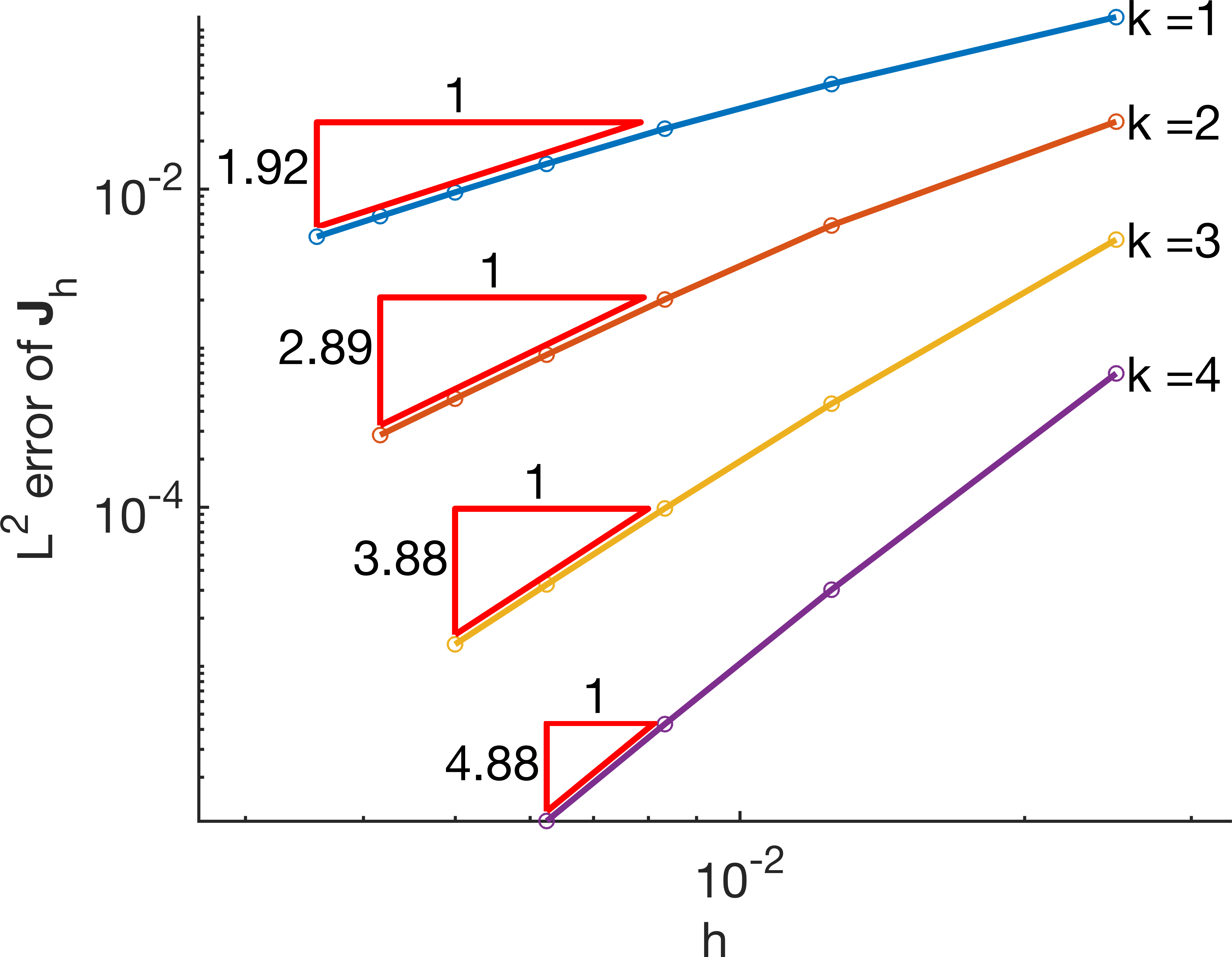}
    \label{SJS:fig:Hartmann_tri_J_conv}}
    \\ \vspace{-0.5em}
    \subfigure{\includegraphics[scale=0.34]{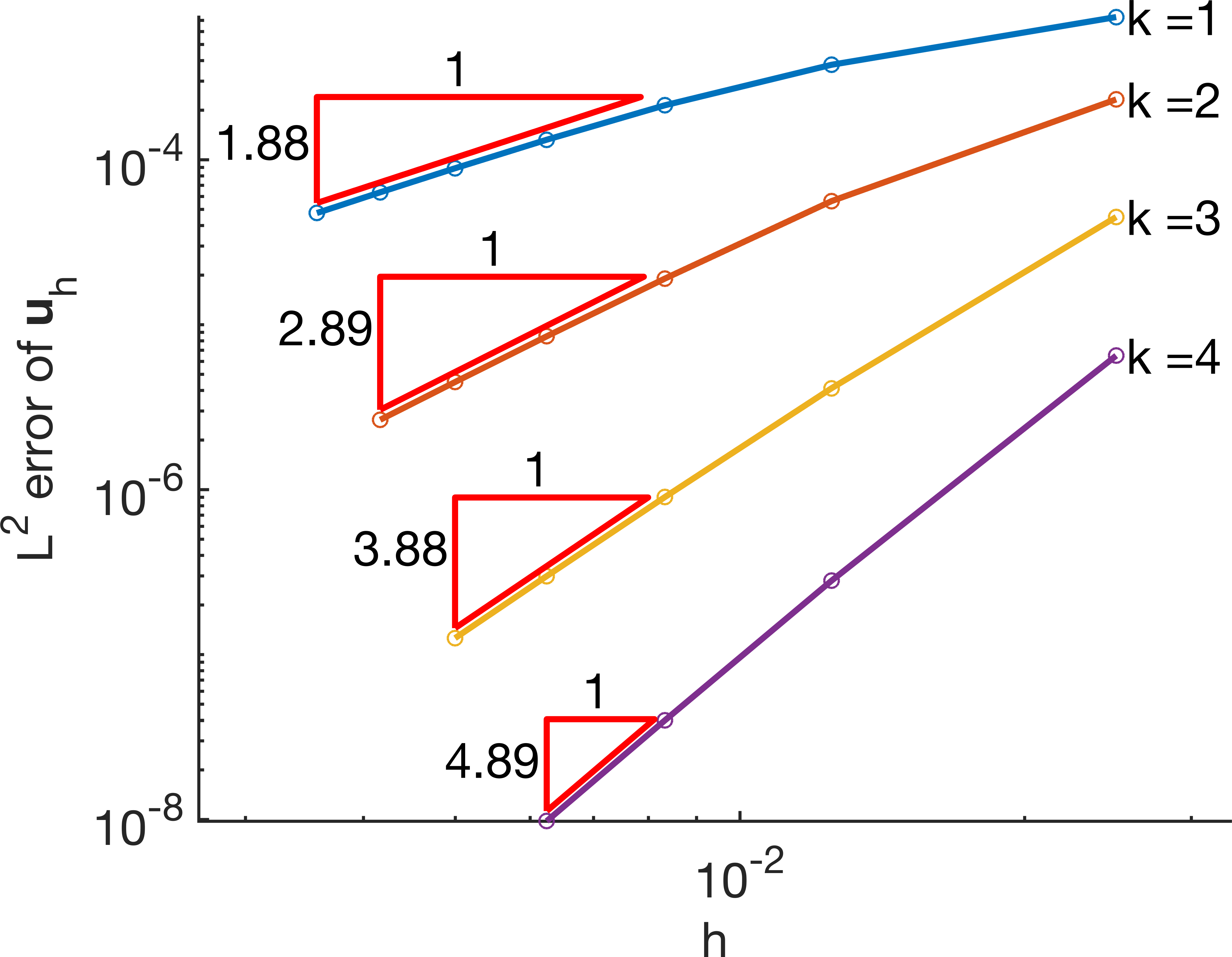}
    \label{SJS:fig:Hartmann_tri_u_conv}}
    \subfigure{\includegraphics[scale=0.34]{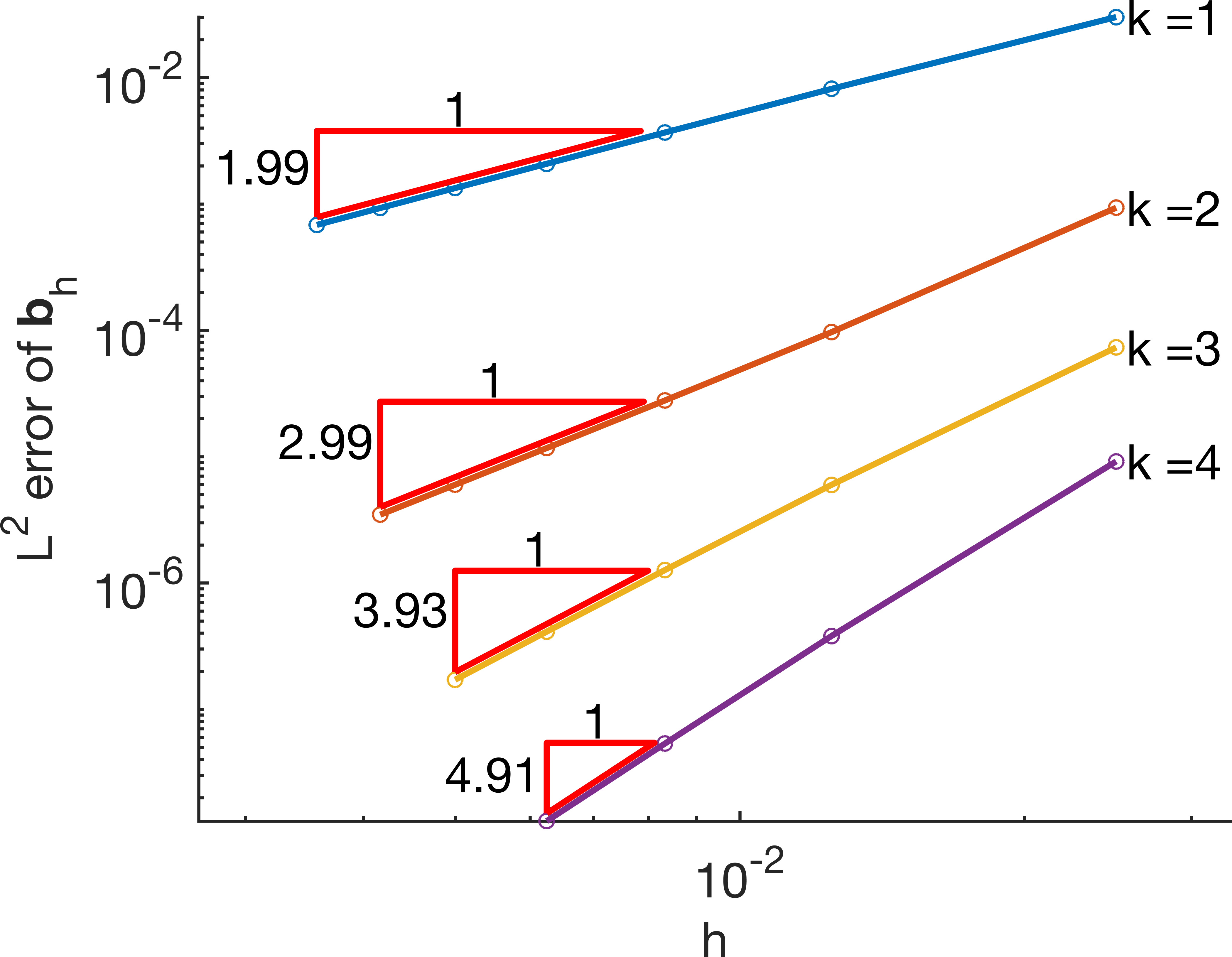}
    \label{SJS:fig:Hartmann_tri_b_conv}}
    \\ \vspace{-0.5em}
    \subfigure{\includegraphics[scale=0.34]{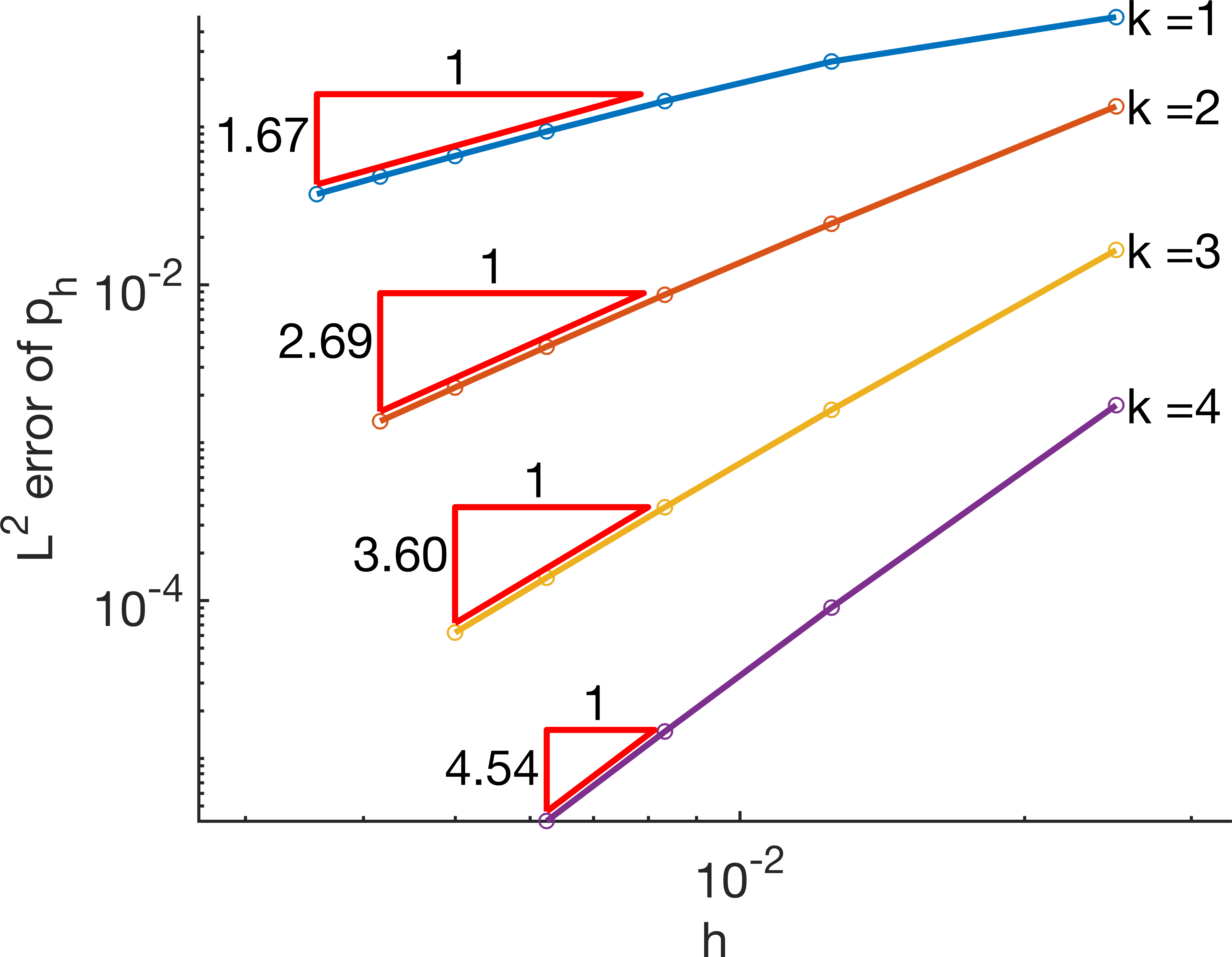}
    \label{SJS:fig:Hartmann_tri_p_conv}}
    \subfigure{\includegraphics[scale=0.34]{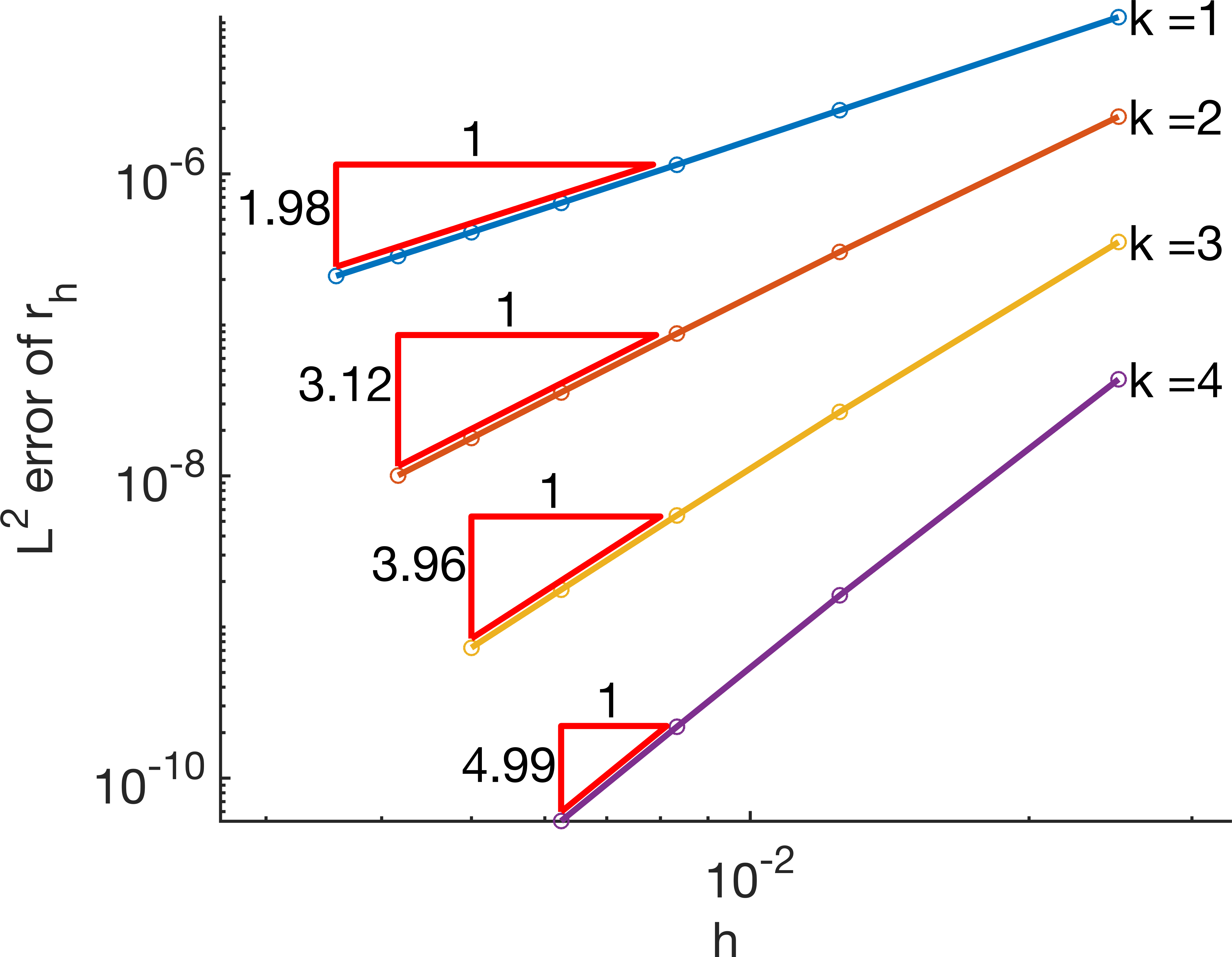}
    \label{SJS:fig:Hartmann_tri_r_conv}}
    \caption{Hartmann flow problem: $L^2$ convergence plots for $\LbH$, $\ubH$, $\pH$, $\HbH$, $\bbH$, and $\rH$.}
    \label{SJS:fig:Hartmann_tri_conv}
  \end{center}
\end{figure}

\subsection{Non-convex Domain}

This example illustrates the convergence of the HDG scheme applied
to a problem posed
on the non-convex domain
$\Omega = (-1,1) \times (-1,1) \setminus [0,1) \times (-1,0]$ in Figure \ref{SJS:fig:Lshaped_mesh} (similar to Section 5.1.1 in \cite{HoustonSchoetzauWei09}).
We take $\Rey = \Rm = \kappa = 1$, $\wb = (2,1)$, and $\db = (x_1,-x_2)$.
We set $\gb$ and $\fb$ such that  the manufactured solution for \eqnref{mhdlin} is the following
\begin{alignat*}{2}
  \ub \ && = \LRp{ -\LRs{x_2 \cos(x_2) + \sin(x_2)} e^{x_1} , x_2 \sin(x_2) e^{x_1} }, \quad
  p &= 2 e^{x_1} \sin(x_2) - p_0, \\
  \bb \ && = \LRp{ -\LRs{x_2 \cos(x_2) + \sin(x_2)} e^{x_1} , x_2 \sin(x_2) e^{x_1} }, \quad
  r &= -\sin(\pi x_1) \sin(\pi x_2),
\end{alignat*}
where $p_0$ is the constant that enables $p$ to satisfy the zero average pressure condition \eqref{eq:global_3_1}.
We use the exact solution to enforce the boundary conditions $\pOmega$, i.e., $\ub_D = \ub$, $\hb_D = \bbt$, and $\r_D = \r$. 

At refinement level $l$, each quadrant  of the domain (see Figure \ref{SJS:fig:Lshaped_mesh} for an example with $l=4$) is subdivided into $l \times l$ squares,
each of which is divided into two triangles from top right to bottom left.
In Figure \ref{SJS:fig:Lshaped_conv} are the convergence plots.  For
this problem, we observe the optimal convergence rates of $k+1$ for
all of the local variables, which matches or exceeds the rates proven
in Section~\secref{error_analysis}.  

\begin{figure}
  \begin{center}
    \includegraphics[scale=0.3]{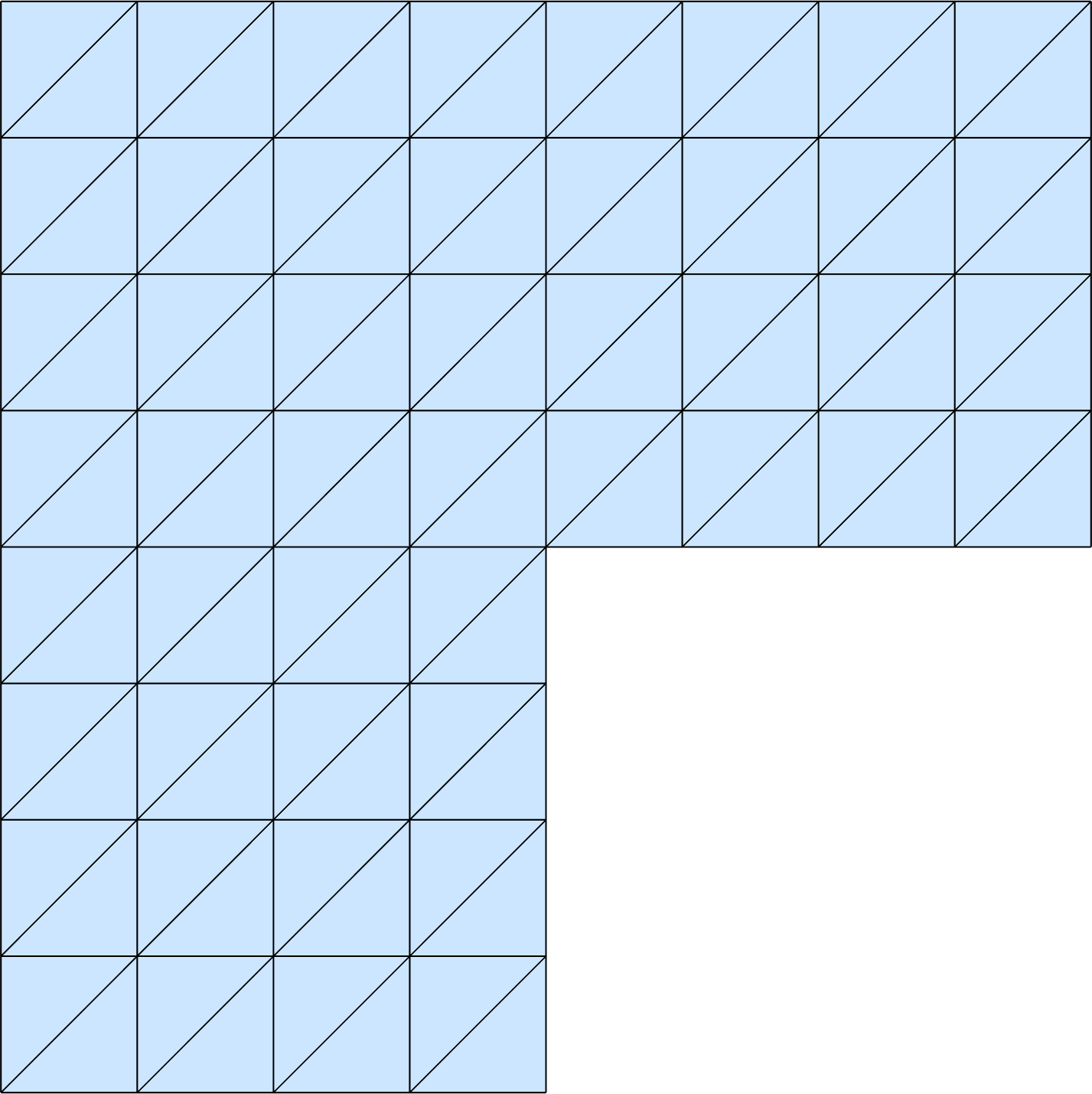}
    \caption{Geometry of the non-convex domain with a mesh at refinement level $l=4$.}
    \label{SJS:fig:Lshaped_mesh}
  \end{center}
\end{figure}

\begin{figure}
  \begin{center}
    \subfigure{\includegraphics[scale=0.34]{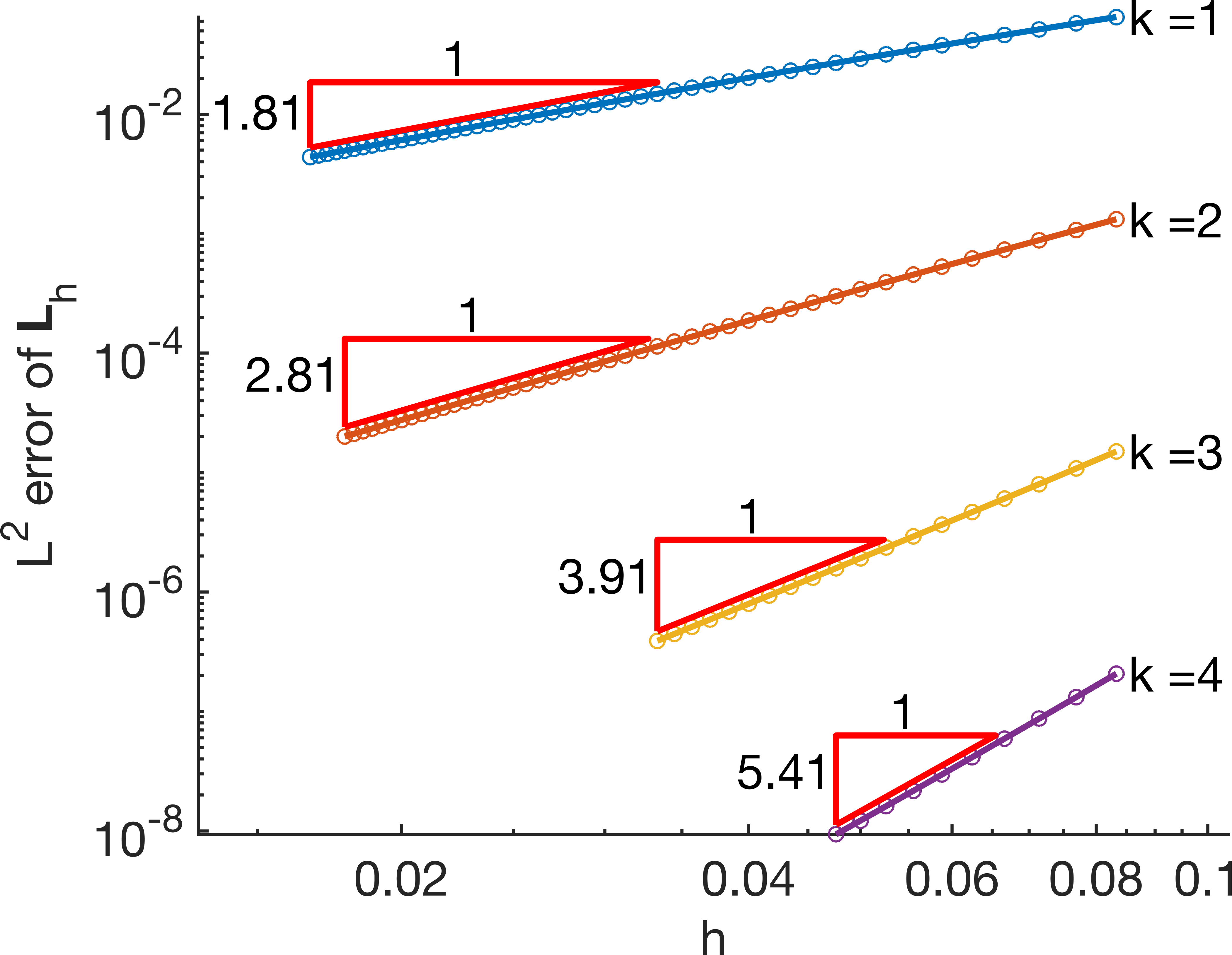}
    \label{SJS:fig:Lshaped_L_conv}}
    \subfigure{\includegraphics[scale=0.34]{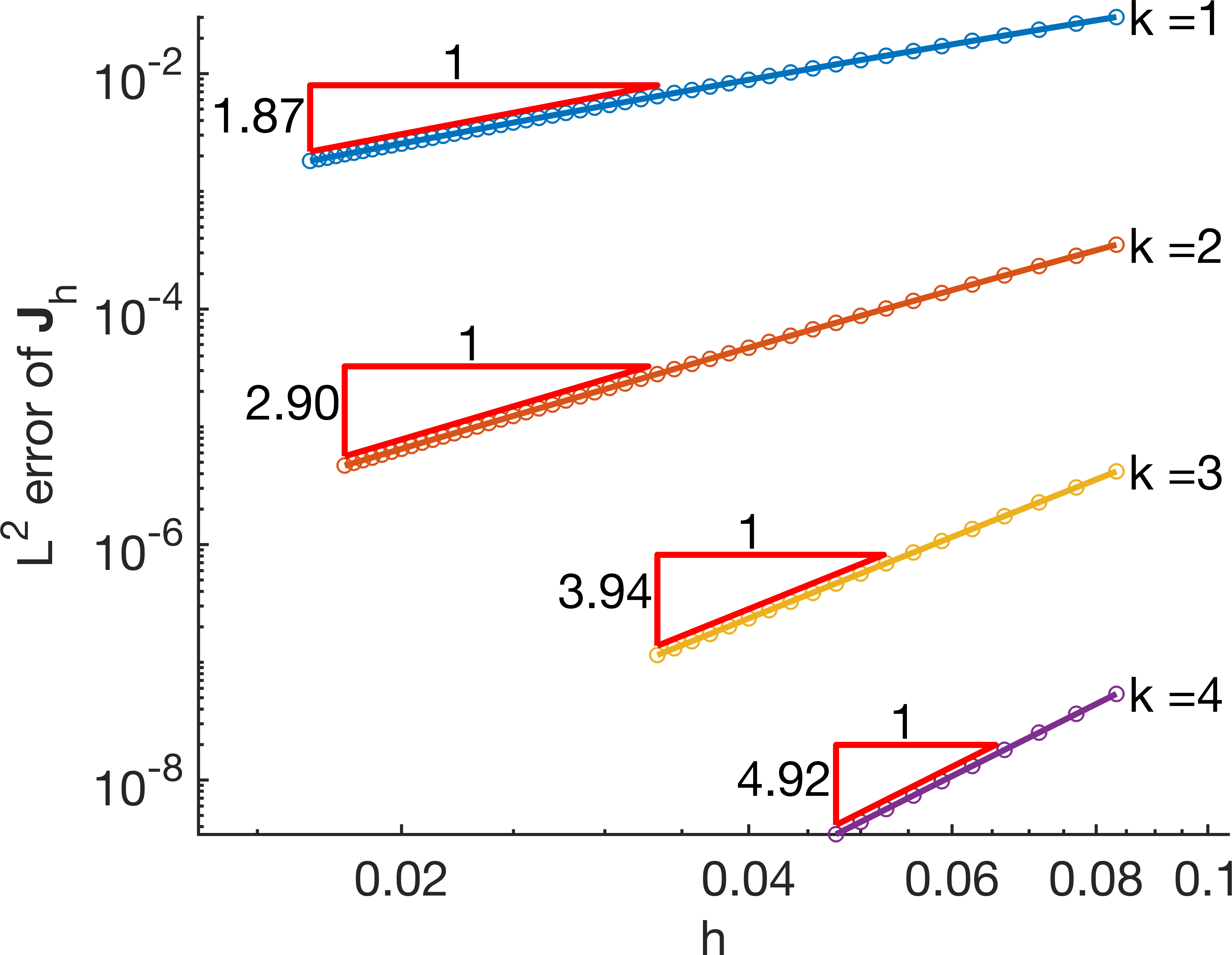}
    \label{SJS:fig:Lshaped_J_conv}}
    \\ \vspace{-0.5em}
    \subfigure{\includegraphics[scale=0.34]{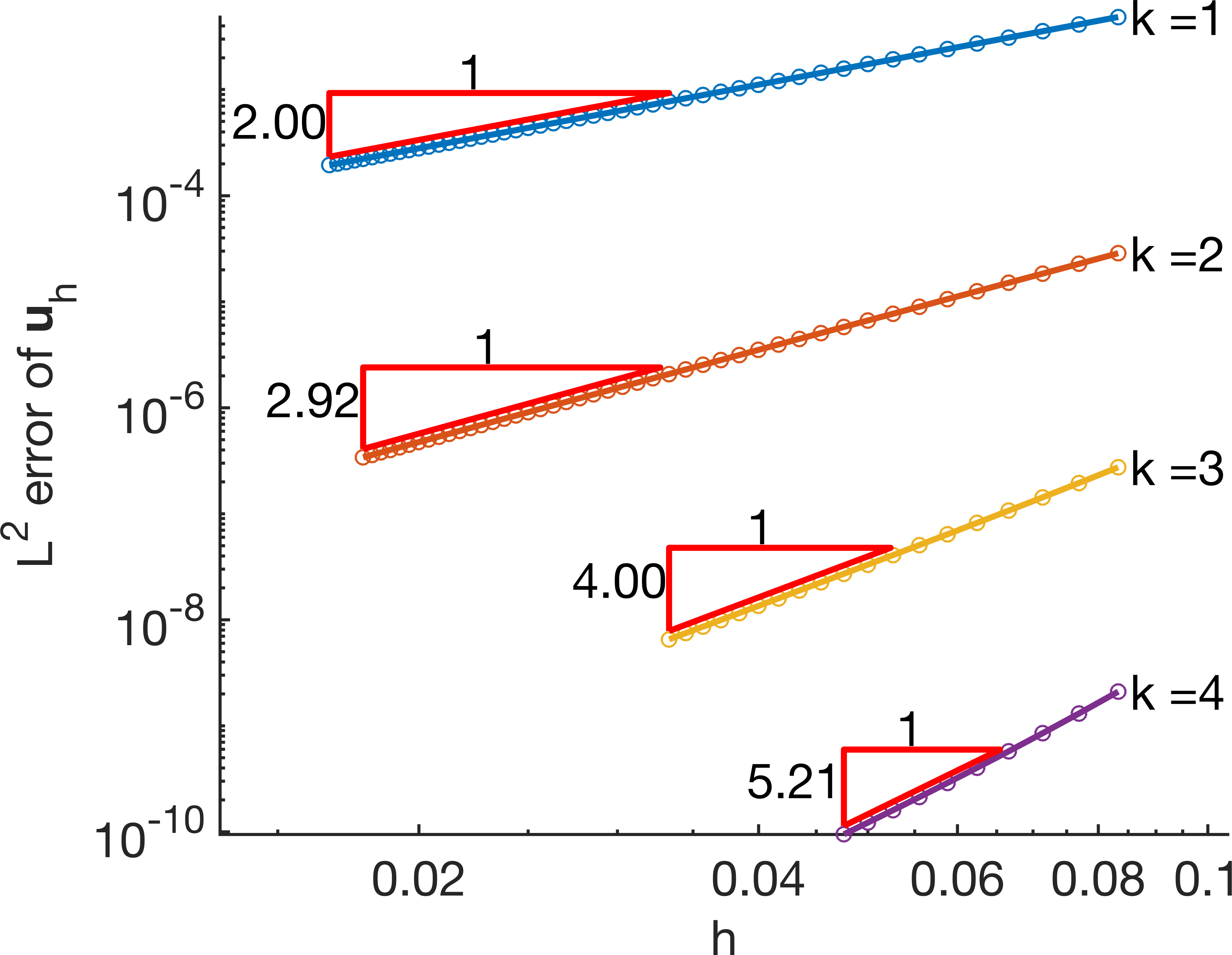}
    \label{SJS:fig:Lshaped_u_conv}}
    \subfigure{\includegraphics[scale=0.34]{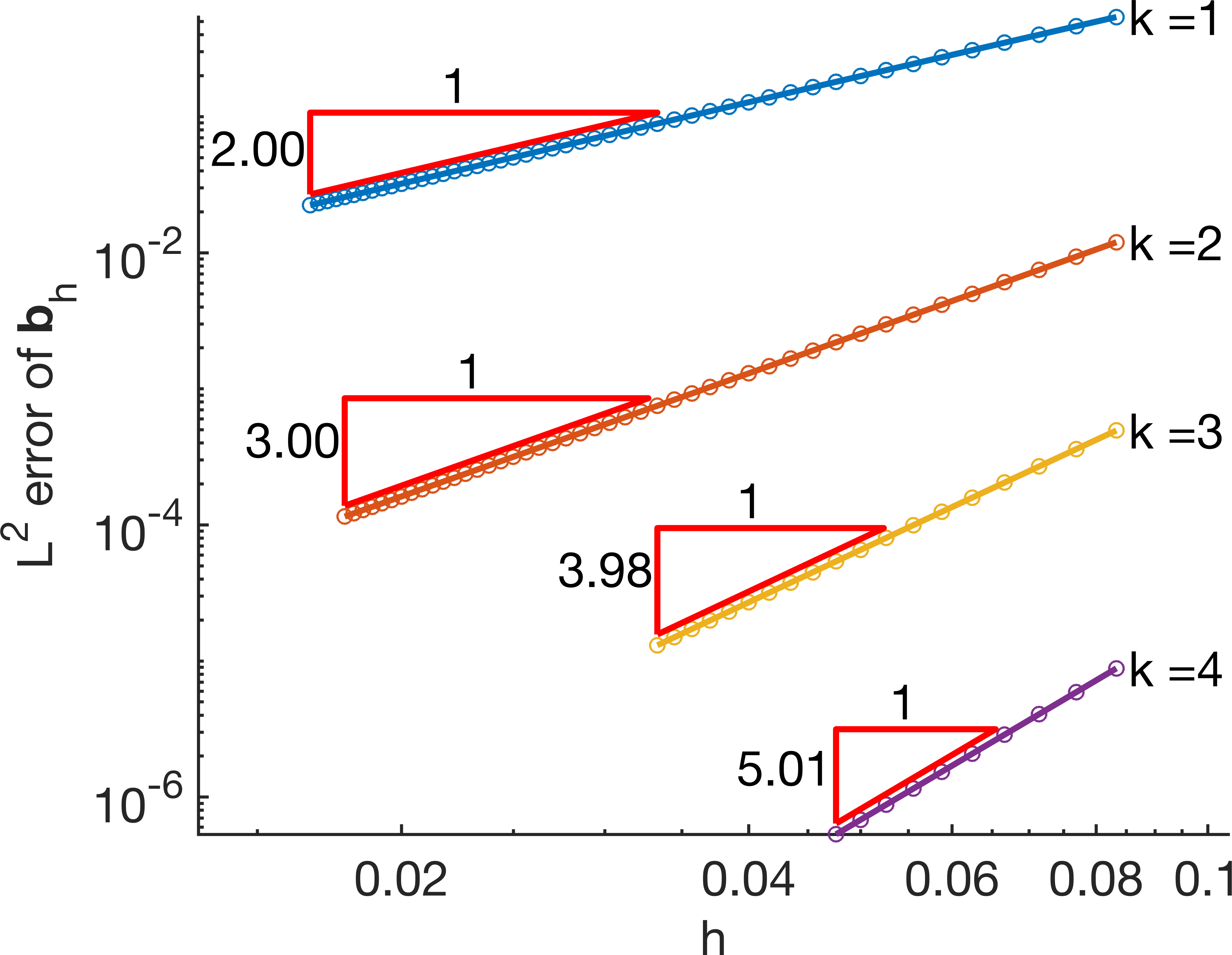}
    \label{SJS:fig:Lshaped_b_conv}}
    \\ \vspace{-0.5em}
    \subfigure{\includegraphics[scale=0.34]{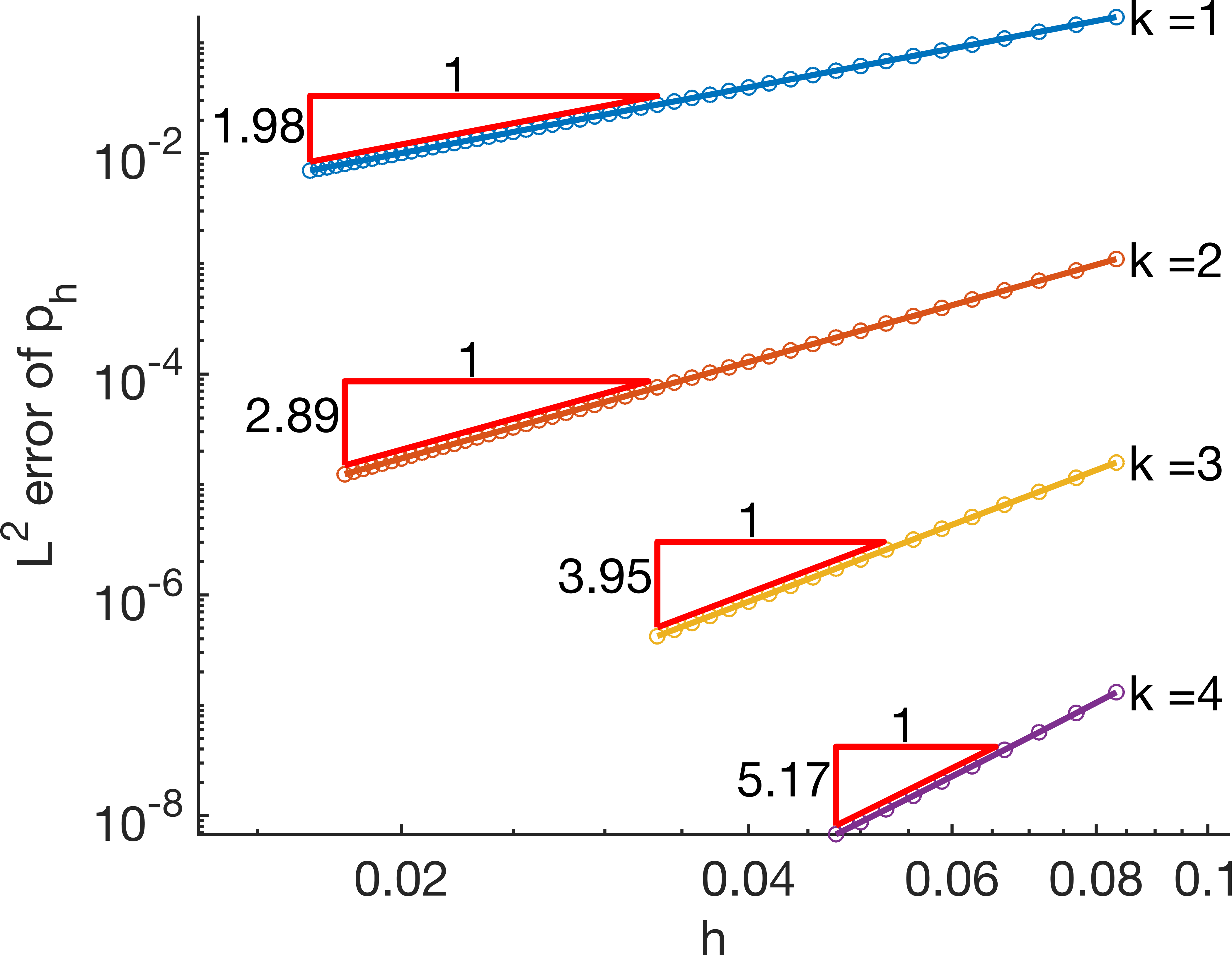}
    \label{SJS:fig:Lshaped_p_conv}}
    \subfigure{\includegraphics[scale=0.34]{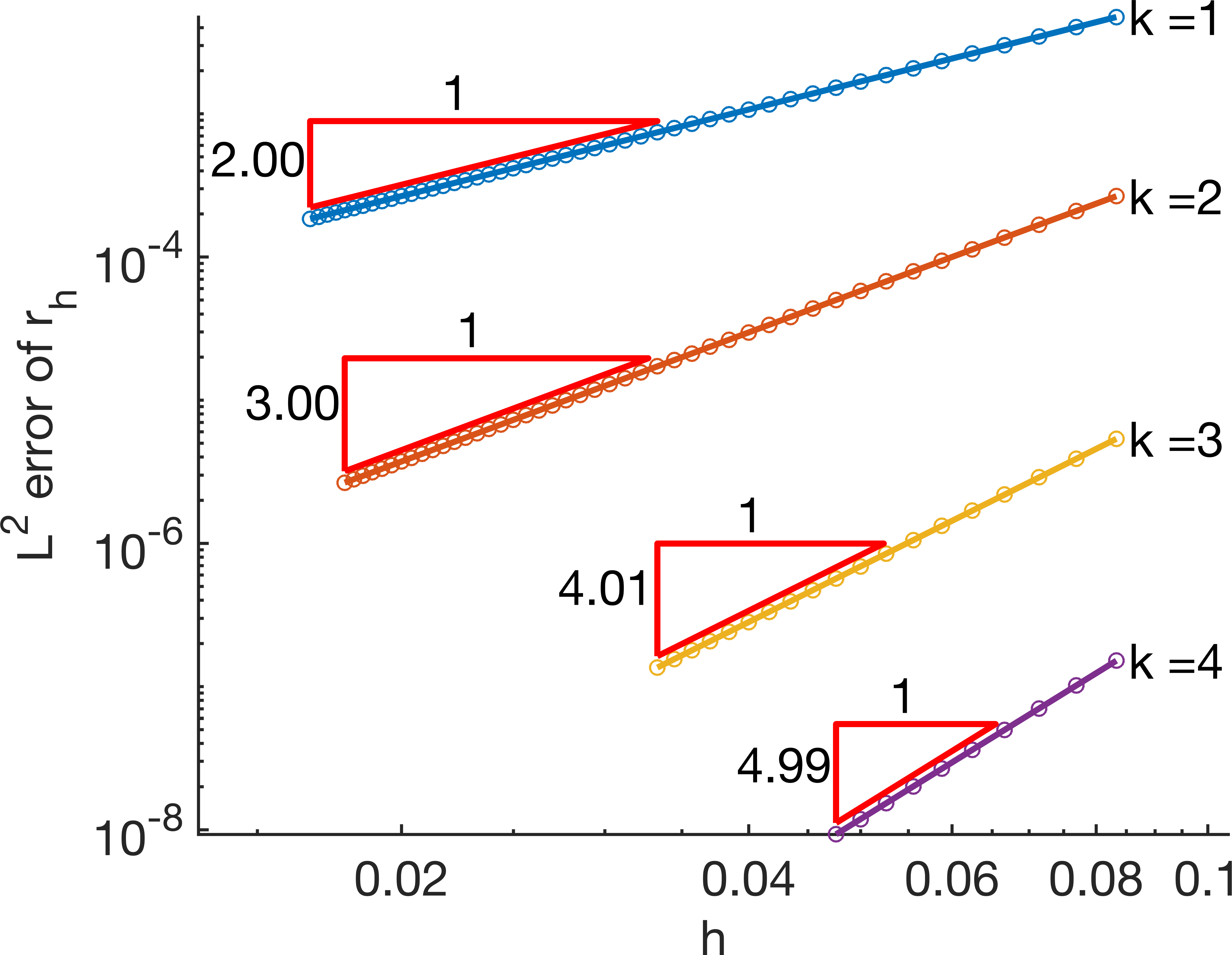}
    \label{SJS:fig:Lshaped_r_conv}}
    \caption{Non-convex domain (L-shaped) problem: $L^2$ convergence plots for $\LbH$, $\ubH$, $\pH$, $\HbH$, $\bbH$, and $\rH$.}
    \label{SJS:fig:Lshaped_conv}
  \end{center}
\end{figure}

\subsection{Singular Solution}

Although we do not discuss the implications of singular solutions on the 
theoretical convergence rates of the HDG scheme, applying the scheme to such a problem
is instructive in assessing its robustness.
This example illustrates the convergence of the HDG scheme
using a manufactured solution with a singularity
(similar to the example in Section 5.2
of \cite{HoustonSchoetzauWei09}). In particular,
we consider the same non-convex domain and mesh refinement as
in the previous example (see Figure~\ref{SJS:fig:Lshaped_mesh}).
We take $\Rey = \Rm = \kappa = 1$, $\wb = \bs{0}$, and $\db = (-1,1)$.
We choose $\gb$ and $\fb$ such that the analytical solution of  \eqnref{mhdlin} has the form
\begin{align*}
  \ub &= \LRp{
    \begin{array}{c}
      \rho^\lambda \LRs{(1+\lambda)\sin(\phi)\psi(\phi)+\cos(\phi)\psi'(\phi)}, \\
      \rho^\lambda \LRs{-(1+\lambda)\cos(\phi)\psi(\phi)+\sin(\phi)\psi'(\phi)}
    \end{array}
  }, &   \bb &= \nabla \LRp{ \rho^{2/3}\sin\LRp{\frac{2\phi}{3}} }, \\
  p &= -\rho^{\lambda-1} \frac{(1+\lambda)^2\psi'(\phi)+\psi'''(\phi)}{1-\lambda}, & 
  r &= 0,
\end{align*}
where
\begin{align*}
  \psi(\phi) &= \cos (\lambda w) \LRs{ \frac{\sin((1+\lambda)\phi)}{1+\lambda} - \frac{\sin((1-\lambda)\phi)}{1-\lambda}} - \cos(({1+\lambda}) \phi)  + \cos((1-\lambda)\phi), \\
  w &= \frac{3\pi}{2}, \qquad \lambda \approx 0.54448373678246.
\end{align*}
On $\pOmega$ we use the exact solution to set the boundary condition, i.e., $\ub_D = \ub$, $\hb_D = \bbt$, and $r_D = r$.
For this problem, it is known that $\ub \in \LRs{H^{1+\lambda}(\Omega)}^2$, $\p \in H^\lambda (\Omega)$, 
and $\bb \in \LRs{H^{2/3}(\Omega)}^2$, 
and that the solution contains magnetic and hydrodynamic singularities that 
are among the strongest singularities \cite{HoustonSchoetzauWei09}.

Convergence results for this problem are shown in Figure \ref{SJS:fig:singular_conv}.
For the fluid variables $\LbH$, $\ubH$, and $\pH$, we observe convergence rates of approximately 
$\lambda$, $2\lambda$, and $\lambda$, respectively.
For the magnetic variables $\HbH$, $\bbH$, and $\rH$, we observe convergence rates of approximately 
$1/2$, $2/3$, and $1/3$, respectively.

\begin{figure}
  \begin{center}
    \subfigure{\includegraphics[scale=0.34]{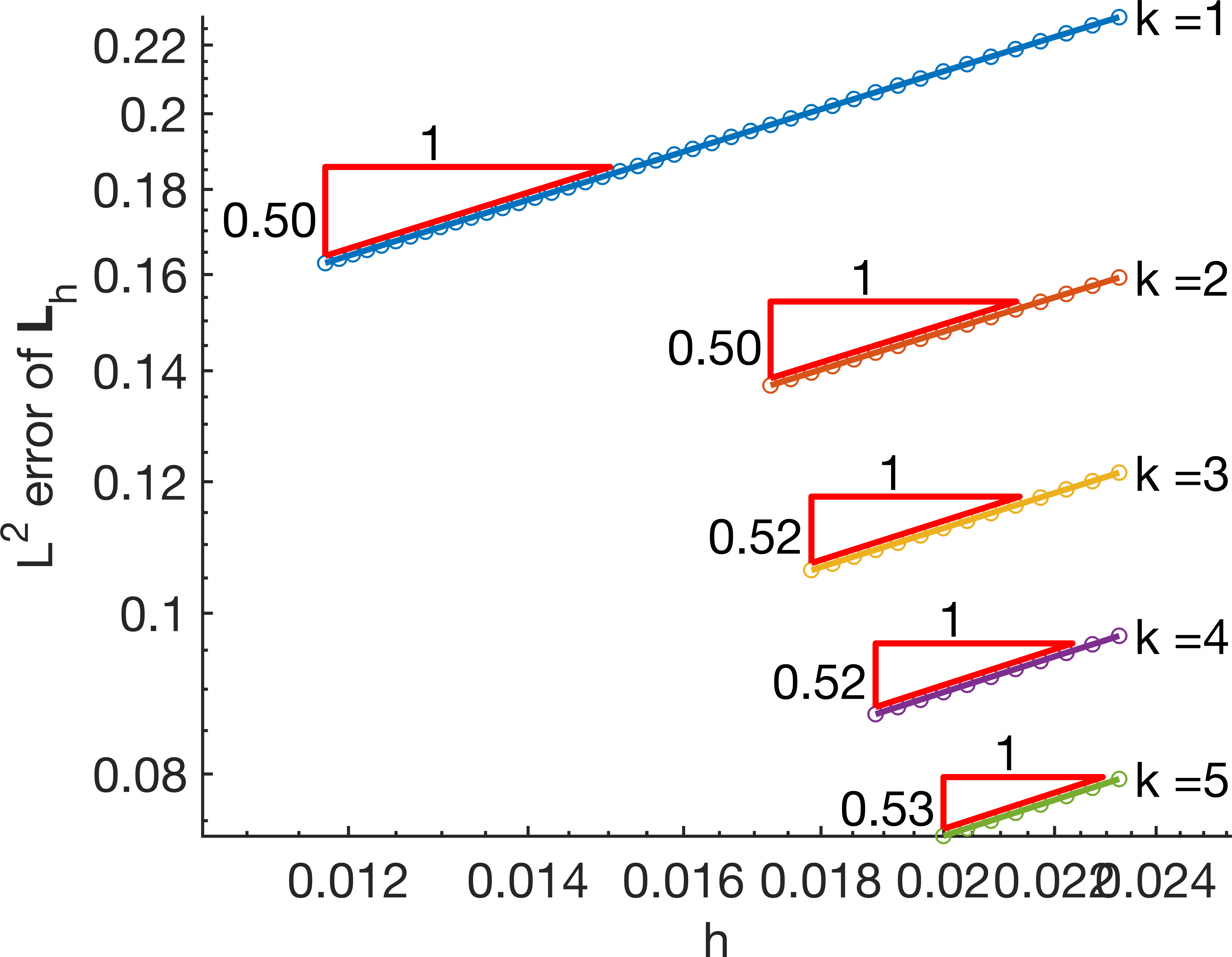}
    \label{SJS:fig:singular_L_conv}}
    \subfigure{\includegraphics[scale=0.34]{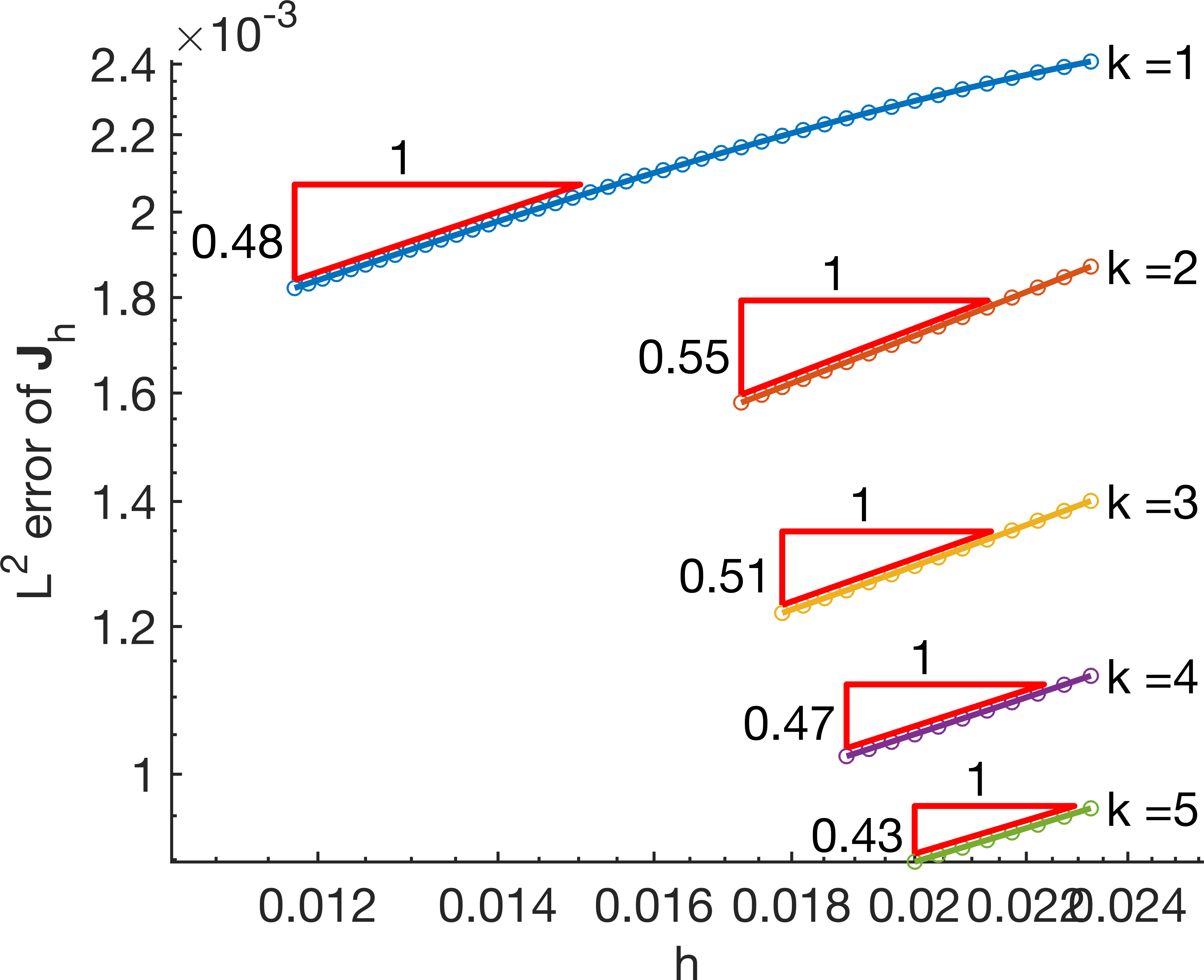}
    \label{SJS:fig:singular_J_conv}}
    \\ \vspace{-0.5em}
    \subfigure{\includegraphics[scale=0.34]{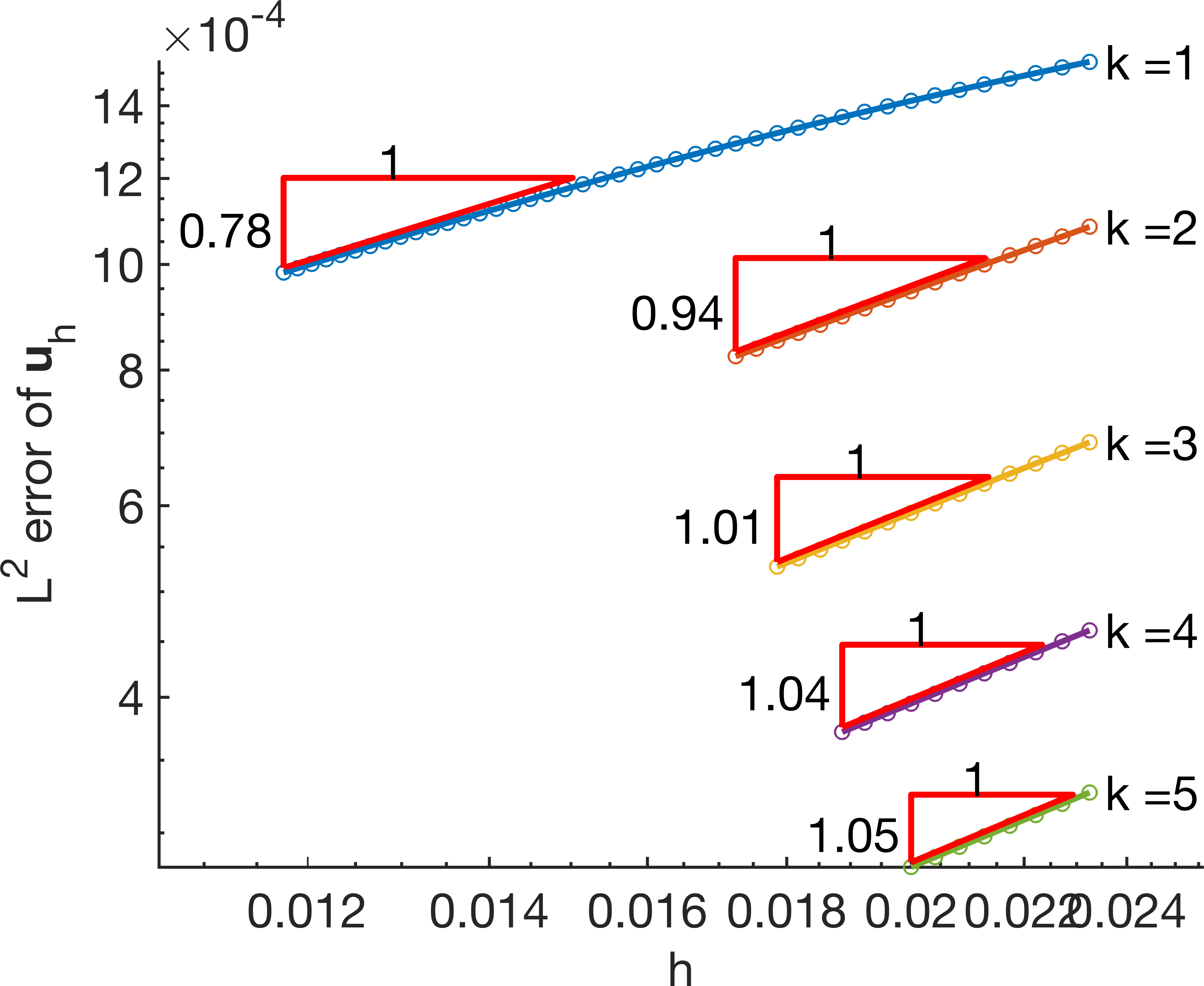}
    \label{SJS:fig:singular_u_conv}}
    \subfigure{\includegraphics[scale=0.34]{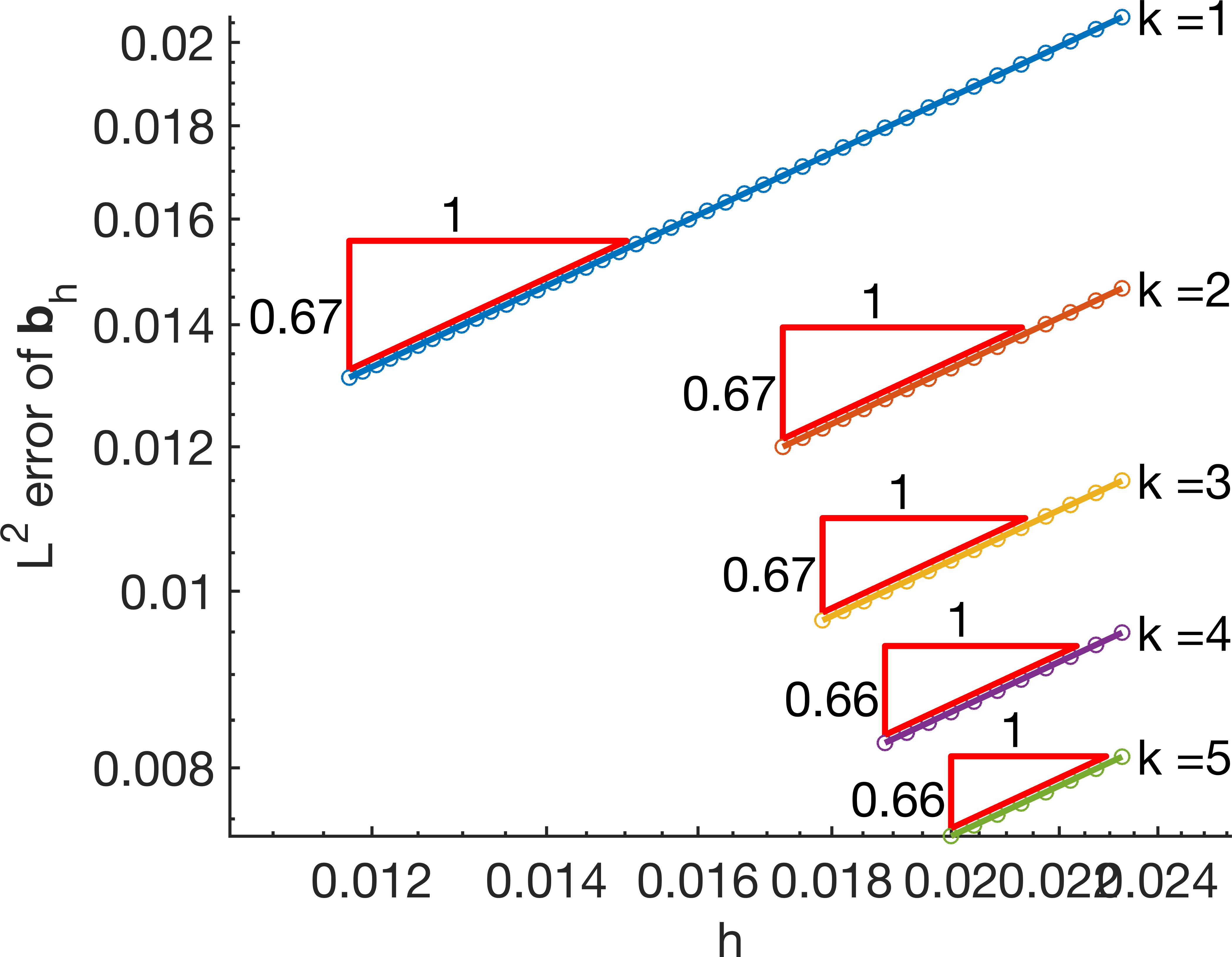}
    \label{SJS:fig:singular_b_conv}}
    \\ \vspace{-0.5em}
    \subfigure{\includegraphics[scale=0.34]{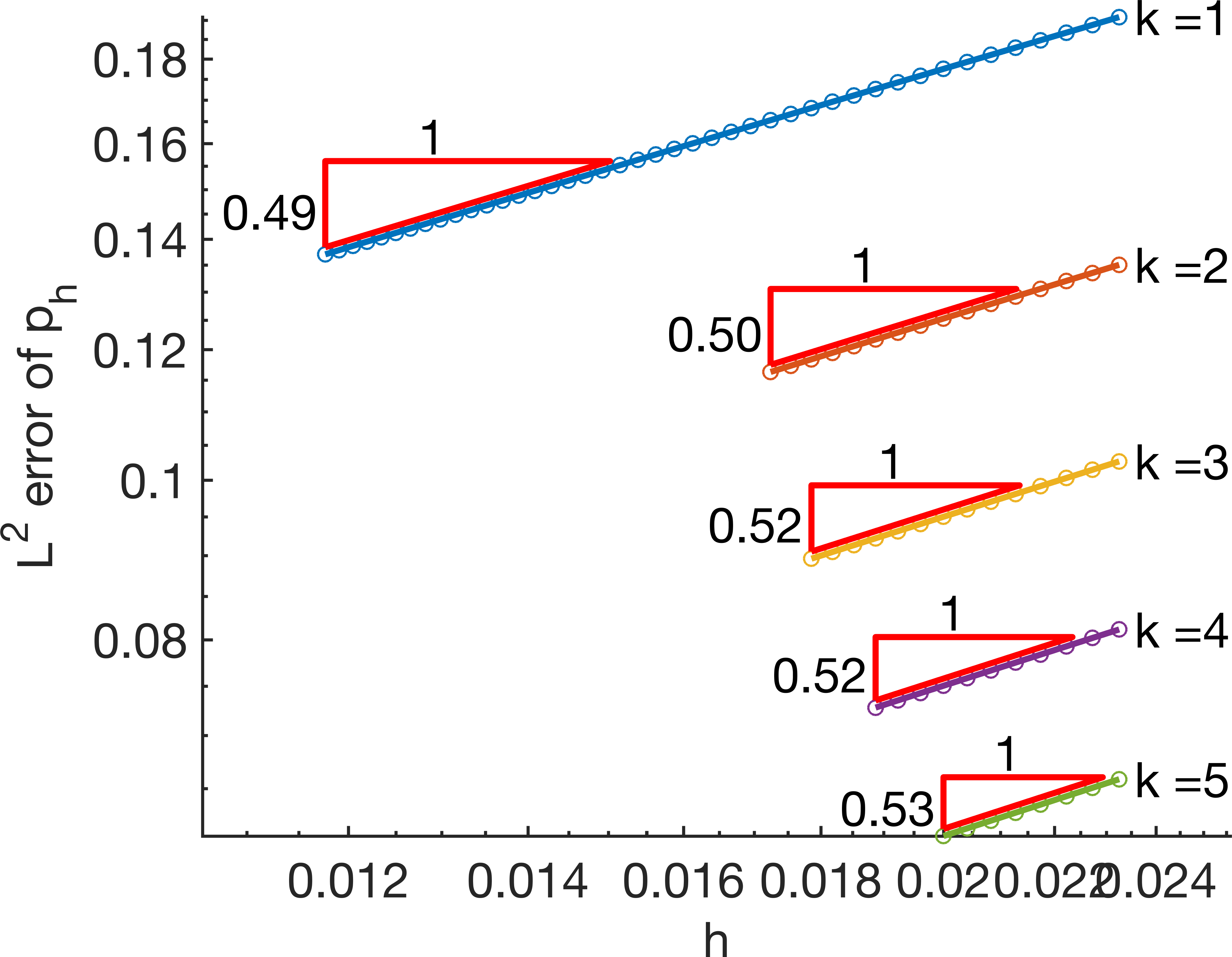}
    \label{SJS:fig:singular_p_conv}}
    \subfigure{\includegraphics[scale=0.34]{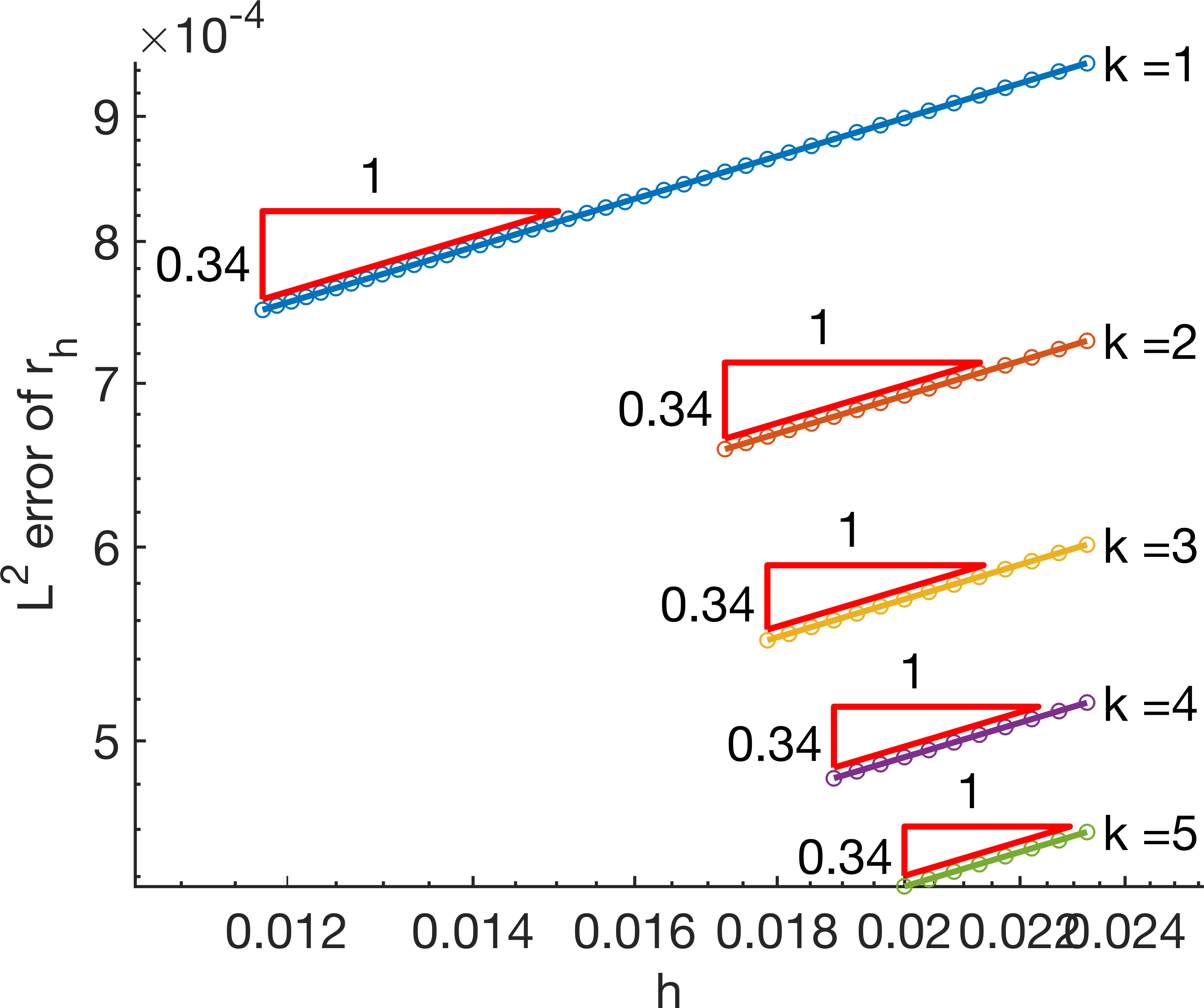}
    \label{SJS:fig:singular_r_conv}}
    \caption{Singular solution problem: $L^2$ convergence.}
    \label{SJS:fig:singular_conv}
  \end{center}
\end{figure}

\subsection{3D numerical experiments on cubical meshes}
We show numerical results for a three dimensional problem
with our HDG method adapted to hexahedral meshes with tensor product polynomial spaces.
Our theoretical analysis is only on the method on tetrahedral meshes, 
so it does not support this method on cubical meshes. 
Nonetheless we present this numerical result here in order to demonstrate that the HDG method 
can be applied to 3D problems and can be implemented using hexahedral meshes.

We set $\Omega = [0,1]^3$, $\wb = (1,2,-4)$, $\db = (-3,1,5)$, and set the forcing functions and boundary conditions to solve for the manufactured solution
\begin{alignat*}{3}
  \ub &= \bb = \LRp{
    \begin{array}{c}
      \sin(2\pi x_1) \sin(2\pi x_2) \sin(2\pi x_3), \\
      \sin(2\pi x_1) \cos(2\pi x_2) \cos(2\pi x_3), \\
      \cos(2\pi (x_1-x_3)) \sin(2\pi x_2) 
    \end{array} } ,  \\
  p &= e^{\LRp{x_1-\half}^2 + \LRp{x_2-\half}^2 + \LRp{x_3-\half}^2} - \pi^{\frac{3}{2}} {\textrm{erf}\LRp{\half}}^3,  
  \qquad  r = 0 , 
\end{alignat*}
with Dirichlet boundary conditions applied on $\partial \Omega$ for $\ub$, $r$, 
and the tangential components of $\bb$.

Convergence rates of $L^2$-errors with respect to uniform mesh refinements are given in Figure~\ref{SJS:fig:3Dprob3_conv}.
The convergence rates of $p$ and $r$ are optimal with order $k+1$, and the convergence rate of $\ub$ is suboptimal with $k+\frac 12$.
On the contrary, the errors of $\Lb$, $\Hb$, $\bb$ seem to have slightly lower convergence rates between $k$ and $k+\frac 12$.

\begin{figure}
  \begin{center}
	\subfigure{\includegraphics[scale=0.32]{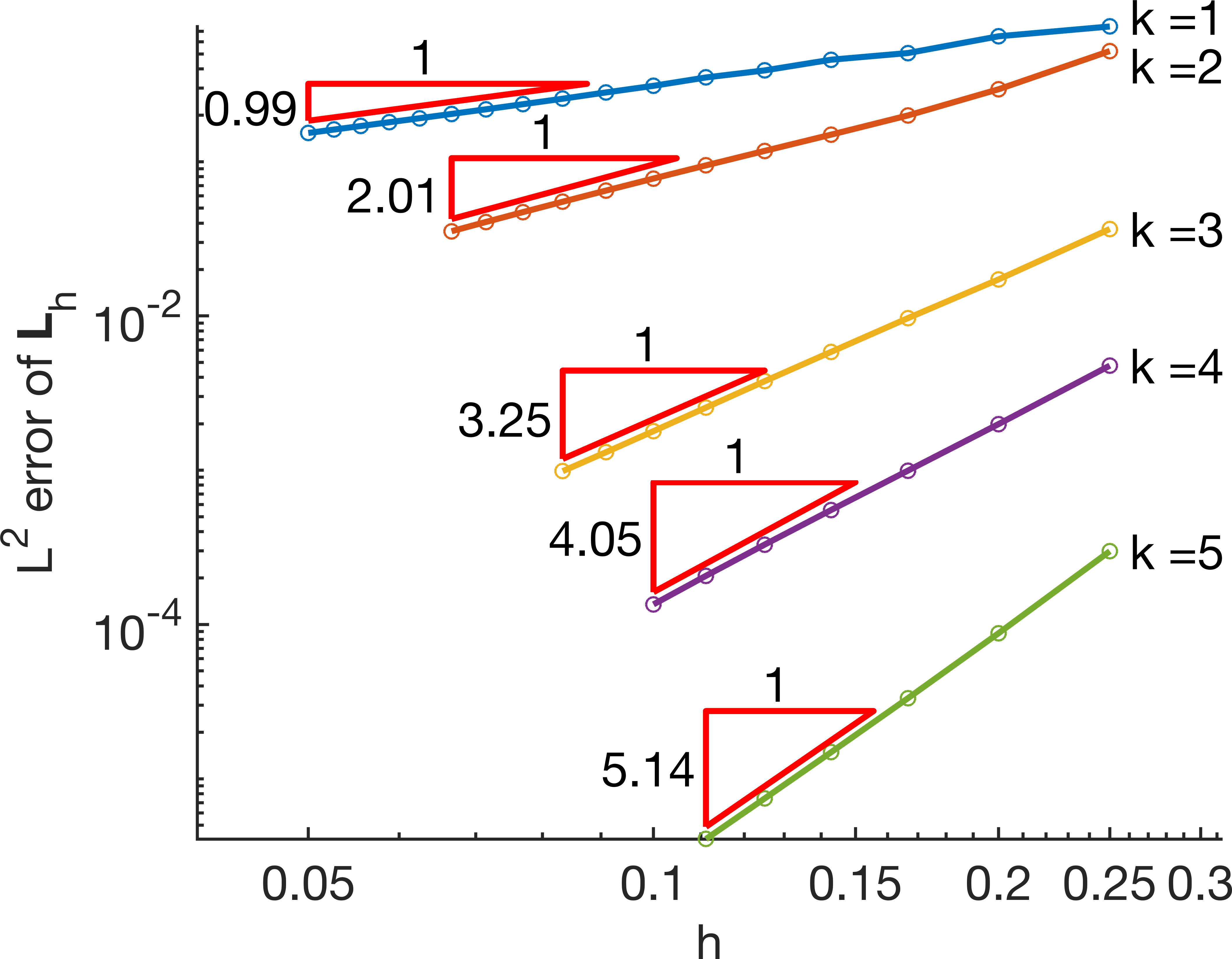}
	\label{SJS:fig:3Dprob3_L_conv}}
	\subfigure{\includegraphics[scale=0.32]{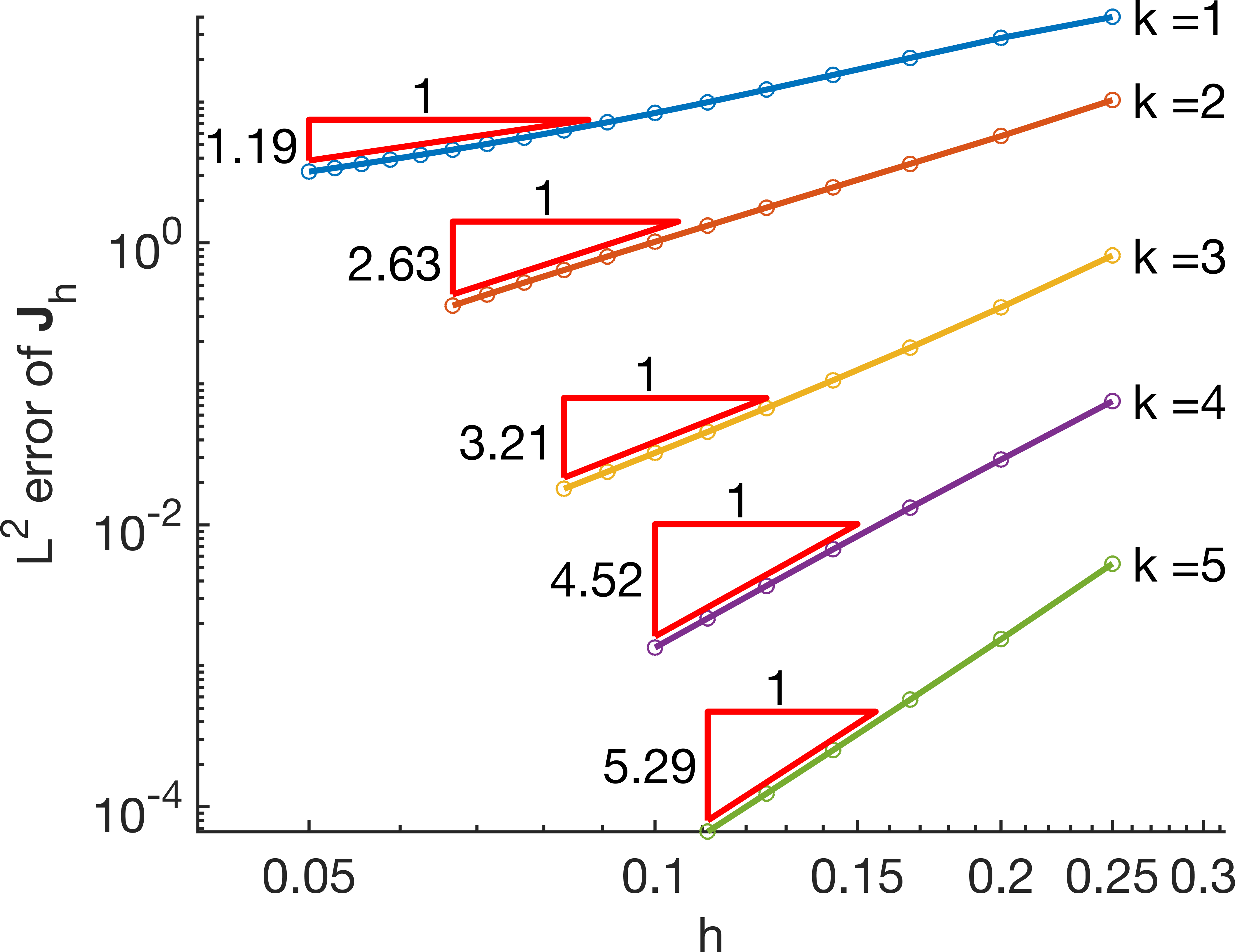}
	\label{SJS:fig:3Dprob3_J_conv}}
	\\
	\subfigure{\includegraphics[scale=0.32]{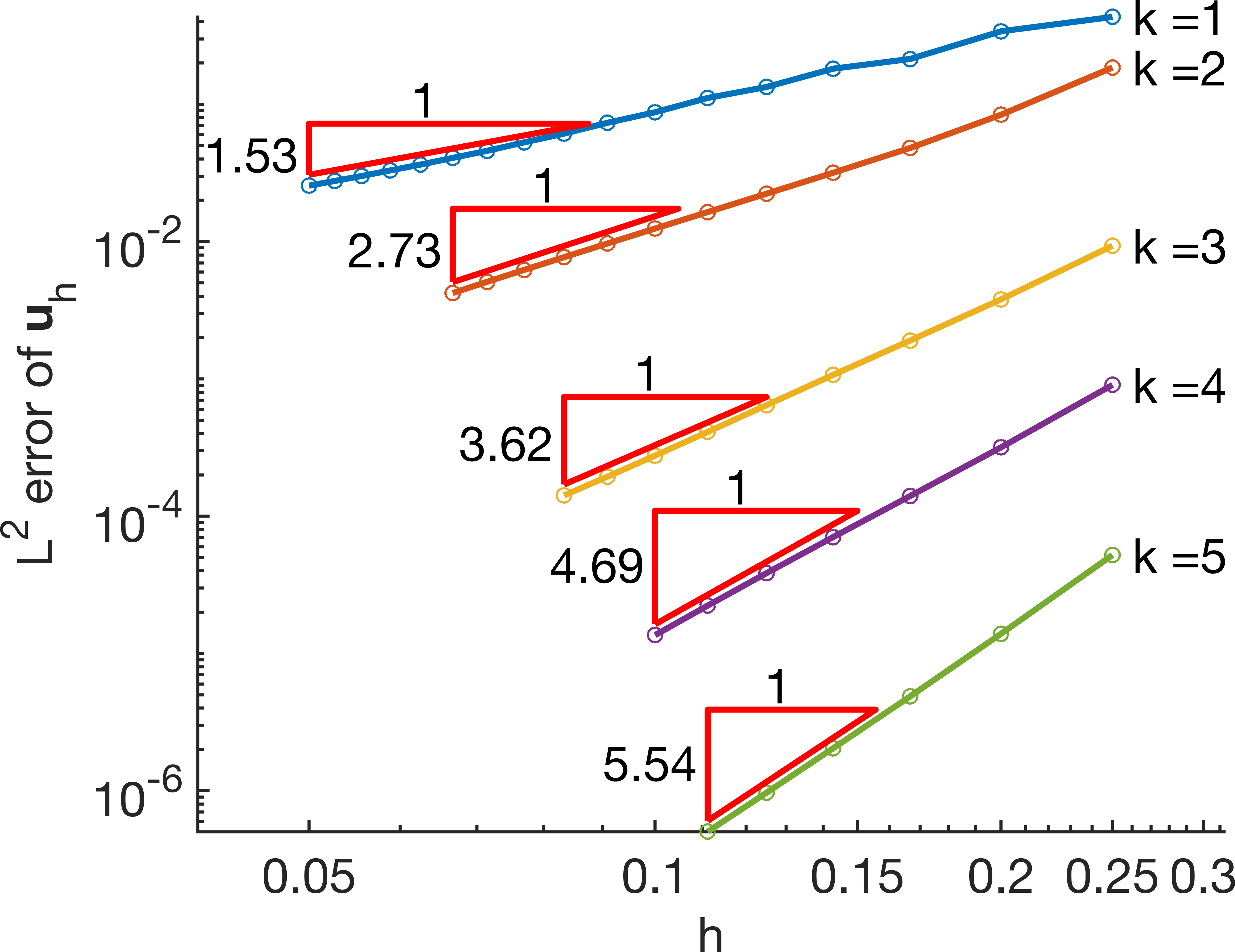}
	\label{SJS:fig:3Dprob3_u_conv}}
	\subfigure{\includegraphics[scale=0.32]{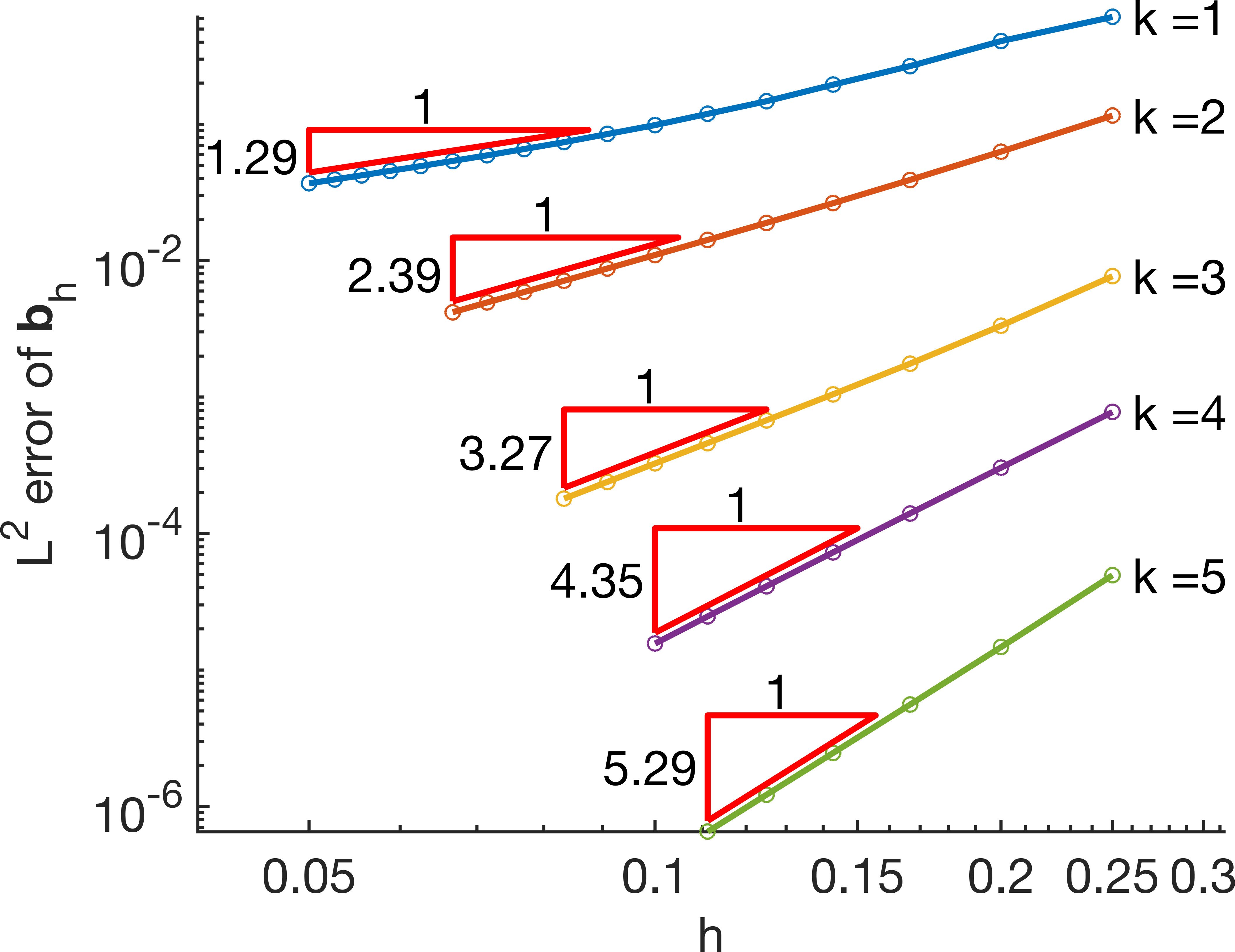}
	\label{SJS:fig:3Dprob3_b_conv}}
	\\
	\subfigure{\includegraphics[scale=0.32]{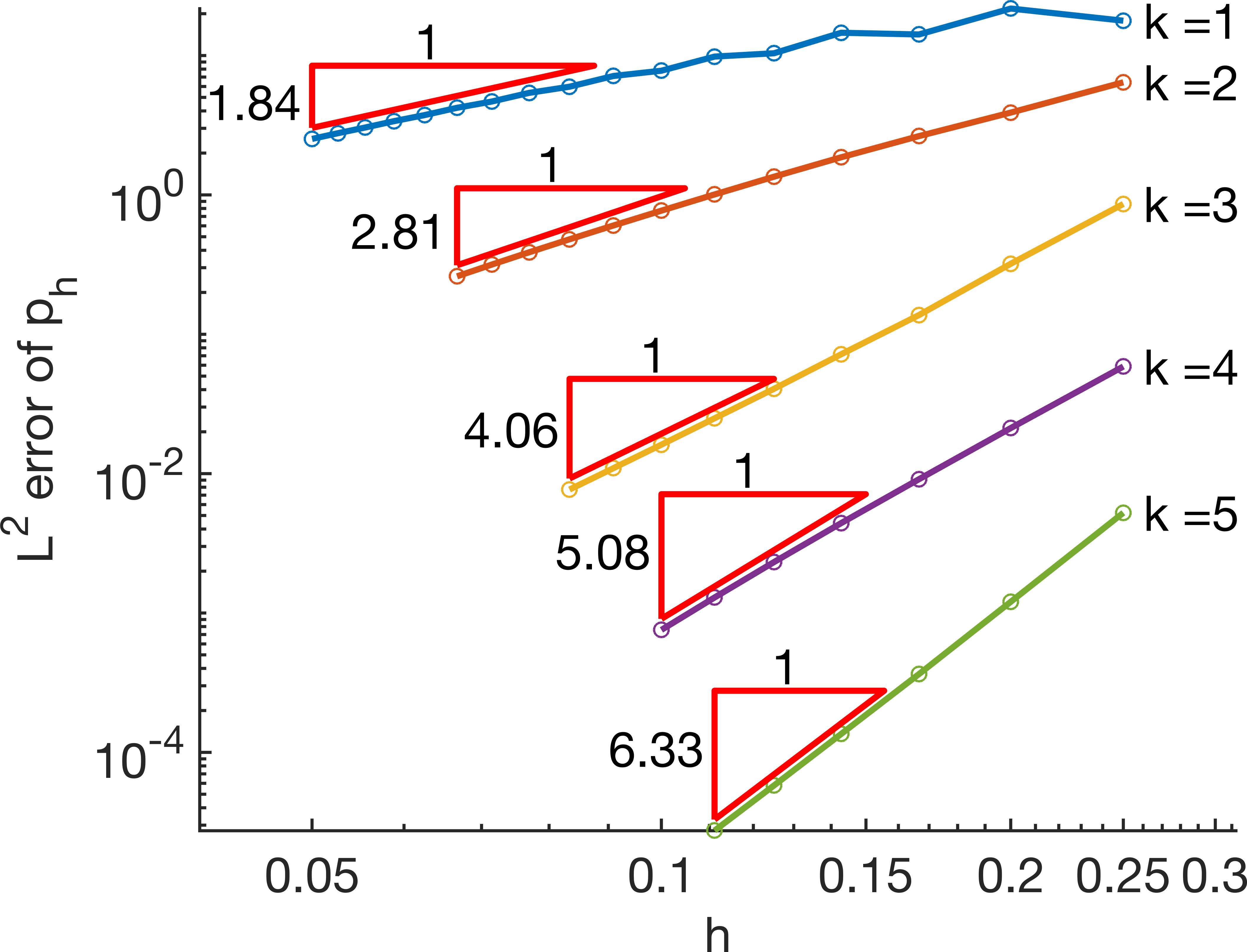}
	\label{SJS:fig:3Dprob3_p_conv}}
	\subfigure{\includegraphics[scale=0.32]{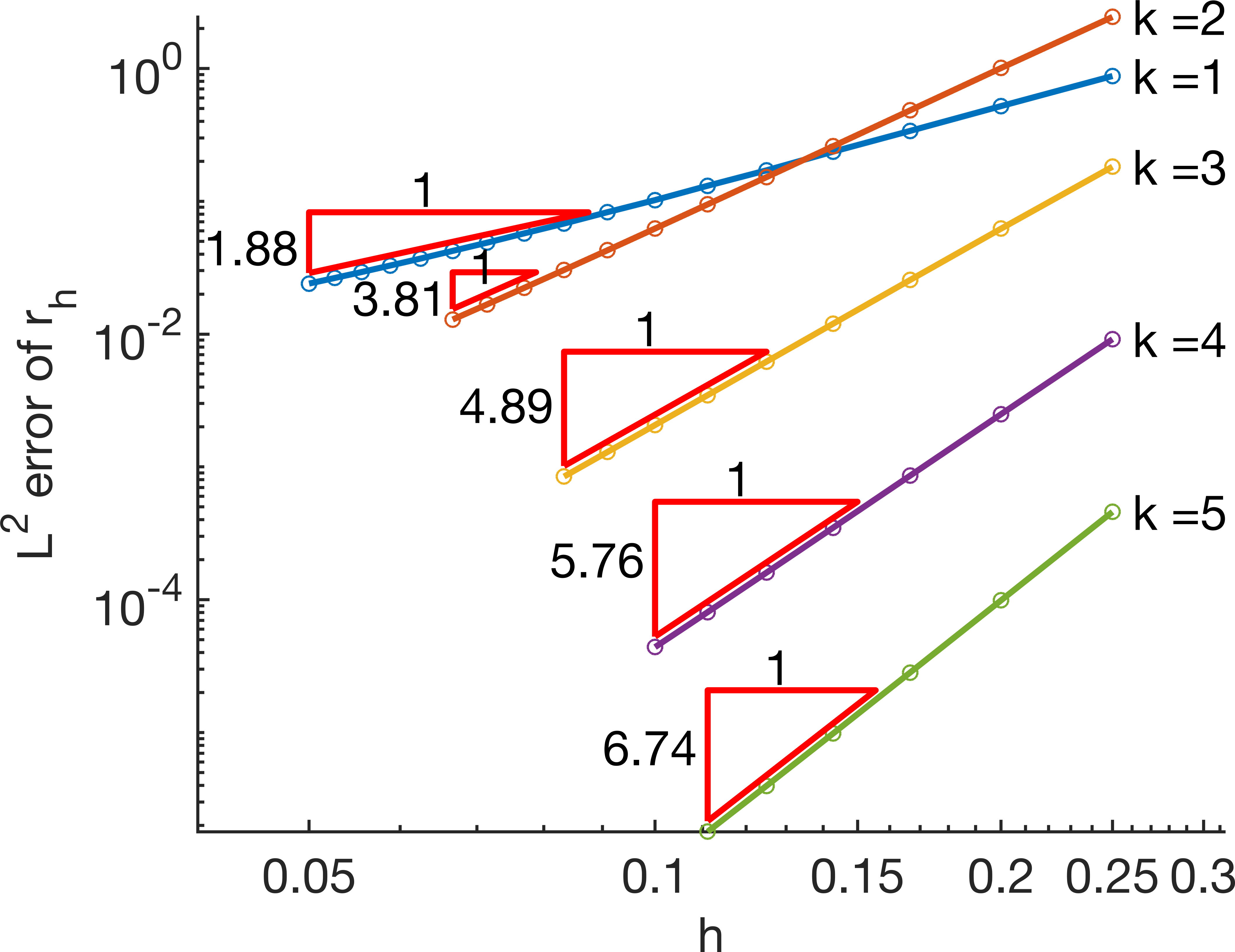}
	\label{SJS:fig:3Dprob3_r_conv}}
    \caption{3D problem with hexahedral mesh: $L^2$ convergence.}
	\label{SJS:fig:3Dprob3_conv}
  \end{center}
\end{figure}

\section{Conclusions and future work}
\seclab{conclusions}
In this paper we have constructed an HDG method for a linearization of the incompressible resistive magnetohydrodynamics equations.
We have carried out the a priori error analysis using elaborate interpolation operators, a duality argument 
with elliptic regularity assumptions, and an energy approach. 
Specifically, this allows us to prove optimal convergence for the velocity variable, $\ub$, and the magnetic field variable, $\bb$, 
and quasi-optimal convergence for the remaining  quantities $\Lb$, $\p$, $\Hb$, and $\r$.
Numerical performances of the method are tested on three examples: the Hartmann flow, a manufactured solution over a non-convex domain,
and singular solution on a non-convex domain. The numerical results show that  the theoretical convergence rates of all the unknowns 
are obtained for smooth solutions even when the elliptic regularity assumption of domain fails to hold. 
Ongoing work includes 3D computation on parallel computers for large-scale problems and  extensions of our HDG method to nonlinear time-dependent magnetohydrodynamics equations.

\appendix

\section{Auxiliary results}
\seclab{auxiliary}
In this appendix we collect some technical results that are useful for our analysis.
\begin{lemma}[Inverse Inequality. {\cite[Lemma 1.44]{PietroErn12}}] \label{lemma:inverse-inequality}
  For $v \in \poly{k}(\K)$ with $\K \in \Omegah$, there exists $C>0$ independent of $h$ such that
  \algns{
    \nor{\Grad v}_{0,\K} \le C h_K^{-1} \nor{v}_{0,\K} .
  }
\end{lemma}
\begin{lemma}[Trace inequality. {\cite[Lemma 1.49]{PietroErn12}}] \label{lemma:cont-inv-trace}
  For $v \in H^1(\Omegah)$ and for $\K \in \Omegah$ with $e \subset \pK$, there exists $C>0$ independent of $h$ such that
  \algns{
    \nor{v}_{0,e}^2 \le C \LRp{\nor{\Grad v}_{0,\K} + h_{\K}^{-1} \nor{v}_{0,\K} } \nor{v}_{0,\K} .
  }
\end{lemma}
Applying the arithmetic-geometric mean inequality to the right side, we can derive 
\algn{
  \label{continuous_trace_ineq}
  \nor{v}_{0,e} \lesssim \LRp{h_k^\half \nor{\Grad v}_{0,\K} + h_{\K}^{-\half} \nor{v}_{0,\K} }.
}
If $v \in H^1(\Omegah)$ is in piecewise polynomial spaces, we can derive the following 
inequality from Lemma~\ref{lemma:cont-inv-trace} and the inverse inequality (Lemma~\ref{lemma:inverse-inequality}):
\algn{
  \label{discrete_trace_ineq}
  \nor{v}_{0,e} \lesssim  { h_{\K}^{-\half} \nor{v}_{0,\K} } .
}

\begin{lemma}
  \lemlab{gradProjectionBound}
  Suppose that $\Pi:H^1 (\K) \ra \poly{k}(\K)$ is a bounded interpolation which is a projection on $\poly{k}(\K)$. Then
  \algns{
    \nor{\Grad \Pi \vb}_{0,\K} \lesssim \nor{\vb}_{1,\K} .
  }
\end{lemma}
\begin{proof}
  For any constant $c$, $\Grad \Pi \vb = \Grad \LRp{ \Pi \vb - c }$, so
  \algns{
    \nor{\Grad \Pi \vb}_{0,\K} &= \nor{\Grad (\Pi \vb - c)}_{0,\K} \\
    &\lesssim h^{-1} \nor{ \Pi \vb - c }_{0,\K} \lesssim h^{-1} \LRp{ \nor{\Pi \vb - \vb}_{0,\K} + \nor{\vb - c}_{0,\K}} \lesssim \nor{\vb}_{1,\K},
  }
  where we have used the Poincar\'e inequality and the approximation property of $\Pi$ in the last inequality.
\end{proof}

\section{Definition of projections and their properties}
\seclab{projections}
In this section we use
$\poly{k}$, $\polyb{k}$, $\polyt{k}$ for spaces of  scalar, $d$-dimensional vector, $d \times d$
matrix-valued polynomials. By $\poly{k}^{\perp}$, $\polyb{k}^{\perp}$,
$\polyt{k}^{\perp}$ we denote the spaces of polynomials of order at most $k$
orthogonal to all polynomials of order at most $\LRp{k-1}$. $\polyb{k}^t(e)$
contains  the tangential component of all polynomials in $\polyb{k}(e)$.
We desire to have error equations conform to the original
equations to facilitate the error analysis.
To begin, 
we define a collective interpolation operator
$\bs{\Pi} \LRp{\Lb, \ub, \p, \Hb, \bb, \r, \ubh, \bbh^t, \rh}$
implicitly through the interpolation errors $\veps_{\ub}^I = \ub
- \Pi \ub$, $\veps_{\bb}^I = \bb - \Pi \bb$, etc, where $\Pi \ub$,
$\Pi \bb$, etc, are components of the collective interpolator $\bs{\Pi}$ on $\ub, \bb$, etc. Specifically:
\begin{itemize}
  \item $L^2$ projections on $\e \in \Gh$ or on $\K \in \Omegah$ are defined as:
    \begin{subequations}
      \eqnlab{projections-tan}
      \begin{align}
        \label{rhat-proj}  \LRa{\erh^I, \gamma }_e &= 0, \quad \gamma \in \poly{k}\LRp{e} , \\
        \label{uhat-proj}  \LRa{\euh^I, \mub }_e &= 0, \quad \mub \in {\polyb{k}\LRp{e}} , \\
        \label{bhat-proj}  \LRa{\ebht^I, \lambdab^t}_e &= 0, \quad \lambdab^t \in {\polyb{k}^t\LRp{e}}, \\
        \label{J-proj}  \LRp{\ehi, \Jb }_\K &= 0, \quad \Jb \in \LRs{\poly{k}\LRp{\K}}^{\tilde{d}}, \\
        \label{p-proj}  \LRp{\ep^I, \q }_\K &= 0, \quad \q \in \poly{k}\LRp{\K} .
      \intertext{\item On each $\K \in \Omegah$ and $e \in \Gh$, $e \subset \pK$, $\Pi \bb$ and  $\Pi r$ are defined as}
        \eqnlab{br-proj1}
        (\eb^I, \cb)_\K &= 0, \quad  \cb \in {\polyb{k-1}(\K)}, \\
        \eqnlab{br-proj2}
        (\er^I, \s)_\K &= 0, \quad s \in \poly{k-1}(\K), \\
        \eqnlab{br-proj3}
        \LRa{ {\eb^I \cdot \n + \alpha_3 \veps_r^I} , \gamma }_e &= 0, \quad \gamma \in \poly{k}(e) .
      \intertext{\item  On each $\K \in \Omegah$ and $e \in \Gh$, $e \subset \pK$,  $\Pi \Lb$ and  $\Pi \ub$ are defined as }
        \eqnlab{Lu-projection1}
        - \LRp{\eL^I, \Gb }_\K + \LRp{\eu^I\otimes \wb, \Gb}_\K &= 0, \quad \Gb \in {\polyt{k-1}\LRp{\K}}, \\
        \eqnlab{Lu-projection2}
        \LRp{\eu^I,\vb}_\K & = 0, \quad \vb \in {\polyb{k-1}\LRp{\K}},\\
        \eqnlab{Lu-projection3}
        \LRa{-\eL^I\n + \LRp{m+\alpha_1}\eu^I,\mub}_\ed \\
        = -\langle \ep^I\n + \half \kappa\db \times (\n \times & (\ebt^I+\ebht^I) ), \mub \rangle_e, \quad \mub \in {\polyb{k}\LRp{e}}. \notag
      \end{align}
    \end{subequations}
\end{itemize}
The well-definedness and optimality of the $\L^2$-projections are clear. The coupled
projector $\Pi\LRp{\bb, \r}:= \LRp{\Pi \bb, \Pi r}$ has been studied
in \cite{CockburnGopalakrishnanSayas10}, and in particular we have
\begin{subequations}
  \eqnlab{ebI-erI}
  \algn{
    \label{eq:ebI-estm} \| \eb^I \|_{0,\K} &\lesssim h^{k +1} \nor{\bb }_{k+1, \K} + \alpha_3 h^{k +1} \nor{\r }_{k +1, \K} , \\
    \label{eq:erI-estm} \| \er^I \|_{0,\K} &\lesssim \alpha_3^{-1} h^{k +1} \nor{ \Div \bb }_{k, \K} + h^{k +1} \nor{\r}_{k +1, \K},
  }
\end{subequations}
where, again, for simplicity we choose the same solution order $k$ for
all the unknowns. Here, we assume that $\bb$ and $\r$ are sufficiently
smooth, that is, $\bb \in \LRs{H^{k+1}\LRp{\Omega}}^d$ and $\r \in H^{k+1}\LRp{\Omega}$.

To understand the approximation capability of the coupled projector
$\Pi\LRp{\Lb,\ub}:=\LRp{\Pi\Lb, \Pi\ub}$
we need to recall a result in \cite[Lemma~A.1]{CockburnGopalakrishnanSayas10}.
\begin{lemma} \lemlab{trace-equiv}
  Suppose that $w \in \poly{k}^{\perp}(K)$. Then, for any $\e \subset \pK$, the map $w \mapsto w|_e \in \poly{k}(e)$ is an isomorphism and 
  $\| w \|_{0,\K}^2 \sim h_K \| w \|_{0,e}^2$ holds with a constant independent of $h_\K$.
\end{lemma}

\begin{lemma}[Estimation for $\eu^I$]
  Suppose $\ub \in \LRs{{H}^{k+1}\LRp{\Omega}}^d$, $\Lb \in \LRs{{H}^{k+1}\LRp{\Omega}}^{d\times d}$, $\r \in H^{k+1}\LRp{\Omega}$, $\bb \in \LRs{{H}^{k+1}\LRp{\Omega}}^d$, and $\p \in H^{k+1}\LRp{\Omega}$. The projection $\Pi\ub$ is well-defined and optimal, i.e.,
  \algns{
    \| \eu^I \|_{0} &\lesssim C(\alpha_1, \wb) [ ({ \alpha_1 + \nor{\wb}_\Linfty + h \nor{\wb}_\Woneinfty }) h^{k+1} \| \ub \|_{k+1} \\
    &\quad + h^{k+1} \| \Div \Lb - \Grad p \|_{k} + \kappa h^{k+1} \| \db \|_\Linfty \LRp{\| \bb \|_{k+1} + \alpha_3 \| r \|_{k +1} } ]
  }
  where $C(\alpha_1, \wb) = 1/(\alpha_1 - \half \nor{\wb}_\Linfty )$.
\end{lemma}
\begin{proof}
  We extend the proof of a result in \cite{CesmeliogluCockburnNguyenEtAl13}. To
  begin, we define 
  $\gb := -\half\kappa\db \times \LRp{\n \times \LRp{\ebt^I+\ebht^I}}$, 
  take $\mub = \vb|_\pK$ for some $\vb \in \polyb{k}^{\perp}(K)$ which will be determined later,
  and rewrite \eqnref{Lu-projection3} as
  \begin{align}
    \LRa{\LRp{\alpha_1+m}\eu^I,\vb}_\pK &=  \LRa{\eL^I \n - \ep^I \n + \gb, \vb }_\pK  \notag \\
    \label{rewrite-Lu-projection3}
    &\hspace{-3em}= \LRp{\Div \eL^I, \vb}_K + \LRp{\eL^I, \Grad \vb}_K - \LRp{\Grad \ep^I, \vb}_K - \LRp{\ep^I, \Div \vb}_K + \LRa{ \gb, \vb}_\pK \\
    &\hspace{-3em}= \LRp{ \Div \Lb - \Grad p, \vb}_K + \LRp{ \eu^I \otimes \wb, \Grad \vb}_K + \LRa{ \gb, \vb }_\pK \notag
  \end{align}
  where we have used the integration by parts in the second equality,
  definitions of the projections $\Pi\Lb$, $\Pi\ub$, and $\Pi\p$, the
  orthogonality between $\Div\LRp{\Pi\Lb}, \Grad \LRp{\Pi\p}$ and $\vb
  \in \polyb{k}^{\perp}(K)$, and the orthogonality $\veps_p^I \perp
  \Div \vb$ in the last equality.
  Now, let $\Pr_k \ub$ be the $L^2$ projection of $\ub$ and define $\delu^I := \ub - \Pr_k \ub$.
  By the triangle inequality, it suffices to estimate the approximation capability of $\delu := \Pr_k \ub - \Pi \ub$.
  From the above formula, we have
  \mltlns{
    \LRa{\LRp{\alpha_1+m}\delu, \vb}_\pK - \LRp{ \delu \otimes \wb, \Grad \vb}_K \\
    = \underbrace{ - \LRa{ \LRp{\alpha_1+m} \delu^I , \vb }_\pK + \LRp{ \delu^I \otimes \wb, \Grad \vb}_K }_{=: F_{\ub} (\vb)} 
    + \underbrace{\LRp{ \Div \Lb - \Grad p, \vb}_K}_{=: F_{\Lb} (\vb)} + \underbrace{\LRa{ \gb, \vb }_\pK}_{=:F_{\gb} (\vb)}.
  }
  Since $\delu \in \polyb{k}^{\perp}(K)$, we can take $\vb = \delu$ to obtain
  \algns{
    \LRa{\LRp{\alpha_1+\frac{m}{2}}\delu, \delu}_\pK
    = F_{\ub}(\delu) + F_{\Lb}(\delu) + F_{\gb}(\delu)
  }
  where we have used $\Div \wb = 0$ in the integration by parts
  \algns{
    - \LRp{ \delu \otimes \wb, \Grad \delu}_K 
    = - \half\LRp{\wb, \Grad\LRp{\delu\cdot\delu}}_K \qquad \\
    = \half\LRp{(\Div\wb) \delu, \delu}_K - \half\LRa{\wb\cdot\n\delu,\delu}_\pK
    = - \LRa{\frac{m}{2}\delu,\delu}_\pK.
  }
  By Lemma~\lemref{trace-equiv} and the fact that $\alpha_1 > \half \nor{\wb}_\Linfty$, we have
  \algn{
    \eqnlab{delu-estm-1}
    \nor{\delu}_{0,\K}^2
    &\lesssim h_\K \nor{\delu}_{0,\pK}^2
    \lesssim \frac{h_{\K}}{\alpha_1 - \half \nor{\wb}_{L^\infty(\K)}}
    \LRa{\LRp{\alpha_1 + \frac m2}\delu , \delu}_\pK \\
    &= \frac{h_{\K}}{\alpha_1 - \half \nor{\wb}_{L^\infty(\K)}}
    \LRp{F_{\ub}(\delu) + F_{\Lb}(\delu) + F_{\gb}(\delu)}. \notag
  }

  We now estimate $| F_{\ub} (\delu) |$.
  Defining $\delw = \wb - \Pr_0 \wb$, note that $(\delu^I \otimes \wb, \Grad \delu)_K = (\delu^I \otimes \delw, \Grad \delu)_K$ holds due to the definition of $\delu^I$. Thus, we have
  \algns{
    |F_{\ub}(\delu)| 
    &= \snor{ \LRa{ (m + \alpha_1) \delu^I, \delu}_\pK  + \LRp{ \delu^I \otimes \delw, \Grad \delu}_K} \\
    &\lesssim (\| \wb \|_{L^\infty(\pK)} + \alpha_1) h_K^{-1} \LRp{ \| \delu^I \|_{0,\K} + h_K \| \Grad \delu^I \|_{0,\K} } \| \delu \|_{0,\K} \\
    &\quad + \| \delu^I \|_{0,\K} \| \wb \|_{W^{1,\infty}(K)} \| \delu \|_{0,\K}
  }
  where we have used $\| \delw \|_{L^\infty(K)} \lesssim h_K \| \wb \|_{W^{1,\infty}(K)}$, 
  the inverse inequality,
  and the continuous and discrete trace inequalies 
  (\eqref{continuous_trace_ineq} and \eqref{discrete_trace_ineq}, respectively) in the last step. 
  Taking the approximation capability of $\Pr_k \ub$ into account, we get
  \algn{ 
    \label{eq:Fu-estm}
    |F_{\ub}(\delu)| \lesssim \LRp{ \alpha_1 + \| \wb \|_{L^\infty(\pK)} + h_K \| \wb \|_{W^{1,\infty}(K)} } h_K^{k} \| \ub \|_{k+1, \K} \nor{\delu}_{0,\K}.
  }

  The estimate of $| F_{\Lb} (\delu) |$ is straightforward since $\delu \in \polyb{k}^{\perp}(K)$:
  \algn{  
    \label{eq:FL-estm}
    | F_{\Lb} (\delu)| &\leq \| \Div \Lb - \Grad p - \Pr_{k-1} (\Div \Lb - \Grad p) \|_{0,\K} \| \delu \|_{0,\K} \\
    &\lesssim h_\K^{k} \| \Div \Lb - \Grad p \|_{k,\K} \| \delu \|_{0,\K}. \notag
  }

  For the estimate of $| F_{\gb} (\delu) |$, note that
  \algn{
    \label{bound_of_grad_ebi}
    \nor{\Grad\eebi}_{0,K}
    &\leq \nor{\Grad\LRp{\bb-\Pr_k\bb}}_{0,K} + \nor{\Grad\LRp{\Pr_k\bb-\Pi\bb}}_{0,K} \\
    &\lesssim \nor{\bb-\Pr_k\bb}_{1,K} + h_K^{-1} \nor{\Pr_k\bb-\Pi\bb}_{0,K} \notag\\
    &\lesssim \nor{\bb-\Pr_k\bb}_{1,K} + h_K^{-1} \nor{\Pr_k\bb-\bb}_{0,K} + h_K^{-1} \nor{\bb-\Pi\bb}_{0,K} \notag\\
    &\lesssim h_K^k \nor{\bb}_{k+1,K} + h_K^{-1} \nor{\eebi}_{0,K}. \notag
  }
  Using the definition of $\Pi^*\bbht$ as the $L^2$-projection,
  the continuous and discrete trace inequalies 
  (\eqref{continuous_trace_ineq} and \eqref{discrete_trace_ineq}, respectively),
  the above estimate of $\nor{\Grad\eebi}_{0,K}$ \eqref{bound_of_grad_ebi},
  and finally the estimate of $\nor{\eebi}_{0,K}$ in \eqnref{ebI-estm},
  we have
  \begin{align}
    \eqnlab{Fg-estm}
    | F_{\gb}(\delu) |  
    &\lesssim \kappa\nor{\db}_{L^\infty(\pK)}\LRp{\nor{\ebt^I}_{0,\pK} 
    + \nor{\ebht^I}_{0,\pK} } \nor{\delu}_{0,\pK} \\
    &\lesssim \kappa\nor{\db}_{L^\infty(\pK)}\nor{\ebt^I}_{0,\pK} \nor{\delu}_{0,\pK} \notag \\
    &\lesssim \kappa\nor{\db}_{L^\infty(\pK)}\LRp{h_\K^{-1} \| \eb^I \|_{0,\K} 
    + \| \Grad\eb^I \|_{0,\K} } \nor{\delu}_{0,\K} \notag \\
    &\lesssim \kappa\nor{\db}_{L^\infty(\pK)}\LRp{h_\K^{-1} \| \eb^I \|_{0,\K} 
    + h_K^k \nor{\bb}_{k+1,\K} } \nor{\delu}_{0,\K} \notag \\
    &\lesssim \kappa\nor{\db}_{L^\infty(\pK)} h_\K^{k}\LRp{\nor{\bb}_{k+1, \K} 
    + \alpha_3 \nor{r}_{k+1, \K}} \nor{\delu}_{0,\K}. \notag
  \end{align}
  Using \eqref{eq:Fu-estm}, \eqref{eq:FL-estm}, and \eqref{eq:Fg-estm} in \eqref{eq:delu-estm-1}, 
  and using the triangle inequality $\| \eu^I \|_{0,\K} \lesssim \| \ub - \Pr_k \ub \|_{0,\K} + \| \delu \|_{0,\K}$ ends the proof.
\end{proof}

To estimate $\| \eL^I \|_{0,\K}$, we need some auxiliary results.
We first recall a result with a sketch of its proof.
\begin{lemma} \label{thm:Pi-PF-estm}
  Let $e_K$ be a fixed face of the simplex $K$. For $\bs{R} \in [L^2(K)]^d$ and $g \in L^2(\pK)$, we define $\Pi(\bs{R}, g) \in \polyb{k}(K)$ as 
  \algns{ 
    \LRp{\Pi(\bs{R}, g) , \bs{\tau}}_K &= \LRp{\bs{R}, \bs{\tau}}_K, & & \forall \bs{\tau} \in \polyb{k-1}(K), \\
    \LRa{ \Pi(\bs{R}, g)  \cdot \n, \mu }_e &= \LRa{ g, \mu}_e, & & \forall \mu \in \poly{k}(e) \quad \text{for } e \not = e_K .
  }
  Then, $\| \Pi(\bs{R}, g) \|_K \lesssim \| \bs{R} \|_K + h_K^{1/2} \| g \|_{\pK}$. 
\end{lemma}
\begin{proof}
  We refer to \cite{cockburn2008superconvergent} for the existence and uniqueness of $\Pi(\bs{R}, g)$. 
  Let $\bs{\sigma}_1 = \Pi(\bs{R}, 0)$ and $\bs{\sigma}_2 = \Pi(\bs{0}, g)$. By the standard scaling argument, 
  \algns{ 
    \| \bs{\sigma}_1 \|_{0,K} \lesssim \| \Pr_{k-1} \bs{\sigma}_1 \|_{0,K} \leq \| \bs{R} \|_{0,K} .
  }
  To estimate $\bs{\sigma}_2$, note that there exists $a_e \in \R$, $e \not = e_K$ such that $(1 \; 0 \; 0)^T = \sum_{e, e \not = e_K} a_e \n_e$, and the first component of $\bs{\sigma}_2$, say $\bs{\sigma}_2^1$, is 
  \algns{ 
    \bs{\sigma}_2^1 = \sum_{e, e \not = e_K} a_e (\bs{\sigma}_2 \cdot \n_e) .
  }
  Since $\bs{\sigma}_2 \cdot \n_e \perp \poly{k-1}(K)$ by the definition of $\bs{\sigma}_2$, 
  $\| \bs{\sigma}_2 \cdot \n_e \|_{0,K} \lesssim h_K^{1/2} \| \bs{\sigma}_2 \cdot \n_e \|_{0,e} \lesssim h_K^{1/2} \| g \|_{0,e}$ 
  by Lemma~\lemref{trace-equiv}.
  The estimate $\| \bs{\sigma}_2 \|_{0,K} \lesssim h_K^{1/2} \| g \|_{0,\pK}$ follows easily 
  by using this inequality to each component of $\bs{\sigma}_2$. 
\end{proof}

We now recall other known facts without proofs (cf. Lemma~4.8 in \cite{CockburnGopalakrishnanNguyenEtAl11}). 
\begin{lemma}
  \label{lemma:basis}
  For a face $e$ of $K$, let $\mc{B}_e$ be an orthogonal basis of the vectors orthogonal to $\n_e$, 
  and let $\mc{B} = \LRc{ \Ir_d } \cup \LRc{ \tb \otimes \n_e, \tb \in \mc{B}_e}$. 
  This $\mc{B}$ is a basis of the space of $d \times d$ matrices. 
\end{lemma}

\begin{lemma}[Estimation for $\eL^I$]
  Assume $\ub \in \LRs{H^{k+1}\LRp{\Omega}}^d$, $\Lb \in \LRs{{H}^{k+1}\LRp{\Omega}}^{d\times d}$, $\r \in H^{k+1}\LRp{\Omega}, \bb \in
  \LRs{{H}^{k+1}\LRp{\Omega}}^d$, and $\p \in H^{k+1}\LRp{\Omega}$. Furthermore, suppose the trace of the
  tensor $\Lb$ vanishes, i.e., $\tr\Lb = 0$. There holds:
  \begin{align*}
    \nor{\eL^I}_0 
    &\lesssim 
    h^{k+1}\nor{p}_{k+1}
    + h^{k+1} \nor{\Lb}_{k+1} 
    + \kappa \nor{\db}_{L^\infty} \LRp{h^{k+1}\nor{\bb}_{k+1} + \alpha_3 h^{k+1} \nor{\r}_{k+1}} \\
    & \quad + \LRp{\alpha_1 + \nor{\wb}_{L^\infty} + h \nor{\wb}_{W^{1,\infty}}} \nor{\eeui}_{0} 
    + \LRp{\alpha_1 + \nor{\wb}_{L^\infty}} h^{k+1}\nor{\ub}_{k+1}. \notag \\
  \end{align*}
\end{lemma}
\begin{proof}
  We proceed in a manner similar to \cite[Theorem 2.3]{CesmeliogluCockburnNguyenEtAl13} with adaptations
  corresponding to our more complicated projectors
  $\bs{\Pi}\LRp{\Lb,\ub}$.

  The dual basis of $\mc{B}$ (see Lemma~\ref{lemma:basis}) can be written as 
  \algns{ 
    \mc{B}^* = \LRc{ \frac 1d \Ir_d } \cup \LRc{ W_{e,\tb}\;:\; e \subset \pK, \tb \in \mc{B}_e },
  }
  where $W_{e,\tb}:(\tb \otimes \n_e) = 1$ for the $e$ and $\tb$ corresponding to the subscripts of $W$, 
  and 0 otherwise.
  Any $d \times d$ matrix, $A$, can be written as 
  \algns{ 
    A = \sum_e \sum_{\tb \in \mc{B}_e} (A : (\tb \otimes \n_e))W_{e,\tb} + \frac{ \tr A}{d} \Ir_d 
    = \sum_e \sum_{\tb \in \mc{B}_e} (A \n_e \cdot \tb ) W_{e,\tb} + \frac{ \tr A}{d} \Ir_d ,
  }
  so 
  \algn{ 
    \label{eq:eLI-basis-decomp}
    \eL^I = \sum_e \sum_{\tb \in \mc{B}_e} (\eL^I \n_e \cdot \tb ) W_{e,\tb} + \frac{ \tr \eL^I}{d} \Ir_d .
  }
  Since $W_{e,\tb}$ is an element of $\mc{B}^*$ independent of mesh size, 
  this identity reduces the estimate of $\| \eL^I \|_{0,K}$ to the estimates of 
  $\| \eL^I \n_e \cdot \tb \|_{0,K}$ with $\tb \in \mc{B}_e$ and $\| \tr \eL^I \|_{0,K}$. 

  We first estimate $\| \eL^I \n \cdot \tb \|_{0,K}$ with $\n = \n_e$ for some $e$. Let $e_K$ be a fixed face of $K$ 
  and define $\Pi_1 \Lb, \Pi_2 \Lb \in \polyt{1}(K)$ as 
  \begin{subequations}
    \eqnlab{Pi1}
    \algn{
      \eqnlab{Pi1-volume}
      \LRp{\Pi_1 \Lb, \Gb}_\K &= \LRp{\Lb , \Gb}_\K , 
      & & \forall \Gb \in \polyt{k-1}(K) , \\
      \eqnlab{Pi1-edges}
      \LRa{\Pi_1 \Lb \n, \mub}_e  &= \LRa{\Lb \n ,\mub}_e,  
      & & \forall \mub \in \polyb{k}(e), e \not = e_K .
    }
  \end{subequations}
  and 
  \begin{subequations}
    \eqnlab{Pi2}
    \algn{
      \eqnlab{Pi2-volume}
      \LRp{\Pi_2 \Lb, \Gb}_\K &= \LRp{\Lb , \Gb}_\K - \LRp{\eu^I \otimes \delw, \Gb}_\K , 
      & & \forall \Gb \in \polyt{k-1}(K) , \\
      \eqnlab{Pi2-edges}
      \LRa{\Pi_2 \Lb \n, \mub}_e  &= \LRa{\Lb \n ,\mub}_e,  
      & & \forall \mub \in \polyb{k}(e), e \not = e_K .
    }
  \end{subequations}
  The existence and uniqueness of $\Pi_1 \Lb$ and $\Pi_2 \Lb$ follow from Lemma~\ref{thm:Pi-PF-estm}.
  By the triangle inequality, 
  \mltln{ 
    \label{eq:eLInt-split}
    \| \eL^I \n \cdot \tb \|_{0,K} \\
    \leq \| \Lb \n \cdot \tb - \Pi_1 \Lb \n \cdot \tb \|_{0,K} 
    + \| (\Pi_1 - \Pi_2) \Lb \n \cdot \tb \|_{0,K} 
    + \| (\Pi_2 - \Pi) \Lb \n \cdot \tb \|_{0,K} .
  }
  Again by the triangle inequality, we bound the first term in \eqref{eq:eLInt-split} as
  \algn{ \label{eq:eLn-estm1}
  \nor{ \Lb \n \cdot \tb - \Pi_1 \Lb \n \cdot \tb}_{0,K}
  &\leq \nor{ \Lb - \Pi_1 \Lb }_{0,K} \\
  &\leq \nor{ \Lb - \Pi^{RTN} \Lb }_{0,K} + \nor{ \Pi^{RTN} \Lb - \Pi_1 \Lb }_{0,K} \notag
}
where $\Pi^{RTN}$ is the row-wise canonical Raviart-Thomas-N\'ed\'elec (RTN) interpolation operator 
into the row-wise $(k+1)$-th order RTN element, which contains $\polyt{k}(\K)$ 
for all $\K \in \Omegah$. 
From \cite[Proposition~2.1 (vi)]{cockburn2008superconvergent},
we have that $\nor{ \Pi^{RTN} \Lb - \Pi_1 \Lb }_{0,K} \lesssim h_K^{k+1}\nor{\Pr_k \Div \Lb}_k$,
and from a well known property of the canonical RTN interpolation operator we have that 
$\nor{ \Lb - \Pi^{RTN} \Lb }_{0,K} \lesssim h_K^{k+1}\nor{\Lb}_{k+1}$.
Therefore, for the first term of \eqref{eq:eLInt-split} we have
\algns{ 
  \nor{\Lb - \Pi_1 \Lb}_{0,K} 
  \lesssim h_K^{k+1} \nor{\Lb}_{k+1,K}. 
}
Note that $\Pi^{RTN} \Lb$ is not necessarily in $\polyt{k}(\Omegah)$ but the above argument does not require $\Pi^{RTN} \Lb \in \polyt{k}(\Omegah)$. 

For the estimate of the second term in \eqref{eq:eLInt-split}, note that the definitions of $\Pi_1$ and $\Pi_2$ give
\begin{subequations}
  \algn{
    \LRp{\Pi_1 \Lb - \Pi_2 \Lb, \Gb}_\K &= \LRp{\eu^I \otimes \delw, \Gb}_\K , 
    & & \forall \Gb \in \polyt{k-1}(K) , \\
    \LRa{\Pi_1 \Lb \n - \Pi_2 \Lb \n, \mub}_e  &= 0,  
    & & \forall \mub \in \polyb{k}(e), e \not = e_K .
  }
\end{subequations}
By Lemma~\ref{thm:Pi-PF-estm}, we can estimate the second term in \eqref{eq:eLInt-split} as 
\algn{  \label{eq:eLn-estm2}
\| (\Pi_1 - \Pi_2) \Lb \n \cdot \tb \|_{0,K} 
&\leq \| \Pi_1 \Lb - \Pi_2 \Lb \|_{0,K} 
\lesssim \| \eu^I \otimes \delw \|_{0,K} \\
&\lesssim h_K \| \eu^I \|_{0,K} \| \wb \|_{W^{1,\infty}(K)} . \notag
  }
  For the estimate of the third term in \eqref{eq:eLInt-split}, recalling \eqnref{Lu-projection1}, \eqnref{Lu-projection2}, and \eqnref{Pi2-volume},
  we derive $\LRp{\Pi_2 \Lb - \Pi \Lb, \Gb}_\K = 0$ for all $\Gb \in \polyt{k-1}(K)$.
  Selecting $\Gb = \LRp{\tb \otimes \n} q$ with $q \in \poly{k-1}(K)$, we have that 
  $\LRp{\Pi_2 - \Pi} \Lb \n \cdot \tb \in \poly{k}^{\perp}(K)$, and by Lemma~\lemref{trace-equiv},
  \algns{
    \nor{\LRp{\Pi_2 - \Pi} \Lb \n \cdot \tb}_{0,\K} 
    \lesssim h_K^{\half} \nor{\LRp{\Pi_2 - \Pi} \Lb \n \cdot \tb}_{0,e}
  }
  for any $e$ of $\pK$.
  From \eqnref{Lu-projection3} and \eqnref{Pi2-edges}, we have, for $e \ne e_K$
  \algns{ 
    \LRa{\LRp{\Pi_2 - \Pi)}\Lb \n, \mub}_e 
    = \LRa{(m+\alpha_1)\eeui + \epi \n - \gb, \mub}_e
    \quad \forall \mub \in \polyb{k}(e),
  }
  with $\gb := -\half\kappa\db \times \LRp{\n \times \LRp{\ebt^I+\ebht^I}}$.
  Choosing $\mub = \LRs{\LRp{\Pi_2 - \Pi}\Lb \n \cdot \tb} \tb$ and applying the Cauchy-Schwarz inequality to the above expression, we have
  \algns{
    \nor{\LRp{\Pi_2 - \Pi}\Lb \n \cdot \tb}_{0,e} 
    &\lesssim \nor{(m+\alpha_1)\eeui - \gb}_{0,e} \\
    &\lesssim \LRp{\alpha_1 + \nor{\wb}_{L^\infty(K)}} \nor{\eeui}_{0,e}
    + \kappa \nor{\db}_{L^\infty(\K)} \nor{\eebi}_{0,e} \\
    &\lesssim \LRp{\alpha_1 + \nor{\wb}_{L^\infty(K)}} \LRp{h_K^{-\half}\nor{\eeui}_{0,\K}+h_K^{\half}\nor{\Grad\eeui}_{0,\K}} \\
    &\quad + \kappa \nor{\db}_{L^\infty(\K)} \LRp{h_K^{-\half}\nor{\eebi}_{0,\K}+h_K^{\half}\nor{\Grad\eebi}_{0,\K}} \\
    &\lesssim \LRp{\alpha_1 + \nor{\wb}_{L^\infty(K)}} \LRp{h_K^{-\half}\nor{\eeui}_{0,\K}+h_K^{k+\half}\nor{\ub}_{k+1,\K}} \\
    &\quad + \kappa \nor{\db}_{L^\infty(\K)} \LRp{h_K^{-\half}\nor{\eebi}_{0,\K}+h_K^{k+\half}\nor{\bb}_{k+1,\K}},
  }
  where we have used the fact that $\nor{\ebht^I}_{0,e} \leq \nor{\ebt^I}_{0,e}$,
  the continuous trace inequality \eqref{continuous_trace_ineq},
  the bound on $\nor{\Grad\eebi}_{0,K}$ given by \eqref{bound_of_grad_ebi},
  and a similar bound for $\nor{\Grad\eeui}_{0,K}$ given by
  \algn{
    \label{bound_of_grad_eui}
    \nor{\Grad \eeui}_{0,K} \lesssim h_K^k \nor{\ub}_{k+1,K} + h_K^{-1} \nor{\eeui}_{0,K}.
  }
  Combining the previous expressions, we have
  \algn{
    \notag \nor{\LRp{\Pi_2 - \Pi} \Lb \n \cdot \tb}_{0,\K} 
    \notag &\lesssim \LRp{\alpha_1 + \nor{\wb}_{L^\infty(K)}} \LRp{\nor{\eeui}_{0,\K} + h_K^{k+1}\nor{\ub}_{k+1,K}} \\
    \notag &\quad + \kappa \nor{\db}_{L^\infty(\K)} \LRp{\nor{\eebi}_{0,\K} + h_K^{k+1}\nor{\bb}_{k+1,K}} \\
    \label{eq:eLn-estm3} &\lesssim \LRp{\alpha_1 + \nor{\wb}_{L^\infty(K)}} \LRp{\nor{\eeui}_{0,\K} + h_K^{k+1}\nor{\ub}_{k+1,K}} \\
    \notag &\quad + \kappa \nor{\db}_{L^\infty(\K)} \LRp{h_K^{k+1}\nor{\bb}_{k+1,\K} + \alpha_3 h_K^{k+1} \nor{\r}_{k+1,K}}, 
  }
  where we have used \eqref{eq:ebI-estm} in the final step, but for simplicity in writing have not
  expanded $\nor{\eeui}_{0,K}$.
  Thus, from the three estimates \eqref{eq:eLn-estm1}, \eqref{eq:eLn-estm2}, \eqref{eq:eLn-estm3} with \eqref{eq:eLInt-split}, we have 
  \algn{  
    \label{eq:eLInt-estm}
    \nor{\eL^I \n \cdot \tb}_{0,K} 
    &\lesssim 
    h_K^{k+1} \nor{\Lb}_{k+1,K} 
    + \LRp{\alpha_1 + \nor{\wb}_{L^\infty(K)}} h_K^{k+1}\nor{\ub}_{k+1,K} \\
    &\quad + \LRp{\alpha_1 + \nor{\wb}_{L^\infty(\K)} + h_K \nor{\wb}_{W^{1,\infty}(K)}} \nor{\eeui}_{0,K} \notag \\
    &\quad + \kappa \nor{\db}_{L^\infty(\K)} \LRp{h_K^{k+1}\nor{\bb}_{k+1,\K} + \alpha_3 h_K^{k+1} \nor{\r}_{k+1,K}}. \notag 
  }

  To complete the estimate of $\| \eL^I \|_{0,K}$, we need to estimate $\| \tr \eL^I \|_{0,K}$. 
  First, by taking $\Gb = q\Ir$ in \eqnref{Lu-projection1} with $q \in \poly{k-1}(K)$, we get
  \algn{ 
    \label{eq:treL-id-1}
    \LRp{ \tr \eL^I, q}_K = \LRp{ \tr(\eu^I \otimes \delw), q }_K =  \LRp{ \Pr_{k-1} \tr(\eu^I \otimes \delw), q }_K , \qquad \q \in \poly{k-1}(\K) 
  }
  where $\Pr_{k-1}$ is the orthogonal $L^2$ projection into $\poly{k-1}(\K)$. 
  For a fixed $e_K \subset \pK$, taking $\mub = w \n_{e_K}$ with $w \in \poly{k}^{\perp}(K)$ in \eqnref{Lu-projection3} and using \eqref{eq:eLI-basis-decomp}, we also get
  \algn{ 
    \label{eq:treL-id-2}
    \LRa{ \tr \eL^I, w }_{e_K} = \LRa{ \zeta, w }_{e_K} = \LRa{ \Pr_{e_K} \zeta, w }_{e_K}, 
  }
  where $\zeta$ is a scalar function on $\K$ defined by 
  \algns{ 
    \zeta := - \LRp{ \LRp{ \sum_e \sum_{\tb \in \mc{B}_e} ( \eL^I \n_e \cdot \tb ) W_{e,\tb} } \n_{e_K} } \cdot \n_{e_K} 
    + \veps_p^I|_{e_K} + (\alpha_1+m)\eu^I \cdot \n_{e_K} - \gb \cdot \n_{e_K} 
  }
  and $\Pr_{e_K}$ is the orthogonal $L^2$ projection into $\poly{k}(e_K)$. 
  We define $\Pi_{e_K} (f,g) \in \poly{k}(K)$ for $f \in L^2(K)$ and $g \in L^2(e_K)$ as 
  \algns{ 
    \LRp{ \Pi_{e_K} (f, g), q }_K &= \LRp{ f, q }_K, & & \forall q \in \poly{k-1}(K), \\
    \LRa{ \Pi_{e_K} (f, g), \mu}_e &= \LRa{ g, \mu }_e, & & \forall \mu \in \poly{k}(e), e = e_K.
  }
  We refer to \cite[Lemma~3.1]{cockburn2008superconvergent} for well-posedness of this interpolation and optimal approximation property. 
  By an argument similar to Lemma~\ref{thm:Pi-PF-estm}, we have 
  \algns{
    \nor{\Pi_{e_K} (f, g)}_{0,\K} \lesssim \nor{f}_{0,\K} + h_{\K}^{\half} \nor{g}_{0, e_K}. 
  }
  For simplicity, we will use $\Pi_{e_K} f$ if $g = f|_{e_K}$. 

  Note that $\tr \eL^I$ is an element-wise polynomial because $\tr \Lb = 0$, so $\tr \eL^I = \Pi_{e_K} \tr \eL^I$.
  From this, 
  the identities \eqref{eq:treL-id-1} and \eqref{eq:treL-id-2}, and the above inequalities from scaling argument,
  we have 
  \algns{ 
    \nor{\tr \eL^I}_{0,K} 
    &= \nor{\Pi_{e_K} \tr \eL^I}_{0,K} \lesssim \nor{\Pr_{k-1} \tr ( \eu^I \otimes \delw )}_{0,K} + h_K^{\half} \nor{\Pr_{e_K} \zeta}_{0,e_K} .
  }
  The optimality of the orthogonal $L^2$ projection and the inverse trace inequality give
  \algns{
    h_K^{\half} \nor{\Pr_{e_K} \zeta}_{0,e_K} \le h_K^{\half} \nor{\Pr_k \zeta}_{0,e_K} \lesssim \nor{\Pr_k \zeta}_{0,\K} \le \nor{\zeta}_{0,\K}, 
  }
  and using this and the previous estimate, we can write
  \algns{
    \nor{\tr \eL^I}_{0,K}  
    &\lesssim \nor{\eu^I \otimes \delw}_{0,K} + \nor{\zeta}_{0,K} \\
    &\lesssim h_K \nor{\wb}_{W^{1,\infty}(K)} \nor{\eeui}_{0,K} 
    + h_K^{k+1}\nor{p}_{k+1,K} \\
    & \quad + \sum_e \sum_{\tb \in \mc{B}_e}\nor{(\eL^I \n_e) \cdot \tb}_{0,K} 
    + \nor{(\alpha_1+m)\eu^I - \gb }_{0,K} \\ 
    &\lesssim 
    h_K^{k+1}\nor{p}_{k+1,K}
    + h_K^{k+1} \nor{\Lb}_{k+1,K} 
    + \LRp{\alpha_1 + \nor{\wb}_{L^\infty(K)}} h_K^{k+1}\nor{\ub}_{k+1,K} \\
    & \quad + \kappa \nor{\db}_{L^\infty(\K)} \LRp{h_K^{k+1}\nor{\bb}_{k+1,\K} + \alpha_3 h_K^{k+1} \nor{\r}_{k+1,K}} \\
    & \quad + \LRp{\alpha_1 + \nor{\wb}_{L^\infty(\K)} + h_K \nor{\wb}_{W^{1,\infty}(K)}} \nor{\eeui}_{0,K} 
  }
  and here we used previous results on $\nor{\eL^I \n}_{0,\K}$ and $\nor{(\alpha_1+m)\eu^I - \gb}_{0,\K}$ in the final inequality. 
  We completed the estimate of $\| \eL^I \|_{0,K}$. 
\end{proof}

We now define the adjoint projection $\Pi^* (\Lb^*, \ub^*, \p^*, \Hb^*,
\bb^*, \r^*, \ubh^*, (\bbh^*)^t, \rh^*)$. As in the splitting of
errors with $\Pi$, we define \algns{ \veps_{\sigma^*}^I = \sigma^* -
\Pi^* \sigma^* } for an adjoint unknown $\sigma^*$. We first define $\Pi^*
\Hb^*$, $\Pi^* \p^*$, $\Pi^* \ubh^*$, $\Pi^* (\bbh^*)^t$, $\Pi^*
\rh^*$ as $L^2$ projections into relevant polynomials spaces, and
define $\Pi^* \bb^*$, $\Pi^* \r^*$ to satisfy
\begin{subequations}
  \eqnlab{abr-proj}
  \algn{
    \label{eq:abr-proj1} (\eebis, \cb)_\K &= 0, & & \forall \cb \in \polyb{k-1}(K) , \\
    \label{eq:abr-proj2} (\eeris, s)_\K &= 0, & & \forall s \in \poly{k-1}(K), \\
    \label{eq:abr-proj3} \LRa{ - \eebis \cdot \n + \alpha_3 \eeris , \gamma }_e &= 0, & & \forall \gamma \in \poly{k}(e) .
  }
\end{subequations}
We then choose $\Pi^* \Lb^*$, $\Pi^* \ub^*$ to satisfy
\begin{subequations}
  \eqnlab{adjoint-projection}
  \algn{
    \label{eq:aLu-proj1} (\elis, \Gb)_\K + (\eeuis \otimes \wb, \Gb)_\K &= 0, & & \forall \Gb \in \polyt{k-1}(K) , \\
    \label{eq:aLu-proj2} (\eeuis, \vb)_\K &= 0, & & \forall \vb \in \polyb{k-1}(K), \\
    \label{eq:aLu-proj3} \LRa{-\elis \n + \alpha_1 \eeuis  , \mub }_e &= \LRa{ \gb, \mub}_e , & & \forall \mub \in \polyb{k}(e)
  }
\end{subequations}
where
\algns{
  \gb = \epis \n - \half \kappa \db \times \LRp{\n \times \LRp{ - (\eebis )^t + \ebthis }} .
}

Assuming that $(\Lb^*, \ub^*, \p^*, \Hb^*,  \bb^*, \r^*, \ubh^*, (\bbh^*)^t, \rh^*)$ are sufficiently regular, 
we can show that the interpolation $\bs{\Pi}^*$ is well-defined and provides optimal approximations.
Due to the similarity between $(\Pi\bb,\Pi\r)$ and $(\Pi^*\bb^*,\Pi^*\r^*)$,
we can conclude
\begin{subequations}
  \eqnlab{ebIs-erIs}
  \algn{
    \label{eq:ebIs-estm} 
    \nor{\eebis}_{0,\K} &\lesssim h^{k +1} \nor{\bb^* }_{k+1, \K} + \alpha_3 h^{k +1} \nor{\r^* }_{k +1, \K} , \\
    \label{eq:erIs-estm} 
    \nor{\eeris}_{0,\K} &\lesssim \alpha_3^{-1} h^{k +1} \nor{ \Div \bb^* }_{k, \K} + h^{k +1} \nor{\r^*}_{k +1, \K}.
  }
\end{subequations}
It can also be shown that
\algns{ 
  \| \eeuis \|_{0} &\lesssim (\alpha_1 - \half \nor{\wb}_\Linfty)^{-1} 
  \left[ ({ \alpha_1 + h \nor{\wb}_\Woneinfty }) h^{k+1} \| \ub^* \|_{k+1} \right. \\
    &\quad \left. + h^{k+1} \| \Div \Lb^* + \Grad \p^* \|_{k} 
    + \kappa h^{k+1} \| \db \|_\Linfty \LRp{\| \bb^* \|_{k+1} 
  + \alpha_3 \| \r^* \|_{k +1} } \right] , \\
  \nor{\elis}_0 
  &\lesssim 
  h^{k+1}\nor{p^*}_{k+1} 
  + h^{k+1}\nor{\Lb^*}_{k+1}
  + \kappa h^{k+1}\nor{\db}_\Linfty\LRp{\nor{\bb^*}_{k+1} + \alpha_3 \nor{r^*}_{k +1}} \\
  &\quad + \LRp{\alpha_1 + h\nor{\wb}_\Woneinfty} \nor{\eeuis}_{0}
  + \alpha_1 h^{k+1} \nor{\ub^*}_{k+1},
}
assuming $\tr \Lb^* = 0$. 
The proofs are analogous to those for the $\bs{\Pi}$ projections, with the only differences resulting
from the absence of $m$ from \eqref{eq:aLu-proj3}.
As a consequence, from the elliptic regularity assumption \eqnref{adjoint-regularity}, we have
\algn{  \label{eq:elliptic-estimate}
\max\left\{
  \nor{\elis}_{0},\; 
  \nor{\eeuis}_{0},\; 
  \nor{\epis}_{0},\; 
  \nor{\eebis}_{0},\; 
  \nor{\eeris}_{0}
  \vphantom{\nor{\epis}_{0}} 
\right\}
&\lesssim h \nor{ \bs{\sigma}, \bs{\theta}}_{0} 
}
and the implicit constant depends on $\db$, $\wb$, $\alpha_1$, $\alpha_2$, $\alpha_3$ but not $h$.

\section{Well-posedness and regularity of the adjoint
problem} \seclab{regularity_of_adjoint} In this section, we discuss
the well-posedness of the adjoint equation \eqnref{adjoint-mhd} and
conditions under which  the regularity result \eqnref{adjoint-regularity} holds.
For the well-posedness
of \eqnref{adjoint-mhd} we can adapt the approach
in \cite{HoustonSchoetzauWei09}. In particular, since
\eqref{eq:amhd2lin1n} and \eqref{eq:amhd5lin1n}
satisfy the following  antisymmetry:
\algns{
  ((\wb \cdot \nabla) \ub^*, \ub^*) = 0, \qquad - (\kappa \db \times
  (\nabla \times \bb^*), \ub^*) + ( \kappa \nabla \times
  (\ub^* \times \db), \bb^*) = 0 , 
}
we can invoke a similar argument as
in \cite{HoustonSchoetzauWei09}
to conclude
\algn{ 
  \label{eq:adjoint-regularity0} 
  \nor{\ub^*}_{1}
  + \nor{\p^*}_{0} + \nor{\r^*}_{1} + \nor{\bb^*}_{0}
  + \nor{\curl \bb^*}_{0} \lesssim \nor{\thetab, \sigmab}_{0}.  
}
because $\Hb^* = \nabla \times \bb^*$.  We next assume that solutions
of the Stokes and time-harmonic Maxwell equations (with $0$ frequency)
satisfy higher order regularities, i.e., when $\wb = \db = 0$
in \eqnref{adjoint-mhd}, the solution $\ub^*$, $\p^*$, $\bb^*$ satisfy
$\nor{\ub^*}_{2} + \nor{\p^*}_{1} \lesssim \nor{\bs{\sigma}}_0$ and
$\nor{\bb^*}_{2} \lesssim \nor{\bs{\theta}}_0$.  Sufficient conditions
for these assumptions are known but the details are beyond the scope
of this paper, so we refer the readers
to \cite{Dauge-89,HiptmairMoiolaPerugia-11,Monk-FEM-Maxwell}.
We now assume that $(\ub^*, \Lb^*, \p^*, \bb^*, \Hb^*, \r^*)$ is the solution
of \eqnref{adjoint-mhd}.  From \eqnref{adjoint-mhd} and the
regularity assumptions of the Stokes and Maxwell equations, we have
\algns{
  \nor{\ub^*}_{2} + \nor{\p^*}_{1} &\lesssim \nor{(\wb \cdot \nabla) \ub^* + \kappa \db \times (\nabla \times \bb^*) + \bs{\theta} }_0 \\
  &\lesssim \nor{\wb}_{L^\infty} \nor{\ub^*}_1 + \nor{\db}_{L^\infty} \nor{\curl \bb^*}_0 + \nor{\bs{\theta}}_0 \\
  &\lesssim (\nor{\wb, \db}_{L^\infty} +1) \nor{\bs{\theta}, \bs{\sigma}}_0, \\
  \nor{\bb^*}_2 &\lesssim \nor{-\kappa \nabla \times (\ub^* \times \db) + \bs{\sigma}} \\
  &\lesssim \nor{\ub^*}_1 \nor{\db}_{L^\infty} + \nor{\ub^*}_0 \nor{\db}_{W^{1,\infty}} + \nor{\bs{\sigma}}_0 \\
  &\lesssim (\nor{\db}_{W^{1,\infty}} + 1) \nor{\bs{\theta}, \bs{\sigma}}_0 .
}
The regularity of $\nor{\r^*}_1$ is  already given in \eqref{eq:adjoint-regularity0} and the regularity of $\nor{\Lb^*}_0$ and $\nor{\Hb^*}_0$ is obvious from the identities $\Lb^* = \nabla \ub^*$ and $\Hb^* = \nabla \times \bb^*$.

\section{Divergence-free post-processing of \texorpdfstring{$\ubH$}{u} and \texorpdfstring{$\bbH$}{b}}
\seclab{appendix:postproc}

In this appendix we adapt a postprocessing procedure
in \cite{CockburnGopalakrishnanNguyenEtAl11} to enforce the pointwise
solenoidal constraint on $\ub$ and $\bb$. For simplicity the
exposition is done only for $d = 3$. For completeness, error estimates of
the post-processed solutions will also be derived.  To begin, let
$\lambda_0$, $\lambda_1$, $\lambda_2$, $\lambda_3$ be the barycentric
coordinates of a tetrahedron $\K$, and $B_{\K}$ be a symmetric
matrix-valued bubble function defined by
\algns{
  B_{\K} = \sum_{i=0}^3 \lambda_i \lambda_{i+1} \lambda_{i+2} \LRp{\Grad \lambda_{i+3} \otimes \Grad \lambda_{i+3} },
}
with the index $i$ counted modulo 4. Let $\mc{N}_k$ be the space of $\R^3$-valued polynomials
\algns{
  \mc{N}_k(\K) = \polyb{k-1}(\K) \oplus \mc{N}_k'(\K),
}
where $\mc{N}_k'(\K)$ is the space of homogeneous $\R^3$-valued polynomials $\vb$ of degree $k$ such that $\vb \cdot \bs{x} = 0$. We define $\mc{S}_k(\K)$ by
\algns{
  \mc{S}_k (\K) = \LRc{ \vb \in \mc{N}_k(\K) \,:\, \LRp{\vb, \Grad \phi}_{\K} = 0 \text{ for all } \phi \in \poly{k}(\K)} .
}
We recall an alternative characterization of the interpolation of the Brezzi--Douglas--Marini (BDM) element \cite[Proposition~A.1]{CockburnGopalakrishnanNguyenEtAl11}:
\begin{subequations}
  \eqnlab{BDM-projection}
  \algn{
    \label{BDM-projection1} \LRa{ ( \Pi^{BDM} \ub - \ub) \cdot \n, \mu }_e &= 0, & & \mu \in \poly{k}(e) , \\
    \label{BDM-projection2} \LRp{ \Pi^{BDM} \ub - \ub, \Grad w }_{\K} &= 0, & & w \in \poly{k-1}(\K), \\
    \label{BDM-projection3} \LRp{ \Curl \LRp{ \Pi^{BDM} \ub - \ub}, (\Curl \vb) B_{\K} }_{\K} &= 0 , & & \vb \in \mc{S}_{k-1}(\K) .
  }
\end{subequations}

Our post-processed solution $\bar{\ub}_h|_{\K} \in \mc{P}_k(\K)$ is defined as
\begin{subequations}
  \eqnlab{u-pp} 
  \algn{
    \label{u-pp1} \LRa{ (\bar{\ub}_h - \ubhH) \cdot \n, \mu }_e &= 0, & & \mu \in \poly{k}(e) , \\
    \label{u-pp2} \LRp{ \bar{\ub}_h - \ubH, \Grad w }_{\K} &= 0, & & w \in \poly{k-1}(\K), \\
    \label{u-pp3} \LRp{ \Curl \bar{\ub}_h - \LbH^{skw}, (\Curl \vb) B_{\K} }_{\K} &= 0 , & & \vb \in \mc{S}_{k-1}(\K) ,
  }
\end{subequations}
where $\LbH^{skw} = \LRp{ \Lb_{h,32} - \Lb_{h,23}, \Lb_{h,13}-\Lb_{h,31}, \Lb_{h,21}-\Lb_{h,12}}$.
In two dimensions, i.e. $d = 2$, $\n \times \bs{a}$, $\Curl \ub$, and $B_{\K}$ are replaced by $n_1 a_2 - n_2 a_1$, $\pd_1 u_2 - \pd_2 \u_1$, and the standard bubble function on $\K$, respectively.

The post-processed solution $\bar{\ub}_h$ is in the BDM space, so its divergence is well-defined. Further, it is divergence-free because, for any $q \in \poly{k-1}(\Omegah)$,
\algns{
  \LRp{ \Div \bar{\ub}_h, q}  &= \LRa{ \bar{\ub}_h \cdot \n, q } - \LRp{ \bar{\ub}_h, \Grad q} =  \LRa{ \ubhH \cdot \n, q } - \LRp{ {\ubH}, \Grad q} = 0
}
where the last equality is due to \eqref{eq:local_3_1}.
For error analysis, it is enough to estimate $\delu := \bar{\ub}_h - \Pi^{BDM} \ub$.
Using  $\Curl \ub = \Lb^{skw}$ together with \eqnref{BDM-projection} and \eqnref{u-pp} gives
\begin{subequations}
  \eqnlab{delta-u}
  \algn{
    \label{delta-u-1} \LRa{ \delu \cdot \n, \mu }_{e} &= \LRa{\euh \cdot \n }_{e}, & & \mu \in \poly{k}(e) , \\
    \label{delta-u-2} \LRp{ \delu , \Grad w }_{\K} &= \LRp{\ub - \ubH, \Grad w}_{\K}, & & w \in \poly{k-1}(\K), \\
    \label{delta-u-3} \LRp{ \Curl \delu , (\Curl \vb) B_{\K} }_{\K} &= \LRp{ (\Lb - \LbH)^{skw} , (\Curl \vb) B_{\K} }_{\K} , & & \vb \in \mc{S}_{k-1}(\K) .
  }
\end{subequations}
Using a  scaling argument yields
\algns{
  \nor{\delu}_{0,\K} \lesssim \LRp{ h_{\K}^{1/2} \nor{\euh}_{0,e}
  + \nor{\ub - \ubH}_{0,\K} + h_\K \nor{\Lb - \LbH}_{0,\K}} .
}
On the other hand, by the triangle inequality we have
\algns{
  \nor{\ub - \bar{\ub}_h}_0 \lesssim \nor{\ub - \Pi^{BDM} \ub}_{0}
  + { h^{1/2} \nor{\euh}_{\pd \Omegah } + \nor{\ub - \ubH}_{0} + h \nor{\Lb - \LbH}_{0}}.
}
Thus, the convergence rates of $\nor{\bar{\ub}_h - \ub}_0$ and $\nor{\ubH - \ub}_0$ are the same.

To post-process $\bbH$, we define $\bar{\bb}_h|_{\K} \in \mc{P}_k(\K)$ as
\begin{subequations}
  \eqnlab{b-pp}
  \algn{
    \label{b-pp1} \LRa{ \LRp{\bar{\bb}_h - \bbH - \alpha_2 \LRp{ \rH - \rhH} } \cdot \n, \mu }_e
    &= 0, & & \mu \in \poly{k}(e) , \\
    \label{b-pp2} \LRp{ \bar{\bb}_h - \bbH, \Grad w }_{\K} &= 0, & & w \in \poly{k-1}(\K), \\
    \label{b-pp3} \LRp{ \Curl \bar{\bb}_h - \HbH, (\Curl \vb) B_{\K} }_{\K} &= 0 , & & \vb \in \mc{S}_{k-1}(\K) ,
  }
\end{subequations}
Then the divergence of $\bar{\bb}_h$ is well-defined and $\Div \bar{\bb}_h = 0$ by \eqref{eq:local_6_1}.
By a completely analogous argument as above and the fact $\Hb = \Curl \bb$, we can conclude
\algns{
  \nor{\bb - \bar{\bb}_h}_0 \lesssim \nor{\bb - \Pi^{BDM} \bb}_0
  + h \nor{\eb, \er}_1 + h^{1/2} \nor{\erh}_{\pd \Omegah} + h \nor{\Hb - \HbH}_0,
}
hence the convergence rates of $\nor{\bb - \bar{\bb}_h}_0$ and $\nor{\bb - \bbH}_0$ are the same.

\bibliographystyle{siamplain}
\bibliography{references,MHD_Shadid,luis,bochev,FEM}

\end{document}